%% file: main.tex
\documentclass{scrartcl}
\usepackage{ifdraft}
\usepackage{enumitem}
\usepackage[usenames,dvipsnames,table]{xcolor}
\usepackage{float}
\usepackage{algorithmic}

\include{includes}

\include{commands}
\usepackage{pifont}
\usepackage{upgreek}
\usepackage{caption}
\usepackage{mathrsfs}
\usepackage{graphicx}
\usepackage[style = authoryear, backend=biber]{biblatex}
\usepackage{tikz}
\usepackage{relsize}
\usepackage{stackengine}
\usetikzlibrary{patterns}
\usetikzlibrary{hobby}
\usetikzlibrary{shapes}
\newcommand{\quotationmarks}[1]{``#1''}\usepackage{tikz}
\usetikzlibrary{shapes,
                tikzmark}
\tikzset{every tikzmarknode/.style={%
        draw=blue, thick, inner sep=4pt,shift={(10pt,10pt)}}
        }
\DeclareCiteCommand{\cite}
  {\usebibmacro{prenote}}
  {\usebibmacro{citeindex}%
   \printtext[bibhyperref]{\usebibmacro{cite}}}
  {\multicitedelim}
  {\usebibmacro{postnote}}
\makeatletter
\newsavebox{\@brx} 
\newcommand{\llangle}[1][]{\savebox{\@brx}{\(\m@th{#1\langle}\)}%
  \mathopen{\copy\@brx\kern-0.5\wd\@brx\usebox{\@brx}}}
\newcommand{\rrangle}[1][]{\savebox{\@brx}{\(\m@th{#1\rangle}\)}%
  \mathclose{\copy\@brx\kern-0.5\wd\@brx\usebox{\@brx}}}
\makeatother

\makeatletter
\newcommand{\dummytikzmarknode}[2][]{%
  \@ifempty{#1}
    {...}
    {...}%
}
\makeatother

\ifdraft
{
  \let\tikzmarknode\dummytikzmarknode
}{}

\renewcommand{\todo}[1]{\todon[inline,color=green!40]{\color{NavyBlue}{\normalsize\textbf{#1}}}}

\newcommand\e{\mathsf{e}}

\definecolor{lightgray}{rgb}{0.8, 0.8, 0.8}

\newcommand\SetIntervalPartitions{\mathbf{IP}}
\newcommand\SetStandardizedIntervalPartitions{\mathbf{StIP}}

\newcommand\NumBlocks{\mathbf{NB}}
\newcommand\cliques{\mathbf{cliques}}
\newcommand\std{\mathbf{std}}
\newcommand\st{\mathbf{st}}
\newcommand\gluepartitions{\cdot_{\textbf{\textup{glue}}}}
\newcommand\bigsigma{\mathbf{S}}

\newcommand\stda{\resizebox{!}{1\height}{$\mathfrak{a}$}}
\newcommand\stdb{\resizebox{!}{1\height}{$\mathfrak{b}$}}
\newcommand\stds{\resizebox{!}{1\height}{$\mathfrak{s}$}}
\newcommand\stdt{\resizebox{!}{1\height}{$\mathfrak{t}$}}

\newcommand\stdg{\resizebox{!}{1\height}{$\mathfrak{g}$}}
\newcommand\stdh{\resizebox{!}{1\height}{$\mathfrak{h}$}}
\newcommand\stdL{\resizebox{!}{1\height}{$\mathfrak{L}$}}
\newcommand\stdM{\resizebox{!}{1\height}{$\mathfrak{M}$}}
\newcommand\unstdI{\mathcal{I}}
\newcommand\unstdJ{\mathcal{J}}
\newcommand\unstdW{\mathcal{W}}
\newcommand\unstdZ{\mathcal{Z}}
\newcommand\superinfiltration{ \raisebox{0.5ex}{\rotatebox[origin=c]{90}{$\Rightarrow$}}}
\newcommand\concvargas{\square}

\newcommand\PC{\mathsf{PC}}
\newcommand\IPC{\mathsf{IPC}}
\newcommand{\adjacentSquares}[1]{%
    \foreach \x in {1,...,#1} {%
        \ifnum\x=#1
            \Box
        \else
            \Box\kern-1.1pt%
        \fi
    }%
}
\newcommand\qspart{\mathlarger{{\text{\color{Salmon}{\ding{171}}}}}}
\newcommand\coqspart{\Delta_{\qspart}}
\newcommand\qswrd{\underset{\tiny\textup{wrd}}{\qspart}}
\newcommand\qsgen{{\text{\color{SpringGreen}\ding{170}}}}
\newcommand\coqsgen{\Delta_{\qsgen}}
\newcommand\conctens{\text{\ding{169}}}
\newcommand\conc{\bullet}
\newcommand\deconc{\Delta_{\conc}}
\newcommand\genconc{\resizebox{!}{0.6\height}{$\blacksquare$}}
\newcommand\deconcgen{\Delta_{\resizebox{!}{0.4\height}{$\blacksquare$}}}
\newcommand\GPC{\mathsf{GPC}}

\newcommand\singleblockm{\mathbf{m}}
\newcommand\singleblockn{\mathbf{n}}

\newcommand\singleblockN{\mathbf{N}}
\newcommand\FreeIntPart{\mathcal{H}_{\tiny\textbf{int}}}
\newcommand\FreePer{\mathcal{H}_{\tiny\textbf{per}}}
\newcommand\FreeVincPer{\mathcal{H}_{\tiny\textbf{vinc}}}
\newcommand\punion{\bigcup}

\usepackage{amsmath}

\newcommand{\mybinom}[2]{%
  \Biggl(\mkern-1mu\begin{matrix}#1\\#2\end{matrix}\mkern-1mu\Biggr.\kern-3pt%
  \Biggl.\vphantom{\begin{matrix}#1\\#2\end{matrix}}\mkern-1mu\Biggr)%
}

\title{Hopf Algebra on Vincular Permutation Patterns}

\author{Joscha Diehl$^\dagger$, Emanuele Verri$^\dagger$}
\date{ \small
    $^\dagger$\textit{Institute of Mathematics and Computer Science, University of Greifswald}\\[2ex]%
}
\begin{document}

\maketitle

\begin{abstract}

We introduce a new Hopf algebra that operates on pairs of finite interval partitions and permutations of equal length. This algebra captures \textit{vincular patterns}, which involve specifying both the permutation patterns and the consecutive occurrence of values.
Our motivation stems
from linear functionals that
encode the number of occurrences of these patterns,
and we show that they behave well with respect to the operations of this Hopf algebra.

\end{abstract}

{
  \hypersetup{linkcolor=black}
  \tableofcontents
}

\section{Introduction}
\textit{Permutation patterns} are ubiquitous in discrete mathematics. Much effort is devoted to developing algorithms that count these patterns efficiently, see for example \cite{even2020independence} and \cite{even2021counting}.

They have also been successfully used in \textit{time series analysis} in the popular work from \cite{bandt2002permutation}, where the authors introduced the concept of \textit{permutation entropy}. For a discrete time series, consider only the order of the values. As an example, the time series
\begin{table}[H]
    \centering
\begin{tabular}{c}
\includegraphics[width=2cm]{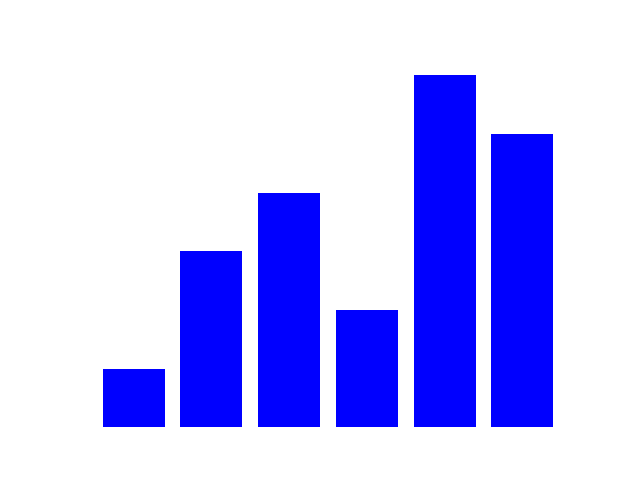}\\
1\,2\,3\,4\,5\,6
\end{tabular}
\end{table}
can be \quotationmarks{reduced} to the permutation $134265 \in \bigsigma_{6}$. Permutation entropy is based on specific permutation patterns, namely consecutive patterns. If we fix an order for the patterns, say 2, we observe
\begin{table}[H]
\centering
\begin{tabular}{cccccc}
\includegraphics[width=2cm]{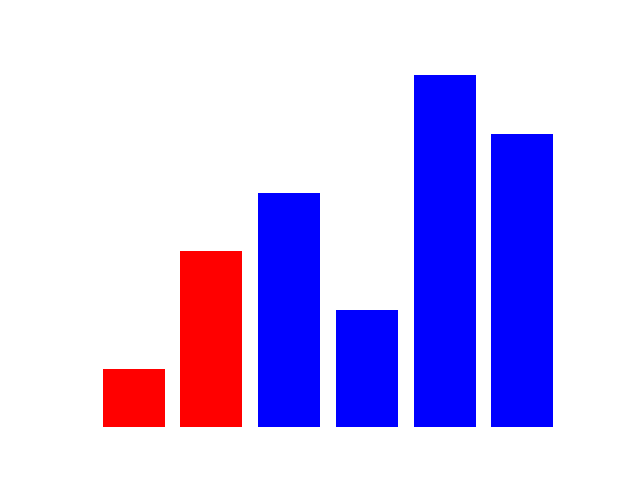} & 
\includegraphics[width=2cm]{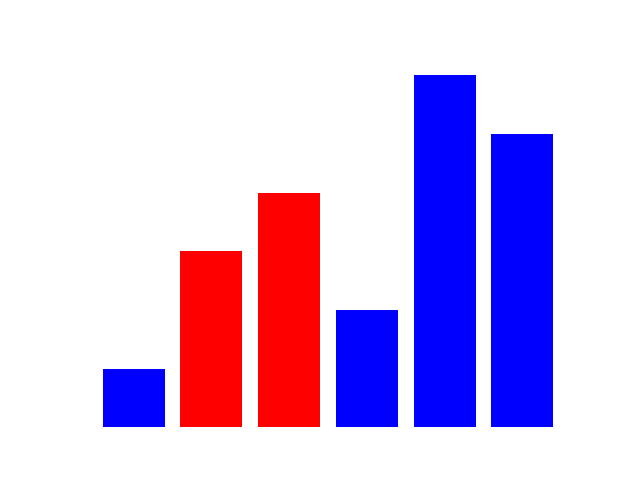} & 
\includegraphics[width=2cm]{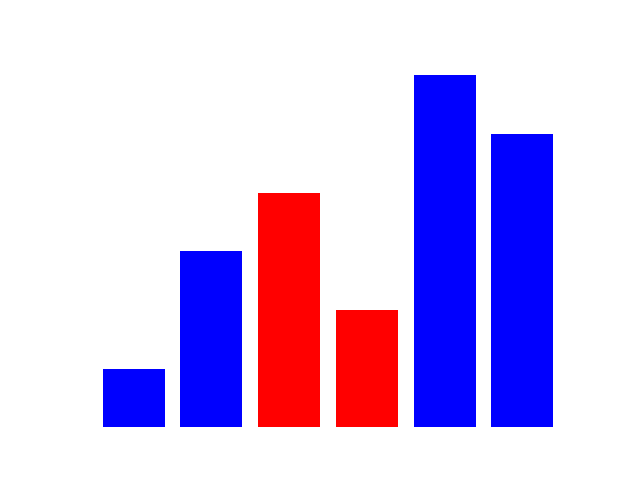} & 
\includegraphics[width=2cm]{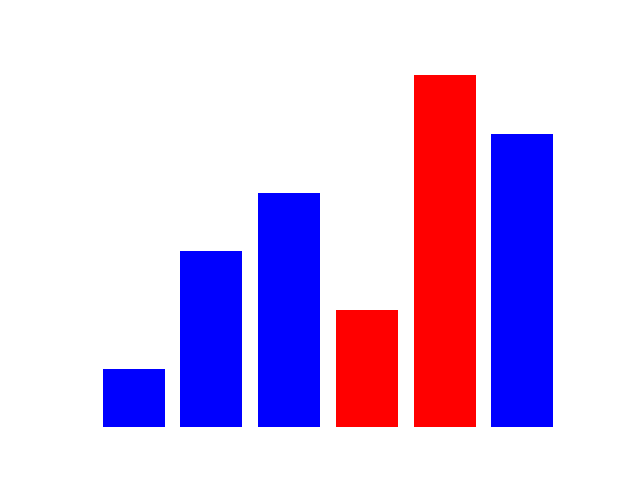} & 
\includegraphics[width=2cm]{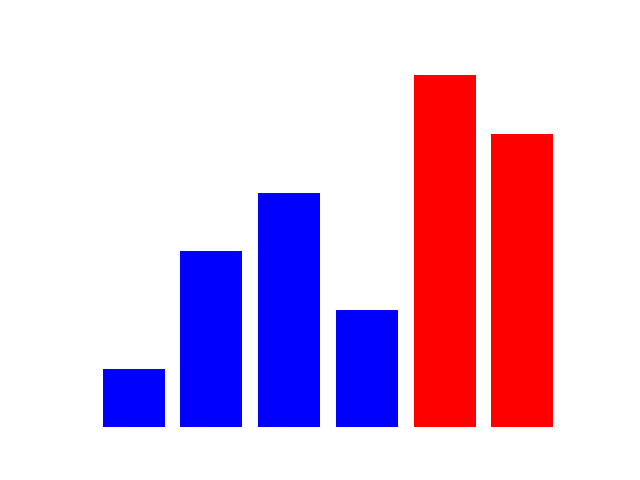}\\
12 &12 & 21 & 12 & 21\\
\end{tabular}
\end{table}
The frequencies of $12,21 \in \bigsigma_{2}$ are needed to compute the permutation entropy of order 2. In general, for  the permutation entropy of order $n$, we would compute the frequencies of the elements of $\bigsigma_{n}$ of consecutive patterns of length $n$.

On the other hand, many notable constructions on permutations have been introduced in \textit{algebraic combinatorics}: the Malvenuto-Reutenauer Hopf algebra on permutations is the most celebrated example, \cite{malvenuto1995duality}. 

 In \cite{vargas2014hopf}, the author introduces a Hopf algebra on permutations. Let $\sigma \in \bigsigma_{m}$ and $\tau \in \bigsigma_{n}$, then
\begin{align*}
&\Delta_{\,\scriptsize\superinfiltration}(\sigma) :=\sum_{\substack{A,B \subset [m]:\\\\ A \cup B = [m]}}\st \left(\sigma\evaluatedAt{A}\right) \otimes \st \left(\sigma\evaluatedAt{B}\right),\\
&\sigma \concvargas \tau:=\sigma_{1} \cdots \sigma_{m} (\tau_{1}+m)\cdots (\tau_{n}+m).
\end{align*}
where, for example
\begin{align*}
&\Delta_{\,\scriptsize\superinfiltration}(21) =\e\otimes21 +\textcolor{red}{2} \cdot 1\otimes1 +\textcolor{red}{2} \cdot  1\otimes21 + 21\otimes \e +\textcolor{red}{2} \cdot 21\otimes1 + 21\otimes21,
\\&21 \concvargas  12=2134.
\end{align*}
and it is there shown that $\FreePer:=(\bigoplus_{n \in \N}\Q[\bigsigma_{n}], \raisebox{0.5ex}{\rotatebox[origin=c]{90}{$\Rightarrow$}}, \Delta_{\scriptsize \concvargas})$ is a \textit{filtered} Hopf algebra. This construction relates to occurrences of permutation patterns. Denote with $\bigsigma:=\bigcup_{n}\bigsigma_{n}$, the set of all permutations, and consider the family of linear functionals $(\textbf{pattern}(\sigma))_{\sigma \in \bigsigma}$, defined on basis elements, $\Lambda \in \bigsigma$ as 
\begin{align*}
\Big\langle \textbf{pattern}(\sigma),\Lambda \Big\rangle := \#\{A \subset [|\Lambda|]|\;\st(\Lambda\evaluatedAt{A})=\sigma\},
\end{align*} i.e. the occurrences of $\sigma$ as a \textit{pattern} on $\Lambda$. We can endow $\Q^{\bigsigma}$, the set of all functions from $\bigsigma$ to $\Q$, with a $\Q$-algebra structure using the pointwise product
\begin{align*}
 \forall f,g \in \Q^{\bigsigma}:\forall \sigma \in \bigsigma:\;\;(f + g)(\sigma)&:=f(\sigma) + g(\sigma),\,
(f\cdot g)(\sigma):=f(\sigma)\cdot g(\sigma).
\end{align*}
Then Vargas showed that
\begin{align*}
 \left(\bigoplus_{n}\Q[\bigsigma_{n}], \raisebox{0.5ex}{\rotatebox[origin=c]{90}{$\Rightarrow$}}\right) &\to    \left(\Q^{\bigsigma},\cdot\right),\\
 \sigma &\mapsto \textbf{pattern}(\sigma)
\end{align*}
is an (injective) algebra homomorpshism, i.e., for all $\Lambda \in \bigsigma$
\begin{align*}
 \Big\langle \textbf{pattern}(\sigma \; \superinfiltration \; \tau),\Lambda \Big\rangle &= \Big\langle \textbf{pattern}(\sigma),\Lambda \Big\rangle   \Big\langle \textbf{pattern}(\tau),\Lambda \Big\rangle.
\end{align*}
Taking on a slightly different approach, we fix $\Lambda \in \bigsigma$, vary the pattern $\sigma$ and define the family of functionals
\begin{align*}
\Big\langle \PC(\Lambda), \sigma \Big\rangle := \#\{A \subset [|\Lambda|]|\;\st(\Lambda\evaluatedAt{A})=\sigma\},
\end{align*}
and call $\PC(\Lambda)$ the \textit{signature} of the permutation $\Lambda$. The term \quotationmarks{signature} is motivated by the following identity 
\begin{align}
\label{eq:signature_permutation}
\forall \sigma,\tau\in \bigsigma:\; \Big\langle \PC(\Lambda),\sigma \Big\rangle \Big\langle \PC(\Lambda),\tau \Big\rangle  &=  \Big\langle \PC(\Lambda),\sigma \;  \raisebox{0.5ex}{\rotatebox[origin=c]{90}{$\Rightarrow$}} \; \tau \Big\rangle 
\end{align}
which is reminiscent of the shuffle identity for the signature of a path
(\cite{hambly2010uniqueness,ree1958lie,chen1957integration}).
For example
\allowdisplaybreaks
\begin{align*}
&\Big \langle \PC (132), 1 \Big \rangle \cdot \Big \langle \PC (132), 1 \Big \rangle
\\&=\Big \langle \PC \Big(\raisebox{-1ex}{ \includegraphics[scale=0.08]{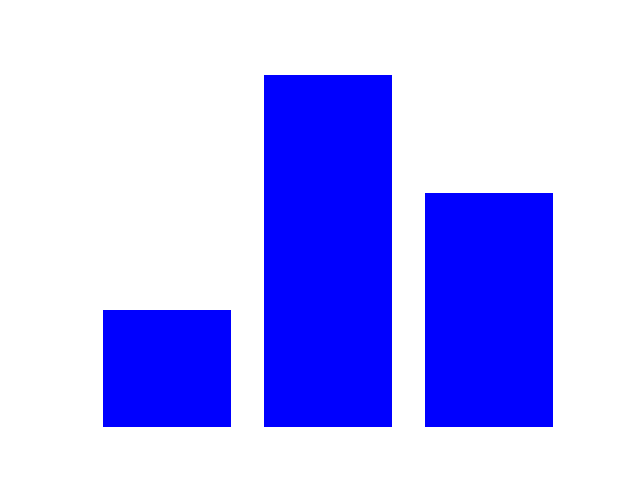}} \Big), \raisebox{-0.5ex}{\includegraphics[width=0.4cm,height=0.4cm]{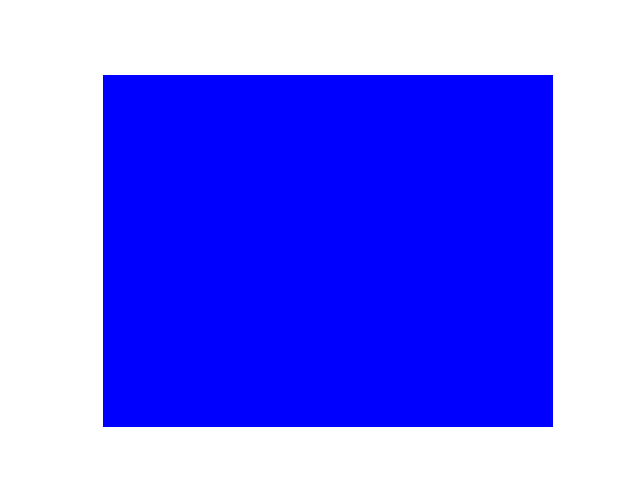}} \Big \rangle \cdot \Big \langle \PC \Big(\raisebox{-1ex}{ \includegraphics[scale=0.08]{pictures/vargas_signature_example/Figure_perm_132.png}} \Big), \raisebox{-0.5ex}{\includegraphics[width=0.4cm,height=0.4cm]{pictures/vargas_signature_example/Figure_perm_1.png}} \Big \rangle\\&=
|\{\{1\},\{2\},\{3\}\} \times \{\{1\},\{2\},\{3\}\}|\\
&=|\{\raisebox{-1ex}{ \includegraphics[scale=0.08]{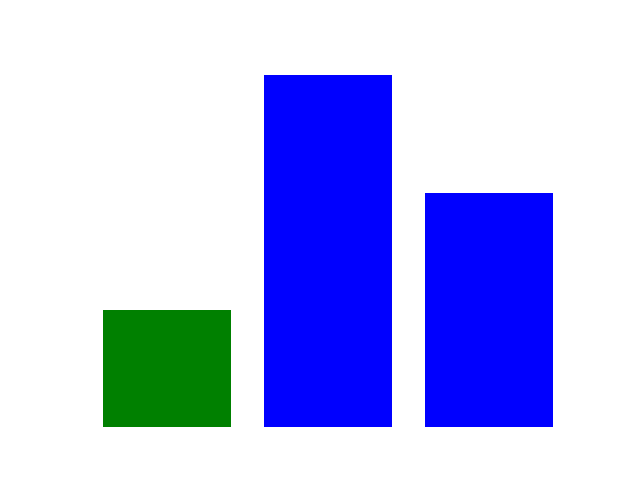}},\raisebox{-1ex}{ \includegraphics[scale=0.08]{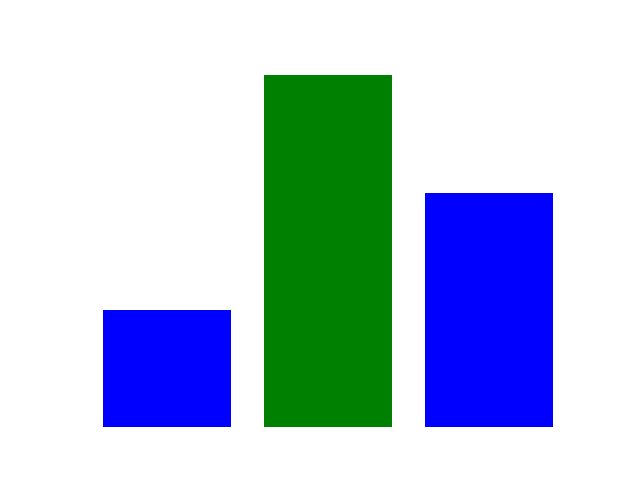}},\raisebox{-1ex}{ \includegraphics[scale=0.08]{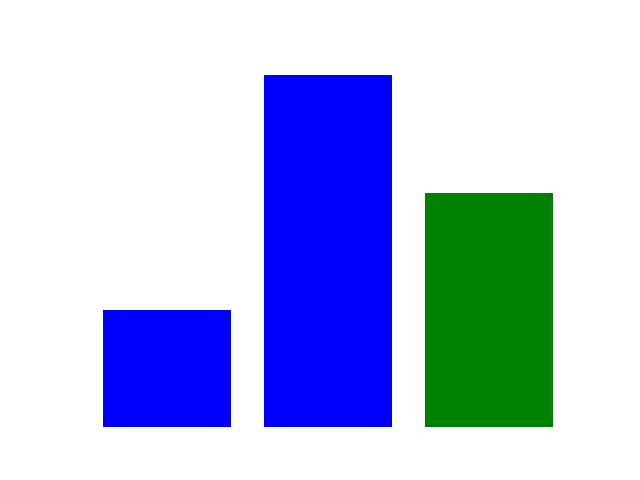}}\} \times \{\raisebox{-1ex}{ \includegraphics[scale=0.08]{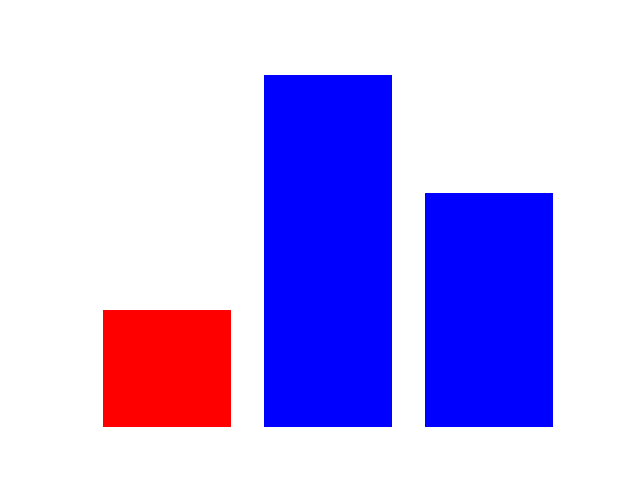}},\raisebox{-1ex}{ \includegraphics[scale=0.08]{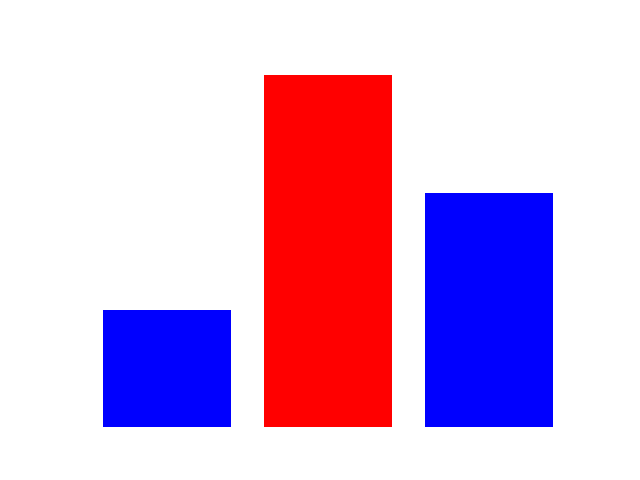}},\raisebox{-1ex}{ \includegraphics[scale=0.08]{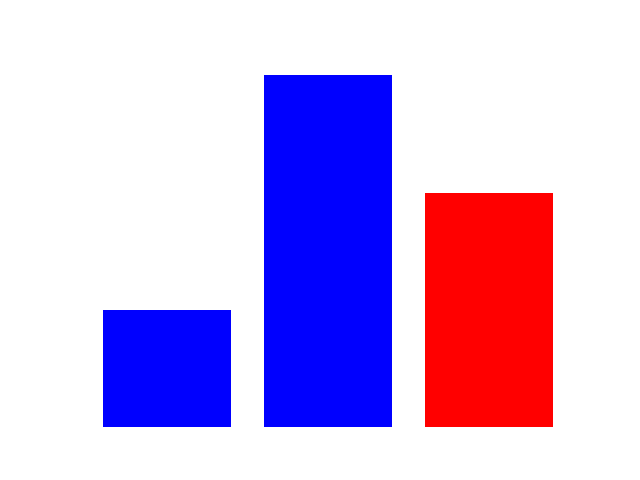}}\}|\\&=|\{(\{1\},\{1\}), (\{1\},\{2\}),(\{1\},\{3\}),(\{2\},\{1\}),(\{2\},\{2\}),(\{2\},\{3\}),\\&(\{3\},\{1\}),(\{3\},\{2\}),(\{3\},\{3\})\}|\\&=|\{\raisebox{-1ex}{ \includegraphics[scale=0.08]{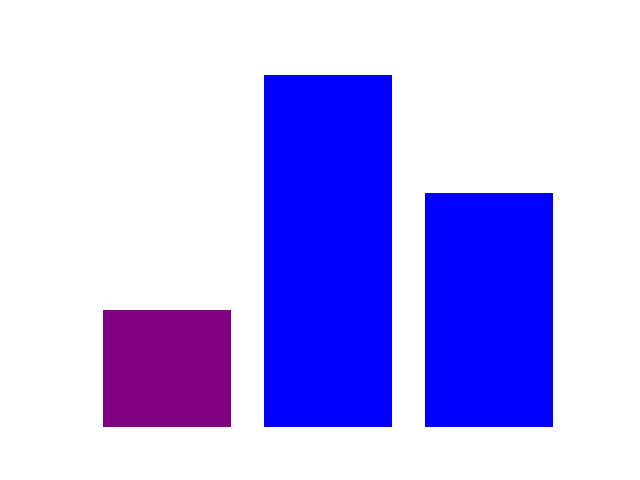}},\raisebox{-1ex}{ \includegraphics[scale=0.08]{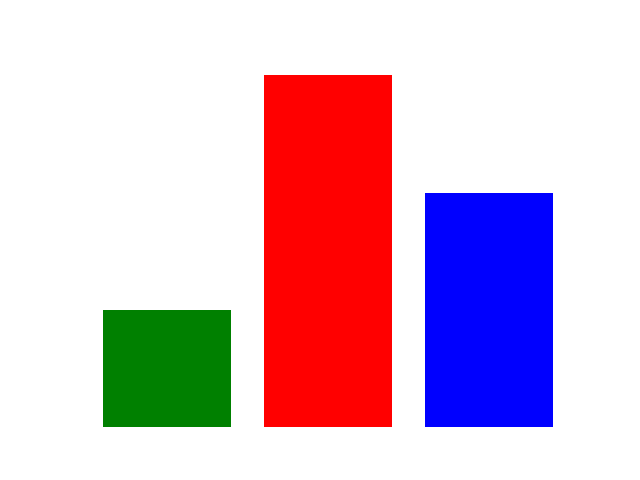}},\raisebox{-1ex}{ \includegraphics[scale=0.08]{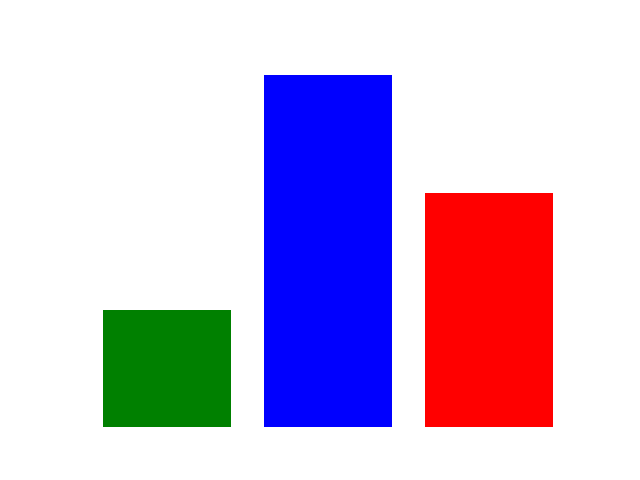}},\raisebox{-1ex}{ \includegraphics[scale=0.08]{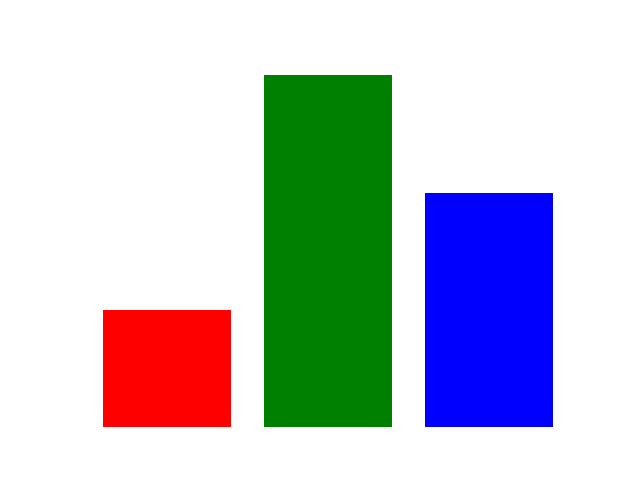}},\raisebox{-1ex}{ \includegraphics[scale=0.08]{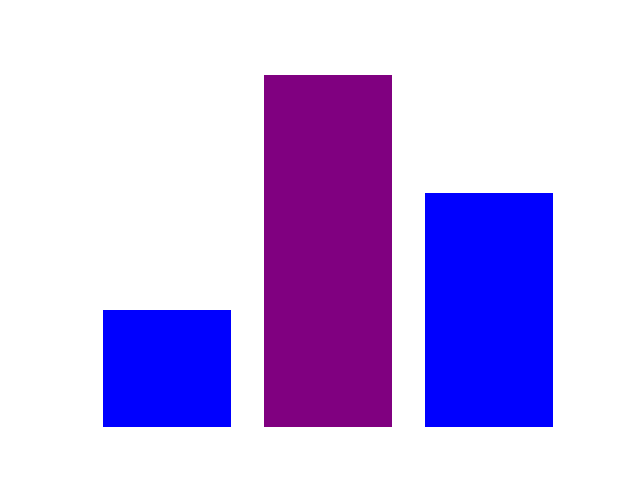}},\raisebox{-1ex}{ \includegraphics[scale=0.08]{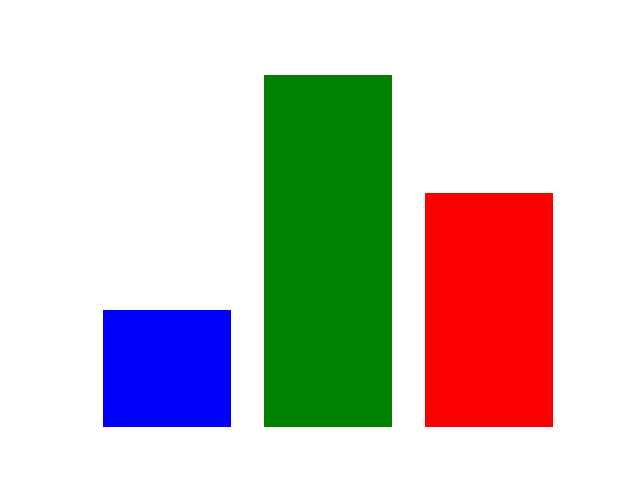}}\\&\raisebox{-1ex}{ \includegraphics[scale=0.08]{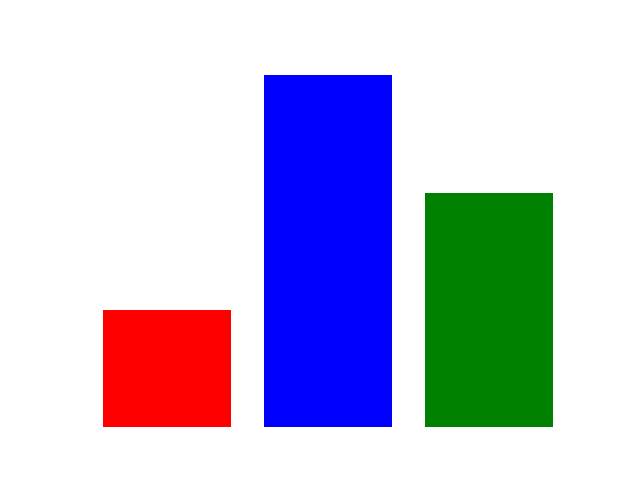}},\raisebox{-1ex}{ \includegraphics[scale=0.08]{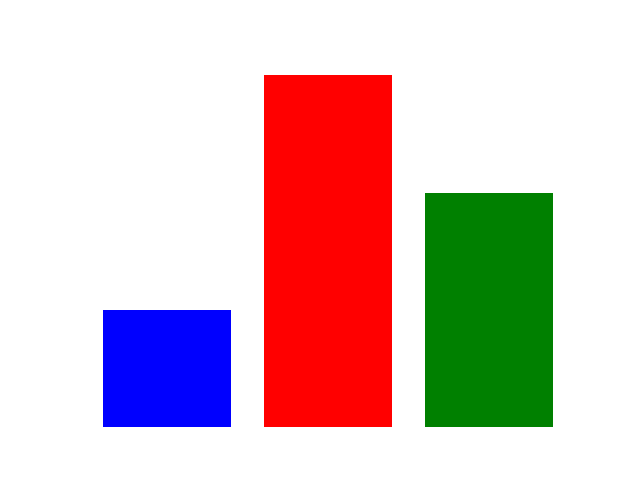}},\raisebox{-1ex}{ \includegraphics[scale=0.08]{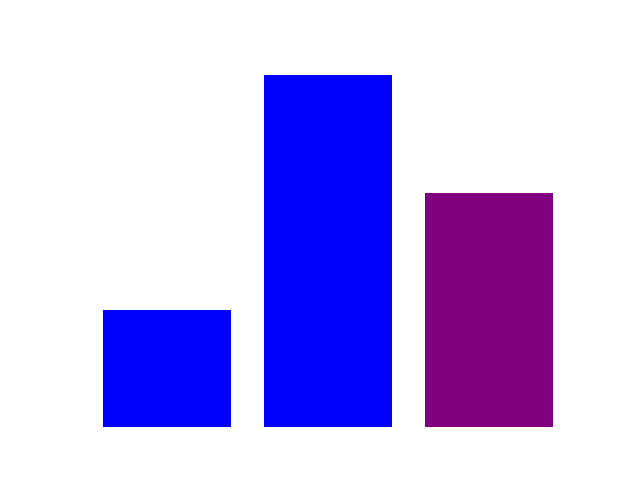}}\}|\\&=|\{(\{1\},\{1\}), (\{2\},\{2\}),(\{3\},\{3\})\}| \\&+ |\{(\{1\},\{2\}), (\{2\},\{1\}),(\{1\},\{3\}),(\{3\},\{1\})\}| \\&+ |\{(\{2\},\{3\}), (\{3\},\{2\})\}|\\\\&=|\{\overbrace{\raisebox{-1ex}{\includegraphics[scale=0.08]{pictures/vargas_signature_example/Figure_1_1.png}}}^{\langle 1 \otimes 1, \Delta_{\scriptsize\,\superinfiltration}(1)\rangle = 1 },\raisebox{-1ex}{ \includegraphics[scale=0.08]{pictures/vargas_signature_example/Figure_2_2.png}},\raisebox{-1ex}{ \includegraphics[scale=0.08]{pictures/vargas_signature_example/Figure_3_3.png}}\}| \\\\&+ |\overbrace{\{\raisebox{-1ex}{ \includegraphics[scale=0.08]{pictures/vargas_signature_example/Figure_1_2.png}},\raisebox{-1ex}{ \includegraphics[scale=0.08]{pictures/vargas_signature_example/Figure_2_1.png}}}^{\langle 1 \otimes 1, \Delta_{\scriptsize\,\superinfiltration}(12)\rangle = 2},\raisebox{-1ex}{ \includegraphics[scale=0.08]{pictures/vargas_signature_example/Figure_1_3.png}},\raisebox{-1ex}{ \includegraphics[scale=0.08]{pictures/vargas_signature_example/Figure_3_1.png}}\}| \\\\&+ |\overbrace{\{\raisebox{-1ex}{ \includegraphics[scale=0.08]{pictures/vargas_signature_example/Figure_2_3.png}},\raisebox{-1ex}{ \includegraphics[scale=0.08]{pictures/vargas_signature_example/Figure_3_2.png}}\}}^{\langle 1 \otimes 1, \Delta_{\scriptsize\,\superinfiltration}(21)\rangle = 2}|\\\\&= \Big \langle \PC (132), 1 \Big \rangle + 2 \Big \langle \PC (132), 12 \Big \rangle + 2 \Big \langle \PC (132), 21 \Big \rangle \\&= \Big \langle \PC (132), 1 + \textcolor{red}{2} \,12 + \textcolor{red}{2} \,21 \Big \rangle\\&= \Big \langle \PC (132), 1 \,\superinfiltration\, 1 \Big \rangle.
\end{align*}

This terminology is also motivated by an identity
similar to \emph{Chen's identity} (\cite[p.10]{lyons2002system})
\begin{align}
\label{eq:perm_chen}
\forall \Lambda, \Upsilon \in \bigsigma:\;\PC(\Lambda \square \Upsilon)  =   \PC(\Lambda) \square \PC(\Upsilon).
\end{align}
Inspired by the work of \cite{bandt2002permutation} and \cite{vargas2014hopf}, aim of this work is to generalize the $\raisebox{0.5ex}{\rotatebox[origin=c]{90}{$\Rightarrow$}}$-Hopf algebra to encode more general patterns. Given a permutation $\Lambda$, we search for permutation patterns $\sigma \in \bigsigma_{n_{k}}$ \begin{figure}[H]
    \centering\includegraphics[scale=0.28]{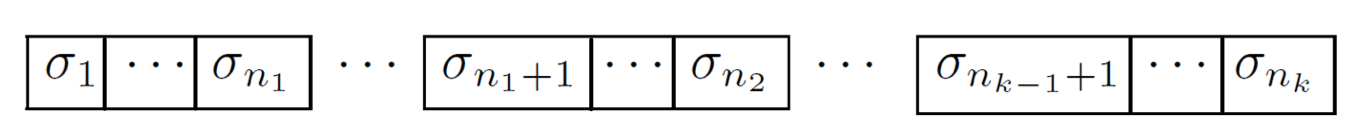}
    
\end{figure}
where for all $j = 1,\dots,k$, $n_{0}:=0$, the values from $n_{j-1}+1$ to $n_{j}$ need to occur consecutively in time. These patterns are known in the literature as \textit{vincular permutation patterns}. They seem to have been first introduced in \cite{babson2000generalized};
see also \cite{branden2011mesh} using the more recent term ``vincular pattern''.

As an example,  Let $\tikzmarknode[rectangle,black]{}{\textup{2}}\tikzmarknode[rectangle,black]{}{\textup{1}}\;\;\tikzmarknode[rectangle,black]{}{\textup{3}}$, be the pattern we search on $134265$. Then
\begin{table}[H]
\centering
\begin{tabular}{ccc}
\includegraphics[width=2cm]{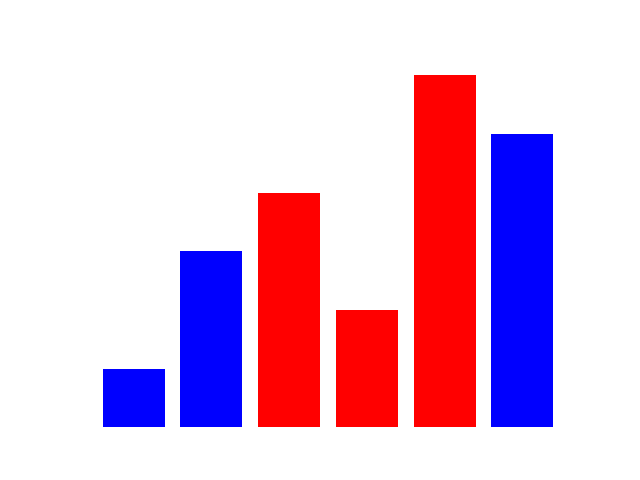} & 
\includegraphics[width=2cm]{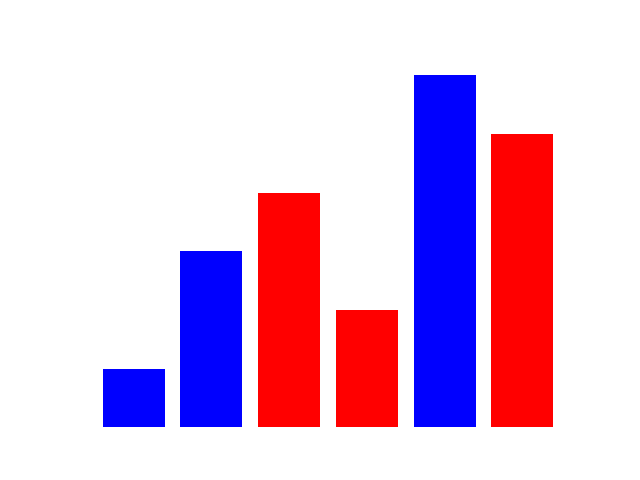} & 
\includegraphics[width=2cm]{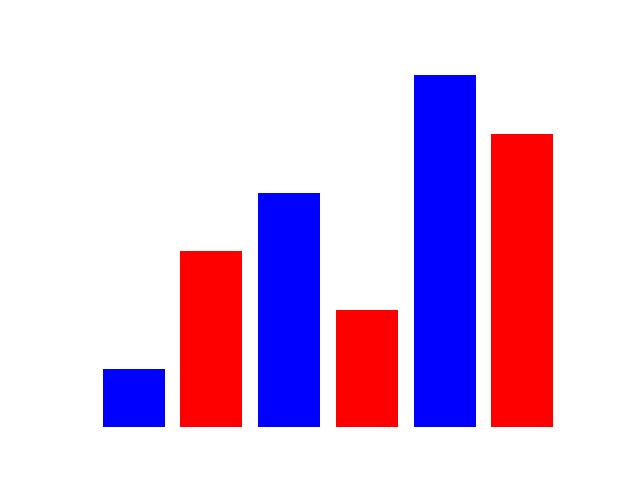} \\
426 & 425 & 325 
\end{tabular}
\end{table}
$426$ and $425$ are compatible with $$\tikzmarknode[rectangle,black]{}{\textup{2}}\tikzmarknode[rectangle,black]{}{\textup{1}}\;\;\tikzmarknode[rectangle,black]{}{\textup{3}}$$ while $325$  is not. We encode this information as a \textit{pair}: $(\stds,\sigma)$, where $\stds$ is an \textit{interval partition} of $[n]$, and $\sigma \in \bigsigma_{n}$, a permutation of the same \quotationmarks{size}. In this example, we have: $\tikzmarknode[rectangle,black]{}{\textup{2}}\tikzmarknode[rectangle,black]{}{\textup{1}}\;\;\tikzmarknode[rectangle,black]{}{\textup{3}}:=(\{\{1,2\},\{3\}\},213)$.   

As an intermediate step, we introduce a new Hopf algebra on interval partitions
\begin{align*}
\FreeIntPart&:=\left(\bigoplus_{n \in \N}\Q[\SetStandardizedIntervalPartitions_{n}],\qspart,\deconc\right)
\end{align*}
where
\begin{align*}
  \SetStandardizedIntervalPartitions_{n}&:=\{\stds\;\text{is an interval partition of}\;[n]\}.
\end{align*}
We define a corresponding family of ``signatures''
\begin{align*}
\left(\IPC\left(\stdL\right)\right)_{\stdL \in \SetStandardizedIntervalPartitions}.
\end{align*} 
 The terminology is again motivated by identities which hold for these functionals, analogous to \Cref{eq:signature_permutation} and \Cref{eq:perm_chen}. Furthermore, we show that this Hopf algebra is isomorphic to a Hopf algebra on words, where the letters correspond to blocks of a partition. Finally, we \quotationmarks{combine} our Hopf algebra on interval partitions with the superinfiltration $\raisebox{0.5ex}{\rotatebox[origin=c]{90}{$\Rightarrow$}}$-Hopf algebra and define a new Hopf algebra on vincular permutation patterns, 
\begin{align*}
\FreeVincPer:=(\bigoplus_{n\in \N} \Q[\SetStandardizedIntervalPartitions_{n} \times \bigsigma_{n}], \qsgen,\deconcgen).    
\end{align*} The $\qsgen$ product originates from our $\qspart$ product on partitions and the $\superinfiltration$ product on permutations, see \cite{vargas2014hopf}. The coproduct is also a \quotationmarks{combination} of the other two. The Hopf algebra $\FreeVincPer$ generalizes  $\FreePer$ and $\FreeIntPart$ as well. A correspondent family of signatures is also introduced.

\section{Notation}

\begin{itemize}

\item[--]  
    We denote the natural numbers, including $0$,
    with $\N$
    and the strictly positive natural numbers
    with $\N_{\ge 1}$.
\item[--]  
    Let $A,B$ be sets. If $A \cap B = \emptyset$, then we write $A \uplus B:= A \cup B$. This notation comes in handy when we assume that the sets are disjoint.
    With this notation,
    \begin{align*}
        |\biguplus_{i\in I}A_{i}|=\sum_{i\in I}|A_{i}|.
    \end{align*}
    if only finitely many sets $A_{i}$ are non-empty and $\forall i\in I: |A_{i}|< \infty$.

    \item[--]
    For $n \in \N_{\ge 1}$ we denote
    \begin{align*}
        [n] := \{1,\dots,n\}. 
    \end{align*}
    We set
    \begin{align*}
        [0] := \emptyset.
    \end{align*}
    
    \item[--]
    We say that $\unstdI$ is a \textit{(finite) partition} $\unstdI = \{ \unstdI_1, ..., \unstdI_I\}$
    of a set $A$
    if the (finite) family $\unstdI_i, i \in I$, consists of non-empty, pairwise disjoint sets, and their union is $A$,
    $\bigcup_{i}\unstdI_{i} = A$.
    All partitions considered in this work are \textit{finite interval partitions} of \textit{finite subsets} of the positive integers. Let $A \subset \N_{\ge 1}$, with $|A|<\infty$.
    We say that $\unstdI$ is an \textbf{interval partition} of $A$
    if $\unstdI$ is a partition of $A$, whose blocks $\unstdI_{j}$ are intervals on $\N_{\ge 1}$. As an example, if $A=\{2,4,5\}$, then $\{\{4,5\},\{2\}\}$ is an interval partition of $A$, while $\{\{2,4\},\{5\}\}$ is not.
    For a partition $\unstdI$ of a set $A$ we also write
    \begin{align*}
        \punion \unstdI := \bigcup_i \unstdI_i = A,
    \end{align*}
    for its ground-set.

    \item[--]
    The set
    \begin{align*}
        \SetIntervalPartitions := \{X| X\;\text{is an interval partition of some finite subset of}\; \N_{\ge 1} \}
    \end{align*}
    contains ``labeled'' finite interval partition. 
    \item[--] We also fix $A \subset \N_{\ge 1}, A < \infty$ and write 
    \begin{align*}
        \SetIntervalPartitions(A) := \{X| X\;\text{is an interval partition of}\;A\}
    \end{align*}
    
    \item[--] The subset
\begin{align*}
    \SetStandardizedIntervalPartitions :=  \biguplus_{n \in \N} \SetStandardizedIntervalPartitions_{n} \subset \SetIntervalPartitions
\end{align*}
where $\SetStandardizedIntervalPartitions_{n}:=\{A| A\;\text{is an interval partition of}\;[n]\}$,
contains ``unlabeled'' or \textit{standardized} interval partitions. Notice that $\SetIntervalPartitions(\emptyset) = \{\emptyset\}$, i.e., the unique partition for the empty set $\emptyset$ is the empty set itself, which means
\begin{align*}
    \SetStandardizedIntervalPartitions_{0}:=\{\emptyset\}.
\end{align*} Also, we have:
\begin{align*}
    \SetStandardizedIntervalPartitions =  \biguplus_{\ell \in \N} \NumBlocks_{\ell}
\end{align*}
where 
\begin{align*}
    \NumBlocks_{\ell}: = \{\stds \in \SetStandardizedIntervalPartitions|\stds = \{\stds_{1},...,\stds_{\ell}\} \}
\end{align*}
i.e., all standardized partitions having $\ell$ blocks.
\item[--] We use gothic letters to denote elements of $\SetStandardizedIntervalPartitions$, for example $\stds \in \SetStandardizedIntervalPartitions$. We use capital calligraphical letters when we consider any element of $\SetIntervalPartitions$, for example $\unstdI \in \SetIntervalPartitions$.

\item[--] We define a partial order relation on $\SetIntervalPartitions$, as follows
\begin{align*}
\unstdI\le \unstdJ \iff \punion \unstdI = \punion \unstdJ\;\land\; \forall j \in J: \exists i \in I: \unstdI_{i} \subset \unstdJ_{j}. 
\end{align*}
If $I$ and $J$ are not comparable, we write $I \perp J$.
As an example: $\{\{2\},\{3\},\{1\}\} \le \{\{2,3\},\{1\}\}$, while $\{\{2\},\{3\},\{1\}\} \perp \{\{5\},\{3\},\{4\}\}$ and $\{\{1,2\},\{3\}\} \perp \{\{1\},\{2,3\}\}$.

\item[--] We denote with $\bigsigma_{n}$ the symmetric group of order $n$,
\begin{align*}
\bigsigma_{n}:=\{f|\;f:[n] \to [n]\;\text{is a bijection} \}.
\end{align*} Let $\bigsigma:= \bigcup_{n}\bigsigma_{n}$ denote the set of all permutations. Notice that $\bigsigma_{0}:=\{f|f:\, \emptyset \to \emptyset\}$ contains the unique empty permutation. We denote elements of $\bigsigma$, with Greek letters, for example, $\sigma \in \bigsigma$. We use one-line notation, for example,
\begin{align*}
    \bigsigma_{2} = \{12,21\}.
\end{align*}

\item[--] \label{def:standardization_word} Let $A$ denote a set which we call \textit{alphabet}. We call its elements \textit{letters}. Denote with $A^{*}$ the set of words on the alphabet, i.e. finite sequences of elements of $A$. If $A$ is a totally ordered set, define the \textit{standardization} of a word $w \in A^{*}$ as the relative order of the letters which appear in $w$. If a letter appears more than once, we order them from left to right. For example, if  $A := \N$ with the usual order
\begin{align*}
  \st\left(234121\right) = 356142. 
\end{align*}
Permutations written in one-line notation can be seen as (particular) words on positive integers. When we restrict $\sigma \in \bigsigma_{n}$ to a subset $A \subset [n]$, with $|A|=k$, we can standardize the subword to get an element of $\bigsigma_{k}$. As an example
\begin{align*}
    \st\left(645123\evaluatedAt{\{2,4,5\}}\right) &= \st\left(423\right) =  312.
\end{align*}
\item[--] If $\sigma \in \bigsigma_{n}$, then $|\sigma|:=n$. Analogously, if $\stds \in \SetStandardizedIntervalPartitions_{n}$, then $|\stds|:=n$.
\end{itemize}

\section{Finite interval partitions}
\label{sec:Finite_interval_partitions}

\subsection{Auxiliary operations}

Before introducing the algebraic operations, we need some auxiliary operations. Let $A$ be a finite subset of $\N_{\ge 1}$. The following operation yields the coarsest interval partition of $A$.
\newcommand\Succ{\mathsf{Succ}}
\begin{definition}[Cliques]
\label{def:cliques}
\begin{align*}
      \cliques: \{A|A \subset \N_{\ge 1}, |A| < \infty\} &\to    \SetIntervalPartitions\\
      A&\mapsto \{[x]_{\sim \langle \Succ_A \rangle}| x \in A\}
\end{align*}
where $(x,y) \in  \Succ_A \subset A \times A \iff y = x+1$ and $\langle  \Succ_A\rangle $ is the equivalence relation generated by $ \Succ_A$. 
    \end{definition}
\begin{example}
\begin{align*}
      \cliques(\{2,4,5,6\}) &= \{\{2\},\{4,5,6\}\}\\\\
      \cliques(\emptyset) &= \emptyset.
\end{align*}
\end{example}
\begin{remark}
\label{rem:IfinercliquesA}
Let $\unstdI$ be any partition of $A$, then 
\begin{align*}
\unstdI \text{ is an interval partition of $A$ }
\Leftrightarrow
\unstdI \le \cliques(A).
\end{align*}

\end{remark}

Now fix an interval partition $\unstdI \in \SetIntervalPartitions$ and a finite subset of positive integers, $A \subset \N_{\ge 1}$. We define an interval partition of $A \cap \punion\unstdI$ as follows.

\begin{definition}[Cliques of $A$ through $\unstdI$]
\label{def:cliques_A_through_I}
Let $\unstdI \in \SetIntervalPartitions$ and $A \subset \N_{\ge 1}, A < \infty$.
\begin{align*}
\unstdI(A)&:=\{\unstdI_{j} \cap c|\; c \in \cliques(A), j \in J\}\setminus \{\emptyset\}.
\end{align*}
 \end{definition}
\begin{remark}
 Since the intersection of two intervals is again an interval, $\unstdI(A)$ is a set of intervals. These intervals are pairwise disjoint since $\unstdI$ and $\cliques(A)$ are partitions and this implies that $\unstdI_{i}\cap c$ and $\unstdI_{j}\cap c^{\prime}$ are also pairwise disjoint. We clearly have $\punion\unstdI(A)=A \cap\punion\unstdI$. Therefore, by \Cref{rem:IfinercliquesA}, one has
\begin{align*}
    \unstdI(A) &\le \cliques(A \cap \punion\unstdI).
\end{align*}
\end{remark}

\begin{example}
For $\unstdI:= \{\{2, 3, 4\}\}$ and $A := \{2,4,5\}$
we have
\begin{align*}
    \unstdI(A) =  \{\{2\},\{4\}\}.
\end{align*}
If $\unstdI:= \{\{2, 3\},\{4,5\},\{6,7\},\{8\}\}$ and $A := \{2,4,5,6\}$
we have
\begin{align*}
    \unstdI(A) =  \{\{2\},\{4,5\},\{6\}\} <  \{\{2\},\{4,5,6\}\} = \cliques(A \cap \punion\unstdI).
\end{align*}
In case $A =\emptyset$
\begin{align*}
    \unstdI(A)=\emptyset.
\end{align*}
\end{example}

Finally, we can convert elements of $\SetIntervalPartitions$, into elements of $\SetStandardizedIntervalPartitions$ as follows.
\begin{definition}[Standardization of a partition] 
\label{def:std_partition}
\begin{align*}
      \std:
      \SetIntervalPartitions
      \to  \SetStandardizedIntervalPartitions
\end{align*}
Let $i_{1},\dots,i_{\ell_{1}},\dots,i_{\ell_{2}},\dots,i_{\ell_{k}} \in \N_{\ge 1},$ where $k \in \N$, such that 
\begin{align*}
    i_{1} \prec \cdots \prec i_{\ell_{1}} < i_{\ell_{1}+1}\prec \cdots \prec i_{\ell_{2}} < \cdots < i_{\ell_{k-1}+1} \prec \cdots \prec i_{\ell_{k}}
\end{align*} where $i \prec j :\iff j-i = 1$. In other words, we consider an interval partition with $k$ blocks. Then define
\begin{align*}
      &\std\left(\left\{\{i_{1},...,i_{\ell_{1}}\},\{i_{\ell_{1}+1},...,i_{\ell_{2}}\},...,\{i_{\ell_{k-1}+1},...,i_{\ell_{k}}\}\right\}\right)\\ &:=  \{\{1,...,\ell_{1}\},\{\ell_{1}+1,...,\ell_{2} \},...,\{\ell_{k-1}+1,...,\ell_{k}\}\}
\end{align*}
with
\begin{align*}
   1 \prec \cdots \prec \ell_{1} \prec \ell_{1}+1 \prec \cdots \prec \ell_{2} \prec \cdots \prec \ell_{k-1}+1 \prec \cdots \prec \ell_{k}.
\end{align*}
\end{definition}
\begin{example}
\begin{align*}
      \std(\{\{2\},\{4,5\},\{6\}\} ) = \{\{1\},\{2,3\},\{4\}\} 
\end{align*}
\end{example}
\begin{remark}
Let $\unstdI \in \SetIntervalPartitions$. We can turn $A \subset \punion\unstdI$, into a standardized interval partition, namely
\begin{align*}
   \std(\unstdI(A)). 
\end{align*}
\begin{example}
If $\unstdI:= \{\{2, 3\},\{4,5\},\{6,7\},\{8\}\}$ and $A := \{2,4,5,6\}$, we have
\begin{align*}
    \std(\unstdI(A)) &= \{\{1\},\{2,3\},\{4\}\}.
\end{align*}
\end{example}
\end{remark}

\subsubsection{Gluing partitions}

We now introduce a binary operation on interval partitions.

\begin{definition}[Gluing partitions]
\label{def:gluing_partitions}
Let $\unstdI, \unstdJ \in \SetIntervalPartitions$.
First define the relation
\begin{align*}
    R_{\unstdI,\unstdJ} \subset (\unstdI \cup \unstdJ) \times (\unstdI \cup \unstdJ)\\ 
     (E,F) \in R_{\unstdI,\unstdJ} \iff E \cap F \neq \emptyset.
\end{align*}

Then define the binary operation as 
\begin{align*}
\unstdI \gluepartitions \unstdJ:= \left\{ \bigcup_{X \in [T]_{\sim\langle R_{\unstdI,\unstdJ} \rangle}} X \Bigg|\; T \in \unstdI \cup \unstdJ\right\},
\end{align*}
where $\sim\langle R_{\unstdI,\unstdJ} \rangle$ is the smallest equivalence relation which contains $R_{\unstdI,\unstdJ}$.
\end{definition}
\begin{example}
Let $\unstdI=\{\{1\},\{2\},\{3,4\},\{5,6,7\}\}$ and $\unstdJ=\{\{2,3\},\{4,5\}\}$. Then $$\unstdI \gluepartitions \unstdJ=\{\{1\},\{2,3,4,5,6,7\}\}.$$ We visualize $R_{\unstdI,\unstdJ}$ in the following picture
\begin{figure}[H]
  \centering
 \begin{tikzpicture}[scale=0.8, rotate=180]
  \tikzset{vertex/.style={rectangle,draw,minimum size=1cm}}
  \tikzset{left/.style={yshift=1.25cm}}
  \tikzset{right/.style={yshift=-1.25cm}}
  \node[vertex,left,draw=none] (A0) at (6.5,0.65) {$\mathcal{I}=$};
  \node[vertex,left] (A1) at (5,0.65) {$\{1\}$};
  \node[vertex,left] (A2) at (3,0.65) {$\{2\}$};
  \node[vertex,left] (A3) at (1,0.65) {$\{3,4\}$};
   \node[vertex,left] (A4) at (-1,0.65) {$\{5,6,7\}$};
  \node[vertex,right] (B1) at (1,0) {$\{4,5\}$};
  \node[vertex,right] (B2) at (3,0) {$\{2,3\}$};
   \node[vertex,right,draw=none] (B0) at (6.5,0){$\mathcal{J}=$};
  \draw [->, thick] (A0);
  \draw [->, thick] (B0);
  \draw [-, thick] (A2) -- (B2);
  \draw [-, thick] (B2) -- (A3);
  \draw [-, thick] (A3) -- (B1);
  \draw [-, thick] (B1) -- (A4);
\end{tikzpicture}
\end{figure}
\end{example}

We now state several lemmas that clarify the behavior of the gluing operation.

\begin{lemma}
\label{lemma:gluing_welldefined}
Let $\unstdI, \unstdJ \in \SetIntervalPartitions$. Then
\begin{align*}
\unstdI \gluepartitions \unstdJ \in \SetIntervalPartitions\left(\punion\unstdI \cup \punion\unstdJ\right).
\end{align*}
In particular $\forall \unstdI, \unstdJ \in \SetIntervalPartitions,\;\; \unstdI \gluepartitions \unstdJ \in \SetIntervalPartitions.$
\end{lemma}
\begin{proof}
Let $T \in \unstdI \cup  \unstdJ$ and  consider its equivalence class
\begin{align*}
    [T]_{\sim\langle R_{\unstdI,\unstdJ} \rangle}.
\end{align*}
From the definition of $\langle R_{\unstdI,\unstdJ} \rangle$ it follows that we can write 
\begin{align*}
    [T]_{\sim\langle R_{\unstdI,\unstdJ} \rangle} = \{X_{1},...,X_{n}\}\;\text{where}\;X_{i} \in \unstdI \lor X_{i} \in \unstdJ
\end{align*}
for some positive integer $n$, where $\forall i \in \{1,...,n-1\}: X_{i} \cap X_{i+1} \neq \emptyset$.
Obviously, the union of two intersecting intervals is again an interval.
It then follows by induction that 
\begin{align*}
    \bigcup_{i=1}^{n}X_{i}
\end{align*}
is an interval. It remains to show that 
\begin{align*}
\unstdI \gluepartitions \unstdJ= \left\{ \bigcup_{X \in [T]_{\sim\langle R_{\unstdI,\unstdJ} \rangle}} X \Bigg|\; T \in \unstdI \cup \unstdJ\right\}
\end{align*}
is a partition of $\punion\unstdI \cup \punion\unstdJ$. Clearly
\begin{align*}
\bigcup_{T \in \unstdI \cup \unstdJ} \bigcup_{X \in [T]_{\sim\langle R_{\unstdI,\unstdJ} \rangle}} X = \punion\unstdI \cup \punion\unstdJ.
\end{align*}
Since
\begin{align*}
\left\{[T]_{\sim\langle R_{\unstdI,\unstdJ} \rangle} \Bigg|\; T \in \unstdI \cup \unstdJ\right\}
\end{align*}
is a partition of the blocks of $\unstdI \cup \unstdJ$, we have that
\begin{align*}
   \forall  T \in \unstdI \cup \unstdJ:  [T]_{\sim\langle R_{\unstdI,\unstdJ} \rangle} \neq \emptyset.
\end{align*}
and therefore $\bigcup_{X \in [T]_{\sim\langle R_{\unstdI,\unstdJ} \rangle}} X \neq \emptyset$. Now let
\begin{align*}
\bigcup_{X \in [T]_{\sim\langle R_{\unstdI,\unstdJ} \rangle}} X \cap \bigcup_{X \in [S]_{\sim\langle R_{\unstdI,\unstdJ} \rangle}} X \neq \emptyset.
\end{align*}
This means $\exists X \in [T]_{\sim\langle R_{\unstdI,\unstdJ} \rangle}$ and $\exists Y \in [S]_{\sim\langle R_{\unstdI,\unstdJ} \rangle}$ such that $X \cap Y \neq \emptyset$. Therefore $[T]_{\sim\langle R_{\unstdI,\unstdJ} \rangle} =[S]_{\sim\langle R_{\unstdI,\unstdJ} \rangle}$, which means
\begin{align*}
\bigcup_{X \in [T]_{\sim\langle R_{\unstdI,\unstdJ} \rangle}} X = \bigcup_{X \in [S]_{\sim\langle R_{\unstdI,\unstdJ} \rangle}} X.
\end{align*}
\end{proof}

This operation is associative.
\begin{lemma}
\label{lemma:gluing_associative}
Let $\unstdI,\unstdI^{\prime},\unstdI^{\prime\prime} \in \SetIntervalPartitions$. Then
\begin{align*}
(\unstdI \gluepartitions \unstdI^{\prime})\gluepartitions \unstdI^{\prime \prime} = \unstdI \gluepartitions (\unstdI^{\prime}\gluepartitions \unstdI^{\prime \prime}).
\end{align*}
\end{lemma}
For a proof of this result, see \Cref{lemma: associativity gluing}, which can be seen as an instance of a more general phenomenon, we refer to \Cref{app:sec:gluing_associative}. 
\begin{lemma}

\label{lemma:gluing_finess}
Let $A, A^{\prime} \subset \N_{\ge 1}$ be finite subsets. Let $\unstdI,\unstdJ \in \SetIntervalPartitions(A)$ and $\unstdI^{\prime},\unstdJ^{\prime} \in \SetIntervalPartitions(A^{\prime})$. If
\begin{align*}
\unstdI \le \unstdJ \;\text{and}\; \unstdI^{\prime} \le  \unstdJ^{\prime},
\end{align*}
then
\begin{align*}
\unstdI \gluepartitions \unstdI^{\prime} \le \unstdJ \gluepartitions \unstdJ^{\prime}.
\end{align*}
\end{lemma}
\begin{proof}
Let $z \in \unstdI \gluepartitions \unstdI^{\prime}$. Then
\begin{align*}
   z = \bigcup_{i=1}^{m}X_{i},
\end{align*}
where $\forall i\in\{1,...,m-1\}: X_{i} \cap X_{i+1} \neq \emptyset$ and $X_{i} \in \unstdI$ or $X_{i} \in \unstdI^{\prime}$. From the hypothesis, we know that 
\begin{align*}
    \forall  i\in\{1,...,m\}: \exists Y_{i} \in \unstdJ \cup \unstdJ^{\prime}: X_{i} \subset Y_{i}.
\end{align*}
Now since
\begin{align*}
    \forall  i\in\{1,...,m-1\}: Y_{i} \cap Y_{i+1} \supset X_i \cap X_{i+1} \neq \emptyset,
\end{align*}
We have that $\exists w  \in \unstdJ \gluepartitions \unstdJ^{\prime}$ such that 
$z \subset \bigcup_{i=1}^{m} Y_{i} \subset  w$ and we are done. 
\end{proof}

\begin{corollary}
\label{corollary:gluing_finess}
Let $\unstdI \in \SetIntervalPartitions$, $A, A^{\prime} \subset \N_{\ge 1}, |A|,|A^{\prime}| < \infty$. Then
\begin{align*}
\unstdI(A) \gluepartitions \unstdI(A^{\prime}) \le \unstdI(A \cup A^{\prime}).
\end{align*}
\end{corollary}
\begin{proof}
Let $\unstdW^{\prime}=\unstdZ^{\prime}=\unstdI(A\cup A')$, $\unstdW=\unstdI(A)$ and $\unstdZ =\unstdJ(A)$. We just need to show that $\unstdW \le \unstdW^{\prime}$ and $\unstdZ \le \unstdZ^{\prime}$ hold. We can then apply \Cref{lemma:gluing_finess}: it follows immediately from the definition of $\gluepartitions$ that $\unstdW^{\prime} \gluepartitions \unstdZ^{\prime} =\unstdI(A\cup A') \gluepartitions \unstdI(A\cup A') = \unstdI(A\cup A')$. First, for any $c  \in \cliques(A)$
    and any $x \in c$
    \begin{align*}
        c = [x]_{\sim \langle \Succ_A \rangle}
        \subset [x]_{\sim \langle \Succ_{A\cup A'} \rangle} \in \cliques(A \cup A').
    \end{align*}

    Let $X \in \unstdI(A)$.
    Claim: there is a $Y \in \unstdI(A\cup A')$
    with
    \begin{align*}
       X \subset Y. 
    \end{align*}
    Indeed, we can write $X = c \cap \unstdI_{j}$ for some $c \in \cliques(A), \unstdI_j \in \unstdI$.
    Then, for any $x \in c$,
    \begin{align*}
        X = (c \cap \unstdI_{j})
        \subset
        ([x]_{\sim \langle \Succ_{A \cup A^{\prime}}\rangle} \cap \unstdI_{j}) =: Y
        \in \unstdI(A \cup A^{\prime}).
    \end{align*}
    Analogously, for every $X \in \unstdI(A')$
    there is a $Y \in \unstdI(A\cup A')$ with $X \subset Y$.
\end{proof}

\begin{lemma}
\label{lemma:gluing_finess3}
Let $\unstdI,\unstdI^{\prime} \in \SetIntervalPartitions$. Consider
\begin{align*}
\unstdJ = \unstdI \gluepartitions \unstdI^{\prime}.
\end{align*}
Then
\begin{align*}
\unstdI \le \unstdJ(\punion\unstdI)\;\text{and}\;\unstdI^{\prime} \le \unstdJ(\punion\unstdI^{\prime}).
\end{align*}
\end{lemma}
\begin{proof}
Let $z \in \unstdJ$, then 
\begin{align*}
    z = \bigcup_{i=1}^{n}X_{i}
\end{align*}
where $\forall i \in \{1,...,n-1\}$, $X_{i} \cap X_{i+1} \neq \emptyset$ and $X_{i} \in \unstdI$ or $X_{i} \in \unstdI^{\prime}$.
Now recall that 
\begin{align*}
   \unstdJ(\punion\unstdI) &= \{c \cap \unstdJ_{i}| c \in \cliques(\punion\unstdI), i \in I\}\setminus\{\emptyset\},\\
    \unstdJ(\punion\unstdI^{\prime})) &= \{c \cap \unstdJ_{i}| c \in \cliques(\punion\unstdI^{\prime}), i \in I\}\setminus\{\emptyset\}.
\end{align*}
Let $w \in \unstdI$, then obviously there exists a $z \in \unstdJ$ such that $w \subset z$. Also by definition of $\cliques(\punion\unstdI)$, we know that $\unstdI \le \cliques(\punion\unstdI)$ and therefore $\exists c \in \cliques(\punion\unstdI)$ such that $w \subset c$. Therefore $w \subset z \cap c$ which means that $\unstdI \le \unstdJ(\punion\unstdI)$. Analogously one shows that $\unstdI^{\prime} \le \unstdJ(\punion\unstdI^{\prime})$.
\end{proof}

\subsection{Algebraic operations}

We now define algebraic operations on $\SetStandardizedIntervalPartitions$. 
We first define the free $\Q$-vector space. It can be graded according to cardinalities or the number of blocks. 

\begin{definition}[$\Q$-vector space over interval partitions] 
\label{def:free_vector_space_interval partitions}
\begin{align}
\label{eq:n_grading_partitions}
    \FreeIntPart&:= \bigoplus_{n \in \N} \Q[\SetStandardizedIntervalPartitions_{n}]\\
\label{eq:blocks_grading_partitions}
    &= \bigoplus_{\ell \in \N} \Q[\NumBlocks_{\ell}].
\end{align}
\end{definition}

\subsubsection{Products, coproducts}

We now define a (non-commutative) product on interval partitions.
\begin{definition}
\label{def:conc_int_part}
Let $\stds:=\{\stds_{1},...,\stds_{m}\},\stdt:=\{\stdt_{1},...,\stdt_{n}\} \in \SetStandardizedIntervalPartitions$. Then define
\begin{align*}
&\conc: \FreeIntPart \otimes \FreeIntPart\to \FreeIntPart\\
    & \stds \conc \stdt := \{\stds_{1},...,\stds_{m},\stdt^{\prime}_{1},...,\stdt^{\prime}_{n}\}
\end{align*}
where for $i=1,...,n$, $\stdt^{\prime}_{i}:=|\punion\stds|+\stdt_{i}$. And extend linearly.  
\end{definition}

\begin{example}
\begin{align*}
        \{\{1,2\}\} \conc  \{\{1\},\{2\}\} =   \{\{1,2\},\{3\},\{4\}\}
\end{align*}
To denote elements of $\SetStandardizedIntervalPartitions$, we will also simply write
\begin{align*}
    \adjacentSquares{2} \conc \adjacentSquares{1} \: \adjacentSquares{1}:= \{\{1,2\}\} \conc  \{\{1\},\{2\}\}  =  \{\{1,2\},\{3\},\{4\}\} =\adjacentSquares{2}\; \adjacentSquares{1}{}\;\adjacentSquares{1}
\end{align*}
and $\e:=\emptyset$.
\end{example}
\begin{proposition}
The $\conc$ product is associative.
\end{proposition}
\begin{proof}
The proof is immediate and is left to the reader.
\end{proof}

We now define a coproduct which is based on the gluing operation.
\begin{definition}
\label{def:coproduct_partitions}
Let $ \stds\in \SetStandardizedIntervalPartitions$.

We define a (cocommutative) coproduct as follows
\begin{align*}
\coqspart(\stds)&:=\sum_{A \cup A^{\prime} = \punion\stds}\;\sum_{\substack{\unstdI \in \SetIntervalPartitions(A),\;\unstdI^{\prime} \in \SetIntervalPartitions(A^{\prime})  \\\\ \unstdI \gluepartitions \unstdI^{\prime} = \stds}}\std(\unstdI) \otimes \std(\unstdI^{\prime})
\end{align*}
and extend linearly. Recall that from \Cref{lemma:gluing_finess3}, it follows that $\unstdI \gluepartitions \unstdI^{\prime} = \stds$ implies that $\unstdI\le \stds(A)$ and $\unstdI \le \stds(A^{\prime})$.
\end{definition}

\begin{example}
\begin{align*}
\coqspart(\{\{1,2\}\})&=\coqspart(\;\adjacentSquares{2}\;)
= \e \otimes  \adjacentSquares{2} + 2  \ \adjacentSquares{1} \otimes  \adjacentSquares{2} + \adjacentSquares{2} \otimes \e \\&+ 2\ \adjacentSquares{2} \otimes \adjacentSquares{1} +  \adjacentSquares{2} \otimes  \adjacentSquares{2} + \adjacentSquares{1} \ \adjacentSquares{1} \otimes  \adjacentSquares{2} +  \adjacentSquares{2} \otimes \adjacentSquares{1}\;\adjacentSquares{1}
\end{align*}
\begin{align*}
\coqspart(\ \{\{1\},\{2\}\}\ )&=\coqspart(\; \adjacentSquares{1} \ \adjacentSquares{1} \; )= \e \otimes \adjacentSquares{1}\ \adjacentSquares{1} + 2\ \adjacentSquares{1} \otimes \adjacentSquares{1} + 2\ \adjacentSquares{1} \otimes \adjacentSquares{1} \ \adjacentSquares{1} + \adjacentSquares{1} \ \adjacentSquares{1} \otimes  \e\\&+ 2\ \adjacentSquares{1} \ \adjacentSquares{1} \otimes \adjacentSquares{1} +  \adjacentSquares{1} \ \adjacentSquares{1} \otimes\adjacentSquares{1} \ \adjacentSquares{1}
\end{align*}
\end{example}

\begin{remark}[Shuffle coproduct]
Notice that $\coqspart$ has a \quotationmarks{co-shuffle part}. Indeed
\begin{align*}
\Delta_{\shuffle}(\stds)&:=\sum_{A \uplus A^{\prime} = \punion\stds}\;\sum_{\substack{\unstdI \in \SetIntervalPartitions(A),\;\unstdI^{\prime} \in \SetIntervalPartitions(A^{\prime})  \\\\ \unstdI \gluepartitions \unstdI^{\prime} = \stds}}\std(\unstdI) \otimes \std(\unstdI^{\prime}).
\end{align*}
is a shuffle coproduct. As an example
\begin{align*}
\Delta_{\shuffle}(\ \adjacentSquares{1} \ \adjacentSquares{1} \ )&= \e \otimes \adjacentSquares{1}\ \adjacentSquares{1} +  \adjacentSquares{1} \ \adjacentSquares{1} \otimes  \e + 2\ \adjacentSquares{1} \otimes \adjacentSquares{1}\,.  
\end{align*}
More precisely, $(\FreeIntPart, \Delta_{\shuffle})$ is isomorphic to a coalgebra on words, where letters are positive integers, see \Cref{rem:words_intergers}. For example
\begin{align*}
\Delta_{\shuffle,\,\textbf{words}}(1 \conctens 1) = \e \otimes 1 \conctens 1 +  1 \conctens 1 \otimes \e + \textcolor{red}{2}\, 1 \otimes 1,
\end{align*}
where $\conctens$ denotes the concatenation of words and $\Delta_{\shuffle,\textbf{words}}$ is the usual shuffle coproduct on words.
\end{remark}
\begin{proposition}
\label{prop:coprod_partitions_coass}
The coproduct $\coqspart$ is coassociative.
\end{proposition}
We first need the following lemma.
\begin{lemma}
\label{lemma:coassociativity_intpart}
Let $\unstdW,\unstdZ\in \SetIntervalPartitions$ such that $\std(\unstdW)=\std(\unstdZ)$. Then
\begin{align*}
&=\bigcup_{A \cup A^{\prime} = \punion\unstdW}\{(\std(\unstdI),\std(\unstdI^{\prime}))|\;\unstdI \in \SetIntervalPartitions(A),\;\unstdI^{\prime} \in \SetIntervalPartitions(A^{\prime}),\;\unstdI \gluepartitions \unstdI^{\prime} = \unstdW\}\\
    &\bigcup_{B \cup B^{\prime} = \punion\unstdZ}\{(\std(\unstdJ),\std(\unstdJ^{\prime}))|\;\unstdJ \in \SetIntervalPartitions(B),\;\unstdJ^{\prime} \in \SetIntervalPartitions(B^{\prime}),\;\unstdJ \gluepartitions \unstdJ^{\prime} = \unstdZ\}.
\end{align*}
In particular, $\forall \unstdW\in \SetIntervalPartitions$
\begin{align*}
    &\bigcup_{A \cup A^{\prime} = \punion\unstdW}\{(\std(\unstdI),\std(\unstdI^{\prime}))|\;\unstdI \in \SetIntervalPartitions(A),\;\unstdI^{\prime} \in \SetIntervalPartitions(A^{\prime}),\;\unstdI \gluepartitions \unstdI^{\prime} = \unstdW\} \\
    &=\bigcup_{B \cup B^{\prime} = \punion\std(\unstdW)_{j}}\{(\std(\unstdJ),\std(\unstdJ^{\prime}))|\;\unstdJ \in \SetIntervalPartitions(B),\;\unstdJ^{\prime} \in \SetIntervalPartitions(B^{\prime}),\;\unstdJ \gluepartitions \unstdJ^{\prime} = \std(\unstdW)\}.
\end{align*}

\end{lemma}
\begin{proof}
We recall the notation used in \Cref{def:std_partition}
\begin{align*}
\unstdW&=\{\{w_{1},...,w_{\ell_{1}}\},\{w_{\ell_{1}+1},...,w_{\ell_{2}}\},...,\{w_{\ell_{k-1}+1},...,w_{\ell_{k}}\}\}\\
     \unstdZ&=\{\{z_{1},...,z_{\ell_{1}}\},\{z_{\ell_{1}+1},...,z_{\ell_{2}}\},...,\{z_{\ell_{k-1}+1},...,z_{\ell_{k}}\}\}\\ 
     \std(\unstdW)&=\std(\unstdZ)=  \{\{1,...,\ell_{1}\},\{\ell_{1}+1,...,\ell_{2} \},...,\{\ell_{k-1}+1,...,\ell_{k}\}\}.
\end{align*}
Consider the map
\begin{align*}
    f:\punion\unstdW &\to \punion\unstdZ\\
    w_{j} &\mapsto z_{j}
\end{align*}

Now, let $A \cup A^{\prime}=\punion\unstdW$, such that $\exists \unstdI\in \SetIntervalPartitions(A),\unstdI^{\prime}\in \SetIntervalPartitions(A^{\prime}): \unstdI \gluepartitions \unstdI^{\prime} = \unstdW$. We can write
\begin{align*}
    &\unstdI = \{\{w_{h_{1}},\dots,w_{h_{2}}\},\{w_{h_{2}+1},\dots,w_{h_{3}}\},\dots,\{w_{h_{l-1}+1},\dots,w_{h_{l}}\}\}\\
     &\unstdI^{\prime} =\{\{w_{s_{1}},\dots,w_{s_{2}}\},\{w_{s_{2}+1},\dots,w_{s_{3}}\},\dots,\{w_{s_{m-1}+1},\dots,w_{s_{m}}\}\}.
\end{align*}
where $(h_{i})$ and $(s_{i})$ are subsequences of $(\ell_{i})$.
Then
\begin{align*}
    &f(\unstdI):=\{\{f(w_{h_{1}}),\dots,f(w_{h_{2}})\},\{f(w_{h_{2}+1}),\dots,f(w_{h_{3}})\},\dots,\{f(w_{h_{l-1}+1}),\dots,f(w_{h_{l}})\}\}\\&=\{\{z_{h_{1}},\dots,z_{h_{2}}\},\{z_{h_{2}+1},\dots,z_{h_{3}}\},\dots,\{z_{h_{l-1}+1},\dots,z_{h_{l}}\}\},\\
    &f(\unstdI^{\prime}):=\{\{f(w_{s_{1}}),\dots,f(w_{s_{2}})\},\{f(w_{s_{2}+1}),\dots,f(w_{s_{3}})\},\dots,\{f(w_{s_{m-1}+1}),\dots,f(w_{s_{m}})\}\}\\&=\{\{z_{s_{1}},\dots,z_{s_{2}}\},\{z_{s_{2}+1},\dots,z_{s_{3}}\},\dots,\{z_{s_{m-1}+1},\dots,z_{s_{m}}\}\}.
\end{align*}
Now, we clearly have
\begin{align*}
   A\cup A^{\prime} = \punion\unstdW \iff  f(A)\cup f(A^{\prime}) = \punion\unstdZ
\end{align*}
and also
\begin{align*}
      &\unstdI \gluepartitions \unstdI^{\prime} = \unstdW
      \iff  f(\unstdI) \gluepartitions f(\unstdI^{\prime}) = \unstdZ.
\end{align*}
since
\begin{align*}
\{w_{h_{i-1}+1},\dots,w_{h_{i}}\} \cap \{w_{s_{j-1}+1},\dots,w_{s_{j}}\} \neq \emptyset \iff \{z_{h_{i-1}+1},\dots,z_{h_{i}}\} \cap \{z_{s_{j-1}+1},\dots,z_{s_{j}}\} \neq \emptyset.
\end{align*}
Moreover
\begin{align*}
&\std(\unstdI)=\std(\{\{w_{h_{1}},\dots,w_{h_{2}}\},\{w_{h_{2}+1},\dots,w_{h_{3}}\},\dots,\{w_{h_{l-1}+1},\dots,w_{h_{l}}\}\})\\
&=\std(\{\{z_{h_{1}},\dots,z_{h_{2}}\},\{z_{h_{2}+1},\dots,z_{h_{3}}\},\dots,\{z_{h_{l-1}+1},\dots,z_{h_{l}}\}\})=\std(f(\unstdI)).
\end{align*}

\end{proof}
\begin{proof}[Proof of \Cref{prop:coprod_partitions_coass}]
Let $\stds \in \SetStandardizedIntervalPartitions$. Since the gluing operation is associative, see \Cref{lemma:gluing_associative}, the expression
\begin{align*}
&\sum_{A \cup A^{\prime} \cup A^{\prime\prime} = \punion\stds }\;\sum_{\substack{\unstdI \in\SetIntervalPartitions(A),\unstdI^{\prime} \in\SetIntervalPartitions(A^{\prime}),\unstdI^{\prime\prime}\in\SetIntervalPartitions(A^{\prime\prime})\\\\\ \textcolor{blue}{\unstdI \gluepartitions \unstdI^{\prime} \gluepartitions \unstdI^{\prime\prime}}  = \stds}}\std(\unstdI)\otimes \std(\unstdI^{\prime})\otimes \std(\unstdI^{\prime\prime})
\end{align*}
is well-defined. 
We have
\begin{align*}
&\sum_{A \cup A^{\prime} \cup A^{\prime\prime} = \punion\stds }\;\sum_{\substack{\unstdI \in\SetIntervalPartitions(A),\unstdI^{\prime} \in\SetIntervalPartitions(A^{\prime}),\unstdI^{\prime\prime}\in\SetIntervalPartitions(A^{\prime\prime})\\\\\ \unstdI \gluepartitions \unstdI^{\prime} \gluepartitions \unstdI^{\prime\prime} = \stds}}\std(\unstdI)\otimes \std(\unstdI^{\prime})\otimes \std(\unstdI^{\prime\prime})\\&=\sum_{B \cup A^{\prime\prime} = \punion\stds }\sum_{\substack{\unstdJ \in\SetIntervalPartitions(B),\unstdI^{\prime\prime} \in\SetIntervalPartitions(A^{\prime\prime})\\\\\ \unstdJ \gluepartitions \unstdI^{\prime\prime} = \stds}}\Bigg(\sum_{A \cup A^{\prime} = B }\;\sum_{\substack{\unstdI \in\SetIntervalPartitions(A),\unstdI^{\prime} \in\SetIntervalPartitions(A^{\prime})\\\\\ \unstdI \gluepartitions \unstdI^{\prime} = \unstdJ}}\std(\unstdI)\otimes \std(\unstdI^{\prime})\Bigg)\otimes \std(\unstdI^{\prime\prime})\\&=^{\hspace{-0.7cm}\tiny \raisebox{1em}{\text{\Cref{lemma:coassociativity_intpart}}}}\sum_{B \cup A^{\prime\prime} = \punion\stds }\sum_{\substack{\unstdJ \in\SetIntervalPartitions(B),\unstdI^{\prime\prime} \in\SetIntervalPartitions(A^{\prime\prime})\\\\\ \unstdJ \gluepartitions \unstdI^{\prime\prime} = \stds}}\\&\Bigg(\sum_{X \cup Y = \cup_{i}\std(\unstdJ)_{i} }\;\sum_{\substack{\unstdI \in\SetIntervalPartitions(X),\unstdI^{\prime} \in\SetIntervalPartitions(Y)\\\\\ \unstdI \gluepartitions \unstdI^{\prime} = \std(\unstdJ)}}\std(\unstdI) \otimes \std(\unstdI^{\prime})\Bigg)\otimes \std(\unstdI^{\prime\prime})\\
&= (\coqspart \otimes \id) \circ \coqspart(\stds).
\end{align*}
Similarly, we have
\begin{align*}
&\sum_{A \cup A^{\prime} \cup A^{\prime\prime} = \punion\stds }\;\sum_{\substack{\unstdI \in\SetIntervalPartitions(A),\unstdI^{\prime} \in\SetIntervalPartitions(A^{\prime}),\unstdI^{\prime\prime}\in\SetIntervalPartitions(A^{\prime\prime})\\\\\ \unstdI \gluepartitions \unstdI^{\prime} \gluepartitions \unstdI^{\prime\prime} = \stds}}\std(\unstdI)\otimes \std(\unstdI^{\prime})\otimes \std(\unstdI^{\prime\prime})\\&=(\id \otimes \coqspart) \circ \coqspart(\stds).
\end{align*}
\end{proof}
\begin{remark}
The dual associative product is given by
\begin{align*}
    \stds\;\qspart\;\stdt &:=\sum_{\mathfrak{g} \in \SetStandardizedIntervalPartitions}\langle \stds \otimes \stdt , \coqspart(\stdg) \rangle \, \stdg\\&= \sum_{\mathfrak{g} \in \SetStandardizedIntervalPartitions}\sum_{A \cup A^{\prime} = \punion \stdg_{i}}\sum_{\substack{\unstdI \in \SetIntervalPartitions(A),\unstdI^{\prime} \in \SetIntervalPartitions(A^{\prime})\\\\ \std(\unstdI)=\stds,\; \std(\unstdI^{\prime})=\stdt  \\\\\unstdI \gluepartitions \unstdI^{\prime} = \stdg }} \stdg\\&= \sum_{n \in \N}\sum_{A \cup A^{\prime} = [n]}\sum_{\substack{\unstdI \in \SetIntervalPartitions(A),\unstdI^{\prime} \in \SetIntervalPartitions(A^{\prime})\\\\ \std(\unstdI)=\stds,\; \std(\unstdI^{\prime})=\stdt  }}\unstdI \gluepartitions \unstdI^{\prime}.
\end{align*}
\end{remark}
\begin{example}
\label{ex:product_partition}
\begin{align*}
  \adjacentSquares{2} \; \qspart \; \adjacentSquares{2} &= 2\;\adjacentSquares{2}\;\adjacentSquares{2} + 2\;\adjacentSquares{3} +\adjacentSquares{2}
\end{align*}
\end{example}
\begin{remark}[Section coefficients]
\label{rem:section_coefficients}
We can write
\begin{align*}
    \langle \stds \otimes \stdt ,\coqspart(\stdg) \rangle&=\#\{A,A^{\prime} \subset \punion \stdg_{j}\;|A\cup A^{\prime} = \punion \stdg,\; \exists!\, \unstdI \in  \SetIntervalPartitions(A), \exists! \,\unstdI^{\prime} \in  \SetIntervalPartitions(A^{\prime}),\\
    &\qquad\qquad\qquad\qquad\qquad\std(\unstdI)=\stds,\std(\unstdI^{\prime})=\stdt,\unstdI \gluepartitions \unstdI^{\prime} = \stdg
    \}.
\end{align*}
\end{remark}

\subsubsection{Bialgebras on interval partitions}
We now show that the algebraic operations are compatible. The unit and the counit maps are given by \begin{align*}
    &u: \Q \to \FreeIntPart\\
    &u: 1 \mapsto \e \\ 
&\varepsilon : \FreeIntPart \to \Q\\
    &\varepsilon(x) = \begin{cases}
    x\;\text{if}\;x \in \Q[\SetStandardizedIntervalPartitions_{0}]\\
    0\;\text{else}\end{cases}
\end{align*}
\begin{theorem}[Bialgebras on interval partitions]
\label{thm:Quasi-shuffle_deconc_HA_part}

The following holds:
\begin{itemize}
    \item[--] $(\FreeIntPart,\conc, \coqspart,u,\varepsilon)$, is a bialgebra.
    \item[--]  $(\FreeIntPart,\qspart,\deconc,u,\varepsilon)$, is a connected filtered Hopf algebra, according to both \Cref{eq:n_grading_partitions} and \Cref{eq:blocks_grading_partitions}.
\end{itemize}
\end{theorem}

\begin{lemma}
\label{lemma:bialgebra}
Let $A,B \in \N_{\ge 1},\;|A|,|B|< \infty$, $\stds,\stdt \in \SetStandardizedIntervalPartitions$, and $\unstdI \in \SetIntervalPartitions(A), \unstdI^{\prime} \in \SetIntervalPartitions(B)$. Assume
\begin{align*}
    \unstdI \gluepartitions \unstdI^{\prime} &= \stds \conc \stdt.
\end{align*}
Recall that
\begin{align*}
& \stds \conc \stdt := \{\stds_{1},...,\stds_{m},\stdt^{\prime}_{1},...,\stdt^{\prime}_{n}\}
\end{align*}
where for $i=1,...,n$, $\stdt^{\prime}_{i}:=|\punion\stds|+\stdt_{i}$.
Then
\begin{enumerate}
  \item 
\begin{align*}
\unstdI &= \unstdI(A\cap \punion\stds) \uplus \unstdI(A\cap \bigcup\stdt^{\prime}_{j}),\\
\unstdI^{\prime} &= \unstdI^{\prime}(B\cap \punion\stds) \uplus \unstdI^{\prime}(B\cap \bigcup\stdt^{\prime}_{j}),
\end{align*}

\item
\begin{align*}
    \unstdI(A\cap \punion\stds) \gluepartitions \unstdI^{\prime}(B\cap \bigcup\stds_{j})&=\stds,\\ \unstdI(A\cap \bigcup\stdt^{\prime}_{j}) \gluepartitions \unstdI^{\prime}(B\cap \bigcup\stdt^{\prime}_{j})&=\stdt^{\prime}. 
\end{align*}

\item
\begin{align*}
    \std(\unstdI)&= \std(\unstdI(A\cap \punion\stds))\conc \std(\unstdI(A\cap \punion\stdt^{\prime})),
    \\\std(\unstdI^{\prime})&= \std(\unstdI^{\prime}(B\cap \punion\stds))\conc \std(\unstdI^{\prime}(B\cap \punion\stdt^{\prime})).
\end{align*}
\end{enumerate}
\end{lemma}
\begin{proof}

First, we have
\begin{align*}
    \unstdI &= \unstdI(A)=\unstdI(A\cap \punion\stds \uplus A\cap \bigcup\stdt^{\prime}_{j})\\&=\{\unstdI_{j} \cap c|j \in J, c \in \cliques(A\cap \punion\stds \uplus A\cap \bigcup\stdt^{\prime}_{j})\}\setminus \{\emptyset\}\\&=\{\unstdI_{j} \cap c|j \in J, c \in \cliques(A\cap \punion\stds) \uplus \cliques(A\cap \bigcup\stdt^{\prime}_{j})\}\setminus \{\emptyset\}\\&=\{\unstdI_{j} \cap c|j \in J, c \in \cliques(A\cap \punion\stds) \}\setminus \{\emptyset\} \\&\uplus \{\unstdI_{j} \cap c|j \in J, c \in \cliques(A\cap \punion\stdt^{\prime}) \}\setminus \{\emptyset\}\\&=\unstdI(A\cap \punion\stds) \uplus \unstdI(A\cap \bigcup\stdt^{\prime}_{j}).
\end{align*}
Analogously one shows that $\unstdI^{\prime} = \unstdI^{\prime}(B\cap \punion\stds) \uplus \unstdI^{\prime}(B\cap \bigcup\stdt^{\prime}_{j})$.
Therefore
\begin{align*}
\unstdI(A\cap \punion\stds) \uplus \unstdI(A\cap \punion\stdt^{\prime})
 \gluepartitions (\unstdI(B\cap \punion\stds) \uplus \unstdI(B\cap \punion\stdt^{\prime}) = \stds \conc \stdt.
\end{align*}
which holds if and only if
\begin{align*}
\unstdI(A\cap \punion\stds) 
 \gluepartitions \unstdI^{\prime}(B\cap \punion\stds) =\stds,
\end{align*}
and
\begin{align*}
\unstdI(A\cap \punion\stdt^{\prime}) 
 \gluepartitions \unstdI^{\prime}(B\cap \punion\stdt^{\prime}) =  \stdt^{\prime},
\end{align*}
Finally, since we can write
\begin{align*}
\unstdI&= \{I_{s,1},I_{s,m},I_{t^{\prime},1},...,I_{t^{\prime},n}\}
\end{align*}
with
\begin{align*}
    &\unstdI(A\cap \punion\stds) = \{I_{s,1},I_{s,m}\},
  \\&\unstdI(A\cap \punion\stdt^{\prime}) = \{I_{t^{\prime},1},...,I_{t^{\prime},n}\}
\end{align*}
we have
\begin{align*}
  \std(\unstdI)&= \std(\{I_{s,1},...,I_{s,m}\}) \uplus \left(|\bigcup_{j=1}^{m}I_{s,j}|+\std(\{I_{t,1},...,I_{t,n}\})\right)\\&\std(\unstdI(A\cap \punion\stds))\conc \std(\unstdI(A\cap \punion\stdt^{\prime})).
\end{align*}
And analogously one shows that $\std(\unstdI^{\prime}(B\cap \punion\stds))\conc \std(\unstdI^{\prime}(B\cap \punion\stdt^{\prime}))$
\end{proof}

\begin{proof}
We show that $\coqspart(\stds \conc \stdt) =   \coqspart(\stds)  \conc   \coqspart(\stdt)$.
We have
\begin{align*}
    \coqspart(\stds \conc \stdt)&=\sum_{A \cup A^{\prime} = \punion\left( \stds \conc \stdt \right)}\;\sum_{\substack{\unstdI \in \SetIntervalPartitions(A),\unstdI \in \SetIntervalPartitions(A^{\prime})\\\\\unstdI \gluepartitions  \unstdI^{\prime}= \stds \conc \stdt}}\std(\unstdI) \otimes \std(\unstdI^{\prime}).
\end{align*}
Thanks to \Cref{lemma:bialgebra}, we can write
\begin{align*}
&\coqspart(\stds \conc \stdt)\\
&=\sum_{\substack{(A \cap \punion\stds)\cup (A^{\prime} \cap \punion\stds) =\punion\stds\\\\(A \cap \punion\stdt^{\prime})\cup (A^{\prime} \cap \punion\stdt^{\prime}) =\punion\stdt^{\prime}}}\sum_{\substack{\unstdI \in \SetIntervalPartitions(A),\,\unstdI^{\prime} \in \SetIntervalPartitions(A^{\prime})\\\\ \unstdI(A \cap \punion\stds) \gluepartitions \unstdI^{\prime}(A^{\prime} \cap \punion\stds) =\stds\\\\ \unstdI(A \cap \punion\stdt^{\prime})  \gluepartitions  \unstdI(A^{\prime} \cap \punion\stdt^{\prime}) =\stdt^{\prime}}}\std(\unstdI(A \cap \punion\stds)) \conc \std(\unstdI(A \cap \punion\stdt^{\prime})) \\\\
&
\qquad\qquad\qquad\qquad\qquad\qquad
\qquad\qquad\qquad\qquad\qquad\qquad\quad
\otimes \std(\unstdI^{\prime}(A^{\prime} \cap \punion\stds) \conc \std(\unstdI^{\prime}(A^{\prime} \cap \punion\stdt^{\prime})) 
   \\\\
&=
  \sum_{\substack{A_{1} \cup A_{1}^{\prime} = \punion \stds\\\\ A_{2} \cup A_{2}^{\prime} = \punion \stdt^{\prime} }}\;
  \sum_{\substack{\unstdI_{1} \in\SetIntervalPartitions(A_{1}), \unstdI_{1}^{\prime} \in\SetIntervalPartitions(A_{1}^{\prime})\\\\\unstdI_{2} \in\SetIntervalPartitions(A_{2}), \unstdI_{2}^{\prime} \in\SetIntervalPartitions(A_{2}^{\prime})\\\\ \unstdI_{1}\gluepartitions  \unstdI_{1}^{\prime}= \stds,\;\unstdI_{2}\gluepartitions  \unstdI_{2}^{\prime}= \stdt^{\prime}}}\std(\unstdI_{1}) \conc \std(\unstdI_{2}) \otimes \std(\unstdI_{1}^{\prime}) \conc \std(\unstdI_{2}^{\prime})\\
&=\sum_{A_{1} \cup A_{1}^{\prime} = \punion \stds}\;
\sum_{\substack{\unstdI_{1} \in\SetIntervalPartitions(A_{1}) \\ \unstdI_{1}^{\prime} \in\SetIntervalPartitions(A_{1}^{\prime})\\ \unstdI_{1}\gluepartitions  \unstdI_{1}^{\prime}= \stds}}\std(\unstdI_{1}) \otimes \std(\unstdI_{1}^{\prime})
\conc 
\sum_{A_{2} \cup A_{2}^{\prime} = \punion \stdt^{\prime}}\;
\sum_{\substack{\unstdI_{2} \in\SetIntervalPartitions(A_{2})\\ \unstdI_{2}^{\prime} \in\SetIntervalPartitions(A_{2}^{\prime})\\ \unstdI_{2}\gluepartitions  \unstdI_{2}^{\prime}= \stdt^{\prime}}}\std(\unstdI_{2}) \otimes \std(\unstdI_{2}^{\prime}) 
\end{align*}
Using \Cref{lemma:coassociativity_intpart}, we have
\begin{align*}
    &\sum_{A_{2} \cup A_{2}^{\prime} = \punion \stdt^{\prime}}\;\sum_{\substack{\unstdI_{2} \in\SetIntervalPartitions(A_{2}), \unstdI_{2}^{\prime} \in\SetIntervalPartitions(A_{2}^{\prime})\\\\ \unstdI_{2}\gluepartitions  \unstdI_{2}^{\prime}= \stdt^{\prime}}}\std(\unstdI_{2}) \otimes \std(\unstdI_{2}^{\prime})\\&=\sum_{B \cup B^{\prime} = \punion \stdt}\;\sum_{\substack{\unstdJ \in\SetIntervalPartitions(B), \unstdJ^{\prime} \in\SetIntervalPartitions(B^{\prime})\\\\ \unstdJ\gluepartitions  \unstdJ^{\prime}= \stdt}}\std(\unstdJ)\otimes \std(\unstdJ^{\prime}). 
\end{align*}

Therefore
\begin{align*}
&\coqspart(\stds \conc \stdt)\\
&=\coqspart(\stds) \conc \coqspart(\stdt)\\
&=\sum_{A \cup A^{\prime} = \punion \stds}\;\sum_{\substack{\unstdI \in\SetIntervalPartitions(A), \unstdI^{\prime} \in\SetIntervalPartitions(A^{\prime})\\\\ \unstdI\gluepartitions  \unstdI^{\prime}= \stds}}\std(\unstdI)\otimes \std(\unstdI^{\prime})\\
&\qquad \conc\sum_{B \cup B^{\prime} = \punion \stdt}\;\sum_{\substack{\unstdJ \in\SetIntervalPartitions(B), \unstdJ^{\prime} \in\SetIntervalPartitions(B^{\prime})\\\\ \unstdJ\gluepartitions  \unstdJ^{\prime}= \stdt}}\std(\unstdJ)\otimes \std(\unstdJ^{\prime}).
\end{align*}
Dually, it follows  that $$(\FreeIntPart,\qspart,\deconc,u,\varepsilon)$$ is also a bialgebra. We recall the well-known fact that any connected filtered bialgebra is automatically a Hopf algebra, see \cite{manchon2008hopf}[Corollary 5].
\end{proof}

The following remark is about how, under certain assumptions, one can learn about the structure of an algebra by endowing it with a Hopf algebra structure.

\begin{remark}[Polynomial algebra]
\label{rem:polynomial_algebra}
We recall the well-known result that any filtered Hopf algebra is isomorphic, as an algebra, to a polynomial algebra. See \cite{cartier2021classical}[Theorem 4.4.1]. Therefore, $(\FreeIntPart,\qspart)$ is a free commutative algebra.
\end{remark}

\begin{remark}
\label{rem:words_intergers}
We provide another description of the Hopf algebra from \Cref{thm:Quasi-shuffle_deconc_HA_part}, by showing that it is isomorphic to a Hopf algebra on words, where the letters are positive integers.
Let $V := \Q[\N_{\ge 1}]$, then
    \begin{align*}
        \Phi: \bigoplus_{n\in \N} V^{\otimes n} \to (\FreeIntPart, \conc),
    \end{align*}
    defined as the unique extension of
    \begin{align*}
        \Phi(n) := \{\{1,2,...,n\}\}
    \end{align*}
    to a map of associative algebras, is a linear isomorphism.
We denote with $\conctens$ the concatenation of tensors. If we restrict $\Phi$ to words on positive integers, which form a linear basis, we obtain an injection onto a basis for the codomain, the set $\SetStandardizedIntervalPartitions$. Indeed, we have:
  \begin{align*}
        \Phi(n_{1} \conctens n_{2} \conctens \;\cdots \conctens \;n_{k}) &= \Phi(n_{1}) \conc \Phi(n_{2}) \conc \cdots \conc  \Phi(n_{k}) \\&=
        \left\{\{1,2,...,n_{1}\},\{n_{1}+1,\dots,n_{1}+n_{2}\},\dots,\{\sum_{j=1}^{k-1}n_{j}+1,...,\sum_{j=1}^{k}n_{j}\} \right\}.
    \end{align*}
As an example
\begin{align*}
 \Phi(1 \conctens 3 \conctens 2) = \left\{\{1\},\{2,3,4\},\{5,6\} \right\}.
\end{align*}
We can use $\Phi$ to define a product on words, which is just a relabeling of $\qspart$. We define on basis elements
    \begin{align*}
       m_{1}\;\conctens\;m_{2}\;\conctens\;\cdots \;\conctens\; m_{\ell}\;&\qswrd\; n_{1}\;\conctens\; n_{2} \;\conctens\; \cdots \;\conctens\; n_{k} \\
                                                                          &:=  \Phi^{-1}\left(\Phi( m_{1}\;\conctens\; m_{2}\;\conctens\;\cdots \;\conctens\; m_{\ell})\; \qspart \;\Phi(n_{1}\;\conctens\; n_{2} \;\conctens\; \cdots \;\conctens\; n_{k})\right).
\end{align*}
 
For example
 \begin{align*}
       2\;\qswrd\;  2&:=  \Phi^{-1}(\Phi(2)\;\qspart \;\Phi(2))\\
       &=  \Phi^{-1}( \ \adjacentSquares{2} \ \qspart \ \adjacentSquares{2} \ )\\
          &=  \Phi^{-1}(\ \textcolor{red}{2}\ \adjacentSquares{2} \ \adjacentSquares{2} + \textcolor{red}{2}\ \adjacentSquares{3} + \adjacentSquares{2} \ )\\
          &=  \textcolor{red}{2} \, 2 \conctens 2 + \textcolor{red}{2}\,3 +2,\\
1 \conctens  2 \; \qswrd \;  1 \conctens 1 &=  \textcolor{red}{3} \, 1 \conctens 2 \conctens 1 +\textcolor{red}{3}\,  1\conctens  2 + 1\conctens 2 \conctens 1\conctens 1\\&+\textcolor{red}{6} \, 1\conctens  1 \conctens 2 +\textcolor{red}{3}\, 1 \conctens 1 \conctens 1 \conctens 2 +\textcolor{red}{2} 1 \conctens 1 \conctens 2 \conctens 1.
    \end{align*}
Then, in analogy to \Cref{thm:Quasi-shuffle_deconc_HA_part}
 \begin{itemize}
     \item[--]$(\bigoplus_{n\in \N} V^{\otimes n},\conctens,\Delta_{\qswrd})$  is a bialgebra,
     \item[--] $(\bigoplus_{n\in \N} V^{\otimes n},\qswrd,\Delta_{\conctens})$ is a connected filtered (by the length of the words, see \Cref{eq:blocks_grading_partitions}) Hopf algebra.
\end{itemize}

\end{remark}

\subsection{Signature of an interval partition}
 We define a family of linear maps on the dual of $\FreeIntPart$, indexed by $\stdL \in \SetStandardizedIntervalPartitions$. These maps, encode \quotationmarks{occurrences} of other standardized interval partitions in $\stdL$. 
\begin{definition}[Signature of an interval partition]
  \label{def:signature_interval_partition}
  For $\stdL \in \SetStandardizedIntervalPartitions$, define the linear map
  $\IPC(\stdL): \FreeIntPart \to \Q$ via
  \begin{align*}
       \Big\langle  \IPC(\stdL), \stds \Big\rangle := \# \{A \subset \punion\stdL|\; \std(\stdL(A)) \ge \stds\}, \quad \stds \in \SetStandardizedIntervalPartitions.
  \end{align*}
\end{definition}
\begin{remark}
The condition $\std(\stdL(A)) \ge \stds$ tells us that we count $\stds$, each time we find a pattern that is equal to $\stds$ or coarser. Let $\stdL = \{\{1\},\{2\},\{3,4\}\}$, and $\stds = \{\{1\},\{2\}\}$. 
\begin{align*}
   \text{For}\;A = \{3,4\},\;\std(\stdL(A)) &=\{\{1,2\}\} \ge \{\{1\},\{2\}\}.
\end{align*}
So that if we look for a fine pattern (arbitrary gaps) we also count coarse pattern (consecutive values), while if we look for a coarse pattern we discard arbitrary gaps and only keep consecutive values. Let $\stdL = \{\{1\},\{2\},\{3,4\}\}$, and $\stds = \{\{1,2\},\{3\}\}$.
\begin{align*}
   \text{For}\;A = \{2,3,4\},\;\std(\stdL(A)) &= \{\{1\},\{2,3\}\} \perp \{\{1,2\},\{3\}\},
\end{align*}
 i.e., the structure of the gaps needs to be respected. 
\end{remark}

\begin{remark}
\label{rem:explicit}
   When we count the occurrences of an interval partition, $\stds= \singleblockn_{1} \conc \cdots \conc \singleblockn_{k}  $ (conveniently written taking the product $\conc$ of single blocks partitions, $\singleblockn_{j} :=\{[n_{j}]\},n_{j}\in \N_{\ge 1}$) in a single block partition, i.e. $\stdL=\singleblockN:=\{[N]\},N\in \N_{\ge 1}$, we have
\begin{align}
\label{eq:single_block_count}
     \Big\langle  \IPC(\stdL), \stds \Big\rangle = \binom{N- (n_{1}+...+n_{k})+k}{k}.
\end{align}
which are the weak $k+1$ compositions of $N- (n_{1}+...+n_{k})$. For example, if we have $N=9, n_{1}=3, n_{2}=5$, the three possible constellations
\begin{center}
    \includegraphics[scale=0.2]{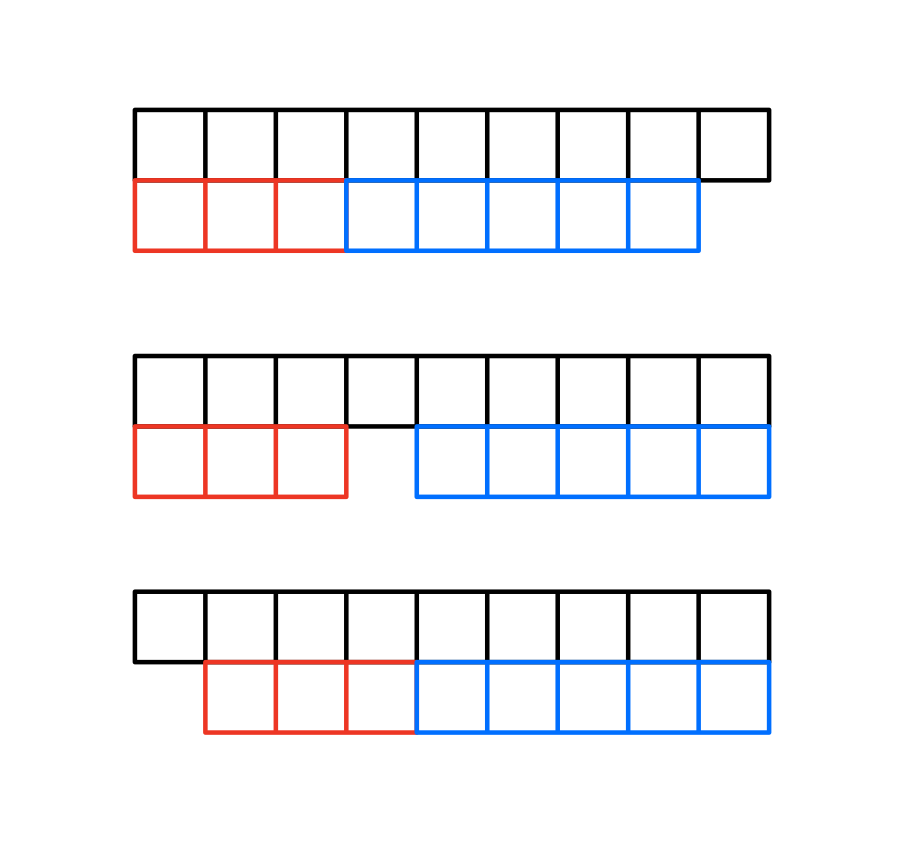}
\end{center}
correspond to $(1,0,0),(0,1,0),(0,0,1)$ respectively: i.e. the weak $3$ compositions of $1$. 
\end{remark}

\subsubsection{Character property and Chen's identity}
We call these maps \textit{signatures} for interval partitions, motivated by the following result.

\begin{theorem}[Character property]~
\label{thm:qsI_intpart}
 The maps $(\IPC(\stdL))_{\stdL}$ are characters under the product $\qspart$. Let $\stdL \in \SetStandardizedIntervalPartitions_N$, with $N \in \N$. For all $\stds,\stdt \in \SetStandardizedIntervalPartitions$, we have
    \begin{align}
    \label{eq:qs_id_part}
       \Big\langle  \IPC(\stdL), \stds \Big\rangle
       \cdot
       \Big\langle  \IPC(\stdL), \stdt \Big\rangle
       =
       \Big\langle  \IPC(\stdL),\, \stds \; \qspart \; \stdt \Big\rangle.
    \end{align}
\end{theorem}
\begin{remark}[Single blocks]
Let $\singleblockN:=\{[N]\}$, $\singleblockm :=\{[m]\}$, $\singleblockn:=\{[n]\}$. Then we have 
\begin{align*} 
\Big\langle  \IPC(\singleblockN), \singleblockm \Big\rangle
       \cdot
       \Big\langle  \IPC(\singleblockN), \singleblockn \Big\rangle&=(N-m+1)(N-n+1)
\end{align*}
thanks to \Cref{rem:explicit}. Since
\begin{align*}
&\singleblockm \; \qspart \; \singleblockn  \\&=\singleblockm \conc \singleblockn+ \singleblockn \conc \singleblockm+2\sum_{i=1}^{\min(m,n)-1}\{[m+n-i]\}\\&+ (\max(m,n)-\min(m,n)+1)\{[\max(m,n)]\}
\end{align*}
using again \Cref{rem:explicit}, we obtain

\begin{align*}
&\Big\langle  \IPC(\singleblockN),\, \singleblockm \; \qspart \; \singleblockn \Big\rangle  \\& =2\binom{N-(m+n)+2}{2} + 2\sum_{i=1}^{\min(m,n)-1}(N-(m+n-i)+1)\\&+ (\max(m,n)-\min(m,n)+1)(N-\max(m,n)+1) 
\end{align*}
It is easy to show that 
\begin{align*}
   \Big\langle  \IPC(\singleblockN), \singleblockm \Big\rangle
       \cdot
       \Big\langle  \IPC(\singleblockN), \singleblockn \Big\rangle &= \Big\langle  \IPC(\singleblockN),\, \singleblockm \; \qspart \; \singleblockn \Big\rangle 
\end{align*}
\end{remark}
We now state several lemmas needed for the proof of \Cref{thm:qsI_intpart}.

\begin{lemma}[$\stds$-refinement of $\unstdI$]
\label{lemma:character0}
Let $\unstdI \in \SetIntervalPartitions$ and  $\stds\in \SetStandardizedIntervalPartitions$. If
\begin{align*}
    \std(\unstdI) \ge \stds,
\end{align*}
then there exists a unique refinement $\mathcal I_\stds$ of $\mathcal I$ such that
\begin{align*}
      \std(\unstdI_{\stds}) = \stds.
\end{align*}
\end{lemma}
\begin{example}
Let $\unstdI =\{\{5\},\{6\},\{10,11,12\}\}$ and $\stds =\{\{1\},\{2\},\{3,4\},\{5\}\}$. Then:
\begin{align*}
    \std(\unstdI) \ge \stds
\end{align*}
and 
\begin{align*}
      \unstdI_{\stds}= \{\{5\},\{6\},\{10,11\},\{12\}\}
\end{align*}
\end{example}

\begin{proof}
Let
\begin{align*}
    \std(\unstdI) \ge \stds.
\end{align*} 

We recall the notation used in \Cref{def:std_partition}. Let $i_{1},\dots,i_{\ell_{1}},\dots,i_{\ell_{2}},\dots,i_{\ell_{k}} \in \N_{\ge 1}$ such that 
\begin{align*}
    i_{1} \prec \cdots \prec i_{\ell_{1}} < i_{\ell_{1}+1}\prec \cdots \prec i_{\ell_{2}} < \cdots < i_{\ell_{k-1}+1} \prec \cdots \prec i_{\ell_{k}},
\end{align*}
where $i \prec j \iff j-i = 1$.
Then
\begin{align*}
 &\std(\unstdI):=\std(\{\{i_{1},\dots,i_{\ell_{1}}\},\{i_{\ell_{1}+1},\dots,i_{\ell_{2}}\},\dots,\{i_{\ell_{k-1}+1},\dots,i_{\ell_{k}}\}\})\\ &:=  \{\{1,\dots,\ell_{1}\},\{\ell_{1}+1,\dots,\ell_{2} \},\dots,\{\ell_{k-1}+1,\dots,\ell_{k}\}\}\\&\ge  \{\{1 \textcolor{blue}{\vert},\dots,\textcolor{blue}{\vert}\ell_{1}\},\{\ell_{1}+1\textcolor{blue}{\vert},\dots,\textcolor{blue}{\vert}\ell_{2}\},\dots,\{\ell_{k-1}+1\textcolor{blue}{\vert},\dots,\textcolor{blue}{\vert} \ell_{k}\}\} = \stds.
\end{align*}  
where
\begin{align*}
1 \prec \cdots \prec \ell_{1} \prec  \ell_{1}+1 \prec \cdots \prec \ell_{2} < \cdots < \ell_{k-1}+1 \prec \cdots \prec \ell_{k}
\end{align*}
and the $\textcolor{blue}{\vert}$ indicate potential \quotationmarks{cuts}, which lead to a different $\stds \in \SetStandardizedIntervalPartitions$. Clearly only
\begin{align*}
    \{\{i_{1} \textcolor{blue}{\vert},\dots,\textcolor{blue}{\vert}i_{\ell_{1}}\},\{i_{\ell_{1}+1}\textcolor{blue}{\vert},\dots,\textcolor{blue}{\vert}i_{\ell_{2}}\},\dots,\{i_{\ell_{k-1}+1}\textcolor{blue}{\vert},\dots,\textcolor{blue}{\vert} i_{\ell_{k}}\}\}
\end{align*}
does the job.
\end{proof}

\begin{lemma}
\label{lemma:character1}
Let $\stds,\stdt,\stdL \in \SetStandardizedIntervalPartitions$. Then the following holds
\begin{align*}
    &\Big\{A,B \subset \punion \stdL\;| \std(\stdL(A)) \ge \stds, \std(\stdL(B)) \ge \stdt\Big\} \\
    &\quad
    =\biguplus_{\substack{(\stdg,C) \in \SetStandardizedIntervalPartitions \times \punion \stdL \\\\ \std(\stdL(C))\ge \stdg}}
    \Big\{A,B \subset \punion \stdL\;|A\cup B = C, \std(\stdL(A)) \ge \stds, \std(\stdL(B)) \ge \stdt,\\
    &\qquad\qquad\qquad\qquad\qquad \std(\stdL(A)_{\stds}\gluepartitions \stdL(B)_{\stdt}) = \stdg \Big\}.
\end{align*}
where the union is taken over disjoint sets (possibly empty). 
\end{lemma}

\begin{proof}
    Let $(A,B)$ be an element of the LHS. Then there exists a unique pair $(\stdg,C)$ such that it belongs to one of the sets on the RHS. We have $C:=A \cup B$ and $\stdg := \std(\stdL(A)_{\stds} \gluepartitions \stdL(B)_{\stdt}).$ The inequality $\std(\stdL(C)) \ge \stdg$ holds since
    \begin{align*}
      \std(\stdL(C)) \ge \std(\stdL(A) \gluepartitions \stdL(B))  \ge \std(\stdL(A)_{\stds} \gluepartitions \stdL(B)_{\stdt}) = \stdg.
    \end{align*}
From left to right: the first inequality is due to \Cref{corollary:gluing_finess}, while the second is due to \Cref{lemma:gluing_finess}.

The other inclusion is trivial.
\end{proof}
\begin{lemma}[Zooming in]
\label{lemma:character2}
Let $\stds,\stdt,\stdL \in \SetStandardizedIntervalPartitions$, $(\stdg,C) \in \SetStandardizedIntervalPartitions \times \punion \stdL$. Assume that $\std(\stdL(C)) =: \stdh \ge \stdg$. Then
\begin{gather*}
    \#\left\{\left.A,A^{\prime} \subset \punion \stdL\;\right|A\cup A^{\prime} = C, \std(\stdL(A)) \ge \stds, \std(\stdL(A^{\prime})) \ge \stdt, \std(\stdL(A)_{\stds}\gluepartitions \stdL(A^{\prime})_{\stdt}) = \stdg \right\} \\
    = \\
    \#\left\{\left. A,A^{\prime} \subset \punion \stdh\;\right|A\cup A^{\prime} = \punion \stdh,\; \std(\stdh(A)) \ge \stds, \std(\stdh(A^{\prime})) \ge \stdt, \stdh(A)_{\stds}\gluepartitions \stdh(A^{\prime})_{\stdt} = \stdg \right\}.
\end{gather*}
\end{lemma}
\begin{proof}
Immediate.
\end{proof}
\begin{lemma}[Invariance]
\label{lemma:character3}
Let $\stdg,\stdh\in \SetStandardizedIntervalPartitions$ such that $\stdh \ge \stdg$. Then
\begin{gather*}
    \Big\{A,A^{\prime} \subset \punion \stdh\;|A\cup A^{\prime} = \punion \stdh,\; \std(\stdh(A)) \ge \stds, \std(\stdh(A^{\prime})) \ge \stdt, \stdh(A)_{\stds}\gluepartitions \stdh(A^{\prime})_{\stdt} = \stdg \Big\}\\
        = \\
    \{A,A^{\prime} \subset \punion \stdg\;|A\cup A^{\prime} = \punion \stdg,\; \std(\stdg(A)) \ge \stds, \std(\stdg(A^{\prime})) \ge \stdt, \stdg(A)_{\stds}\gluepartitions \stdg(A^{\prime})_{\stdt} =\stdg\}.
\end{gather*}
\end{lemma}
\begin{proof}
  Denote the set on the left-hand side with $\textup{Set}_{0}$
  and the one on the right-hand side with $\textup{Set}_{1}$.
Let $(A,A^{\prime})\in \textup{Set}_{1}$. Since $\stdh \ge \stdg$ means that $\forall i \in I: \exists j \in J: \stdg_{i} \subset \stdh_{j}$ and since
\begin{align*}
    \stdh(A):=\{c \cap \stdh_{j}|c \in \cliques(A),j \in J\}\setminus\{\emptyset\}\\
    \stdg(A):=\{c \cap \stdg_{j}|c \in \cliques(A),j \in J\}\setminus\{\emptyset\}
\end{align*}
we have $\stdh(A) \ge \stdg(A)$ and, analogously, $\stdh(A^{\prime})\ge \stdg(A^{\prime})$. This means that $ \std(\stdh(A)) \ge \std(\stdg(A)) \ge \stds$ and  $ \std(\stdh(A^{\prime})) \ge \std(\stdg(A^{\prime})) \ge \stdt$. Therefore $\stdh(A)$ and $\stdh(A^{\prime})$ can be refined so that $\stdh(A)_{\stds}\gluepartitions \stdh(A^{\prime})_{\stdt} = \stdg(A)_{\stds}\gluepartitions \stdg(A^{\prime})_{\stdt} = \stdg$, see \Cref{lemma:character0}. So, $(A,A^{\prime})\in \textup{Set}_{0}$. For the other direction, let $(A,A^{\prime})\in \textup{Set}_{0}$. Since $\stdh(A)_{\stds}\gluepartitions \stdh(A^{\prime})_{\stdt} = \stdg$, using \Cref{lemma:gluing_finess3}, we have\begin{align*}
    &\stdg(A) \ge \stdh(A)_{\stds} \\
     &\stdg(A^{\prime})\ge \stdh(A^{\prime})_{\stdt}.
\end{align*}
Then it is immediate to verify that $(A,A^{\prime})\in \textup{Set}_{1}$.
\end{proof}

\begin{lemma}[Recovering section coefficients]
\label{lemma:character4}
Let $\stds,\stdt,\stdg \in \SetStandardizedIntervalPartitions$. Then:
\begin{gather*}
   \left\{ A,A^{\prime} \subset \punion \stdg\;\Big|A\cup A^{\prime} = \punion \stdg,\; \std(\stdg(A)) \ge \stds, \std(\stdg(A^{\prime})) \ge \stdt, \stdg(A)_{\stds}\gluepartitions \stdg(A^{\prime})_{\stdt} =\stdg \right\} \\
    = \\
\left\{ A,A^{\prime} \subset \punion \stdg\;\Big|A\cup A^{\prime} = \punion \stdg,\; \exists!\, \unstdI \in  \SetIntervalPartitions(A), \exists! \,\unstdI^{\prime} \in  \SetIntervalPartitions(A^{\prime}), \std(\unstdI)=\stds,\std(\unstdI^{\prime})=\stdt,\unstdI \gluepartitions \unstdI^{\prime} = \stdg \right\}.
\end{gather*}
In particular, see \Cref{rem:section_coefficients}
\begin{align*}
   &\#\{A,A^{\prime} \subset \punion \stdg\;|A\cup A^{\prime} = \punion \stdg,\; \std(\stdg(A)) \ge \stds, \std(\stdg(A^{\prime})) \ge \stdt, \stdg(A)_{\stds}\gluepartitions \stdg(A^{\prime})_{\stdt} =\stdg\} \\&=
   \langle \stds \otimes \stdt, \coqspart(\stdg) \rangle.
\end{align*}
\end{lemma}  
\begin{proof}
 Denote the set on the left-hand side with $\textup{Set}_{0}$
  and the one on the right-hand side with $\textup{Set}_{1}$.
Let $(A,A^{\prime}) \in \text{Set}_{0}$, then $A\cup A^{\prime} = \punion \stdg$, $\stdg(A)_{\stds} \in \SetIntervalPartitions(A)$ and $\stdg(A^{\prime})_{\stdt} \in \SetIntervalPartitions(A^{\prime})$, $\std(\stdg(A)_{\stds}) =\stds$, $\std(\stdg(A^{\prime})_{\stdt}) =\stdt$, $\stdg(A)_{\stds} \gluepartitions \stdg(A^{\prime})_{\stdt} = \stdg$. Let $(A,A^{\prime}) \in \text{Set}_{1}$, then $A\cup A^{\prime} = \punion \stdg$. Since $\unstdI \gluepartitions \unstdI^{\prime} = \stdg$, with $\unstdI \in \SetIntervalPartitions(A)$ and $\unstdI^{\prime} \in \SetIntervalPartitions(A^{\prime})$, we have $\unstdI \le \stdg(A)$, $\unstdI^{\prime} \le \stdg(A^{\prime})$, and $\stds=\std(\unstdI) \le \std(\stdg(A))$, $\stdt= \std(\unstdI^{\prime}) \le \std(\stdg(A^{\prime}))$ as well. Therefore, using \Cref{lemma:character0}, we have $\unstdI = \stdg(A)_{\stds}$ and $\unstdI^{\prime} = \stdg(A)_{\stdt}$.

\end{proof}
\begin{proof}[Proof of \Cref{thm:qsI_intpart}]

\begin{align*}
&\Big\langle  \IPC(\stdL), \stds\Big\rangle \Big\langle  \IPC(\stdL), \stdt \Big\rangle
\\&=|\{A,B \subset \punion \stdL\;| \std(\stdL(A)) \ge \stds, \std(\stdL(B)) \ge \stdt \}|\\
    &=| \biguplus_{\substack{(\stdg,C) \in \SetStandardizedIntervalPartitions \times \punion \stdL \\\\ \std(\stdL(C))\ge \stdg}}\{A,B \subset \punion \stdL\;|A\cup B = C, \std(\stdL(A)) \ge \stds, \std(\stdL(B)) \ge \stdt,\\& \std(\stdL(A)_{\stds}\gluepartitions \stdL(B)_{\stdt}) = \stdg \}|\\ &= \sum_{\substack{(\stdg,C) \in \SetStandardizedIntervalPartitions \times \punion \stdL \\\\ \std(\stdL(C))\ge \stdg}}\langle \qspart(\stds \otimes \stdt) , \stdg \rangle= \sum_{\stdg \in \SetStandardizedIntervalPartitions} \sum_{ \substack{ C   \in \punion \stdL \\ \std(\stdL(C))\ge \stdg}}\langle \qspart(\stds \otimes \stdt) , \stdg \rangle\\&= \sum_{\stdg \in \SetStandardizedIntervalPartitions} \langle \qspart(\stds \otimes \stdt) , \stdg \rangle \sum_{ \substack{ C   \in \punion \stdL \\ \std(\stdL(C))\ge \stdg}}1= \sum_{\stdg \in \SetStandardizedIntervalPartitions} \langle \qspart(\stds \otimes \stdt) , \stdg \rangle  \Big\langle  \IPC(\stdL), \stdg\Big\rangle \\&=  \Big\langle  \IPC(\stdL), \sum_{\stdg \in \SetStandardizedIntervalPartitions} \langle \qspart(\stds \otimes \stdt) , \stdg \rangle \stdg \Big\rangle   = \Big\langle  \IPC(\stdL), \stds \, \qspart \,\stdt \Big\rangle.
\end{align*}
We used \Cref{lemma:character1} for the second equality and \Cref{lemma:character2,lemma:character3,lemma:character4} for the third one.
\end{proof}

\newcommand{\sameheight}[1]{\vphantom{d}#1}

\begin{example}
  In the following expression, we write $\tikzmarknode[rectangle,black]{}{\sameheight{a}} \tikzmarknode[rectangle,black]{}{\sameheight{b}}\tikzmarknode[rectangle,black]{}{\sameheight{c}} \tikzmarknode[rectangle,black]{}{\sameheight{d}}$ for $\{\{1,2,3,4\}\}$ to illustrate how the mechanism works.
\begin{align*}
   &\Big \langle\IPC\left(\;\tikzmarknode[rectangle,black]{}{\sameheight{a}} \tikzmarknode[rectangle,black]{}{\sameheight{b}}\tikzmarknode[rectangle,black]{}{\sameheight{c}} \tikzmarknode[rectangle,black]{}{\sameheight{d}}\;\right),\;\; \tikzmarknode[rectangle,red]{}{\sameheight{a}} \tikzmarknode[rectangle,red]{}{\sameheight{b}} \; \Big \rangle \cdot \Big \langle\IPC\left(\;\tikzmarknode[rectangle,black]{}{\sameheight{a}} \tikzmarknode[rectangle,black]{}{\sameheight{b}}\tikzmarknode[rectangle,black]{}{\sameheight{c}} \tikzmarknode[rectangle,black]{}{\sameheight{d}}\;\right),\;\; \tikzmarknode[rectangle,blue]{}{\sameheight{a}}\;\;\tikzmarknode[rectangle,blue]{}{\sameheight{b}} \; \Big \rangle\\\\&= \Big \vert\Big\{\;\tikzmarknode[rectangle,red]{}{\sameheight{a}} \tikzmarknode[rectangle,red]{}{\sameheight{b}}\tikzmarknode[rectangle,black]{}{\sameheight{c}} \tikzmarknode[rectangle,black]{}{\sameheight{d}},\;\;\tikzmarknode[rectangle,black]{}{\sameheight{a}} \tikzmarknode[rectangle,red]{}{\sameheight{b}}\tikzmarknode[rectangle,red]{}{\sameheight{c}} \tikzmarknode[rectangle,black]{}{\sameheight{d}},\;\;\tikzmarknode[rectangle,black]{}{\sameheight{a}} \tikzmarknode[rectangle,black]{}{\sameheight{b}}\tikzmarknode[rectangle,red]{}{\sameheight{c}} \tikzmarknode[rectangle,red]{}{\sameheight{d}}\;\Big\} \times \Big\{\;\tikzmarknode[rectangle,blue]{}{\sameheight{a}} \tikzmarknode[rectangle,blue]{}{\sameheight{b}}\tikzmarknode[rectangle,black]{}{\sameheight{c}} \tikzmarknode[rectangle,black]{}{\sameheight{d}},\;\;\tikzmarknode[rectangle,blue]{}{\sameheight{a}} \tikzmarknode[rectangle,black]{}{\sameheight{b}}\tikzmarknode[rectangle,blue]{}{\sameheight{c}} \tikzmarknode[rectangle,black]{}{\sameheight{d}},\\&\tikzmarknode[rectangle,blue]{}{\sameheight{a}} \tikzmarknode[rectangle,black]{}{\sameheight{b}}\tikzmarknode[rectangle,black]{}{\sameheight{c}} \tikzmarknode[rectangle,blue]{}{\sameheight{d}},\;\;\tikzmarknode[rectangle,black]{}{\sameheight{a}} \tikzmarknode[rectangle,blue]{}{\sameheight{b}}\tikzmarknode[rectangle,blue]{}{\sameheight{c}} \tikzmarknode[rectangle,black]{}{\sameheight{d}},\;\;\tikzmarknode[rectangle,black]{}{\sameheight{a}} \tikzmarknode[rectangle,blue]{}{\sameheight{b}}\tikzmarknode[rectangle,black]{}{\sameheight{c}}  \tikzmarknode[rectangle,blue]{}{\sameheight{d}},\;\;\tikzmarknode[rectangle,black]{}{\sameheight{a}} \tikzmarknode[rectangle,black]{}{\sameheight{b}}\tikzmarknode[rectangle,blue]{}{\sameheight{c}}  \tikzmarknode[rectangle,blue]{}{\sameheight{d}}\;\Big\}\Big \vert\\\\&= \Big \vert \Big \{\overbrace{\underbrace{\tikzmarknode[rectangle,violet]{}{\sameheight{a}} \tikzmarknode[rectangle,violet]{}{\sameheight{b}}\tikzmarknode[rectangle,black]{}{\sameheight{c}} \tikzmarknode[rectangle,black]{}{\sameheight{d}}}_{(\{\sameheight{a},\sameheight{b}\},\{\sameheight{a},\sameheight{b}\})}}^{   
 \{\{\sameheight{a},\sameheight{b}\}\} \gluepartitions \{\{\sameheight{a}\},\{\sameheight{b}\}\} = \{\{\sameheight{a},\sameheight{b}\}\}},\;\;\underbrace{\tikzmarknode[rectangle,black]{}{\sameheight{a}} \tikzmarknode[rectangle,violet]{}{\sameheight{b}}\tikzmarknode[rectangle,violet]{}{\sameheight{c}}\tikzmarknode[rectangle,black]{}{\sameheight{d}}}_{(\{\sameheight{b},\sameheight{c}\},\{\sameheight{b},\sameheight{c}\})},\;\;\underbrace{\tikzmarknode[rectangle,black]{}{\sameheight{a}} \tikzmarknode[rectangle,black]{}{\sameheight{b}}\tikzmarknode[rectangle,violet]{}{\sameheight{c}}\tikzmarknode[rectangle,violet]{}{\sameheight{d}}}_{(\{\sameheight{c},\sameheight{d}\},\{\sameheight{c},\sameheight{d}\})},\\\\&\overbrace{\underbrace{\tikzmarknode[rectangle,violet]{}{\sameheight{a}} \tikzmarknode[rectangle,red]{}{\sameheight{b}}\tikzmarknode[rectangle,blue]{}{\sameheight{c}}\tikzmarknode[rectangle,black]{}{\sameheight{d}}}_{(\{\sameheight{a},\sameheight{b}\},\{\sameheight{a},\sameheight{c}\})}}^{\{\{\sameheight{a},\sameheight{b}\}\} \gluepartitions \{\{\sameheight{a}\},\{\sameheight{c}\}\} = \{\{\sameheight{a},\sameheight{b}\},\{\sameheight{c}\}\}},\;\;\overbrace{\underbrace{\tikzmarknode[rectangle,red]{}{\sameheight{a}} \tikzmarknode[rectangle,violet]{}{\sameheight{b}}\tikzmarknode[rectangle,blue]{}{\sameheight{c}} \tikzmarknode[rectangle,black]{}{\sameheight{d}}}_{(\{\sameheight{a},\sameheight{b}\},\{\sameheight{b},\sameheight{c}\})}}^{\{\{\sameheight{a},\sameheight{b}\}\} \gluepartitions \{\{\sameheight{b}\},\{\sameheight{c}\}\} = \{\{\sameheight{a},\sameheight{b}\},\{\sameheight{c}\}\}},\\\\&\;\;\overbrace{\underbrace{\tikzmarknode[rectangle,violet]{}{\sameheight{a}} \tikzmarknode[rectangle,red]{}{\sameheight{b}}\tikzmarknode[rectangle,black]{}{\sameheight{c}} \tikzmarknode[rectangle,blue]{}{\sameheight{d}}}_{(\{\sameheight{a},\sameheight{b}\},\{\sameheight{a},\sameheight{d}\})}}^{\{\{\sameheight{a},\sameheight{b}\}\} \gluepartitions \{\{\sameheight{a}\},\{\sameheight{d}\}\} = \{\{\sameheight{a},\sameheight{b}\},\{\sameheight{d}\}\}},\;\;\overbrace{\underbrace{\tikzmarknode[rectangle,red]{}{\sameheight{a}} \tikzmarknode[rectangle,violet]{}{\sameheight{b}}\tikzmarknode[rectangle,black]{}{\sameheight{c}} \tikzmarknode[rectangle,blue]{}{\sameheight{d}}}_{(\{\sameheight{a},\sameheight{b}\},\{\sameheight{b},\sameheight{d}\})}}^{\{\{\sameheight{a},\sameheight{b}\}\} \gluepartitions \{\{\sameheight{b}\},\{\sameheight{d}\}\} = \{\{\sameheight{a},\sameheight{b}\},\{\sameheight{d}\}\}},\\\\& \underbrace{\tikzmarknode[rectangle,black]{}{\sameheight{a}} \tikzmarknode[rectangle,violet]{}{\sameheight{b}}\tikzmarknode[rectangle,red]{}{\sameheight{c}} \tikzmarknode[rectangle,blue]{}{\sameheight{d}}}_{(\{\sameheight{b},\sameheight{c}\},\{\sameheight{b},\sameheight{d}\})},\;\;\underbrace{\tikzmarknode[rectangle,black]{}{\sameheight{a}} \tikzmarknode[rectangle,red]{}{\sameheight{b}}\tikzmarknode[rectangle,violet]{}{\sameheight{c}} \tikzmarknode[rectangle,blue]{}{\sameheight{d}}}_{(\{\sameheight{b},\sameheight{c}\},\{\sameheight{c},\sameheight{d}\})}, \\\\& \overbrace{\underbrace{\tikzmarknode[rectangle,blue]{}{\sameheight{a}} \tikzmarknode[rectangle,violet]{}{\sameheight{b}}\tikzmarknode[rectangle,red]{}{\sameheight{c}} \tikzmarknode[rectangle,black]{}{\sameheight{d}}}_{(\{\sameheight{b},\sameheight{c}\},\{\sameheight{a},\sameheight{b}\})}}^{\{\{\sameheight{b},\sameheight{c}\}\} \gluepartitions \{\{\sameheight{a}\},\{\sameheight{b}\}\} = \{\{\sameheight{a}\},\{\sameheight{b},\sameheight{c}\}\}},\;\;\overbrace{\underbrace{\tikzmarknode[rectangle,blue]{}{\sameheight{a}} \tikzmarknode[rectangle,red]{}{\sameheight{b}}\tikzmarknode[rectangle,violet]{}{\sameheight{c}} \tikzmarknode[rectangle,black]{}{\sameheight{d}}}_{(\{\sameheight{b},\sameheight{c}\},\{\sameheight{a},\sameheight{c}\})}}^{\{\{\sameheight{b},\sameheight{c}\}\} \gluepartitions \{\{\sameheight{a}\},\{\sameheight{c}\}\} = \{\{\sameheight{a}\},\{\sameheight{b},\sameheight{c}\}\}},\\\\& \underbrace{\tikzmarknode[rectangle,blue]{}{\sameheight{a}} \tikzmarknode[rectangle,black]{}{\sameheight{b}}\tikzmarknode[rectangle,violet]{}{\sameheight{c}} \tikzmarknode[rectangle,red]{}{\sameheight{d}}}_{(\{\sameheight{c},\sameheight{d}\},\{\sameheight{a},\sameheight{c}\})},\;\;\underbrace{\tikzmarknode[rectangle,blue]{}{\sameheight{a}} \tikzmarknode[rectangle,black]{}{\sameheight{b}}\tikzmarknode[rectangle,red]{}{\sameheight{c}} \tikzmarknode[rectangle,violet]{}{\sameheight{d}}}_{(\{\sameheight{c},\sameheight{d}\},\{\sameheight{a},\sameheight{d}\})},\;\;\underbrace{\tikzmarknode[rectangle,black]{}{\sameheight{a}} \tikzmarknode[rectangle,blue]{}{\sameheight{b}}\tikzmarknode[rectangle,violet]{}{\sameheight{c}} \tikzmarknode[rectangle,red]{}{\sameheight{d}}}_{(\{\sameheight{c},\sameheight{d}\},\{\sameheight{b},\sameheight{c}\})},\;\;\underbrace{\tikzmarknode[rectangle,black]{}{\sameheight{a}} \tikzmarknode[rectangle,blue]{}{\sameheight{b}}\tikzmarknode[rectangle,red]{}{\sameheight{c}} \tikzmarknode[rectangle,violet]{}{\sameheight{d}}}_{(\{\sameheight{c},\sameheight{d}\},\{\sameheight{b},\sameheight{d}\})},\\\\& \overbrace{\underbrace{\tikzmarknode[rectangle,red]{}{\sameheight{a}} \tikzmarknode[rectangle,red]{}{\sameheight{b}}\tikzmarknode[rectangle,blue]{}{\sameheight{c}} \tikzmarknode[rectangle,blue]{}{\sameheight{d}}}_{(\{\sameheight{a},\sameheight{b}\},\{\sameheight{c},\sameheight{d}\})}}^{\{\{\sameheight{a},\sameheight{b}\}\} \gluepartitions \{\{\sameheight{c}\},\{\sameheight{d}\}\} = \{\{\sameheight{a},\sameheight{b}\},\{\sameheight{c}\},\{\sameheight{d}\}\}},\\\\& \overbrace{\underbrace{\tikzmarknode[rectangle,blue]{}{\sameheight{a}} \tikzmarknode[rectangle,red]{}{\sameheight{b}}\tikzmarknode[rectangle,red]{}{\sameheight{c}}\tikzmarknode[rectangle,blue]{}{\sameheight{d}}}_{(\{\sameheight{b},\sameheight{c}\},\{\sameheight{a},\sameheight{d}\})}}^{\{\{\sameheight{b},\sameheight{c}\}\} \gluepartitions \{\{\sameheight{a}\},\{\sameheight{d}\}\} = \{\{\sameheight{a}\},\{\sameheight{b},\sameheight{c}\},\{\sameheight{d}\}\}},\\\\&\overbrace{\underbrace{\tikzmarknode[rectangle,blue]{}{\sameheight{a}} \tikzmarknode[rectangle,blue]{}{\sameheight{b}}\tikzmarknode[rectangle,red]{}{\sameheight{c}}\tikzmarknode[rectangle,red]{}{\sameheight{d}}}_{(\{\sameheight{c},\sameheight{d}\},\{\sameheight{a},\sameheight{b}\})}}^{\{\{\sameheight{c},\sameheight{d}\}\} \gluepartitions \{\{\sameheight{a}\},\{\sameheight{b}\}\} = \{\{\sameheight{a}\},\{\sameheight{b}\},\{\sameheight{c},\sameheight{d}\}\}} \Big \} \Big \vert.
\end{align*}
As an example, notice that
\begin{align*}
   \tikzmarknode[rectangle,blue]{}{\sameheight{a}} \tikzmarknode[rectangle,blue]{}{\sameheight{b}}\tikzmarknode[rectangle,red]{}{\sameheight{c}}\tikzmarknode[rectangle,red]{}{\sameheight{d}}
\end{align*}
is counted as $\tikzmarknode[rectangle,black]{}{\sameheight{a}} \;\;\tikzmarknode[rectangle,black]{}{\sameheight{b}}\;\;\tikzmarknode[rectangle,black]{}{\sameheight{c}}\tikzmarknode[rectangle,black]{}{\sameheight{d}}$, since
\begin{align*}
\Big \langle \; \tikzmarknode[rectangle,black]{}{\sameheight{a}}\;\;\tikzmarknode[rectangle,black]{}{\sameheight{b}} \; \otimes \;\tikzmarknode[rectangle,black]{}{\sameheight{a}} \tikzmarknode[rectangle,black]{}{\sameheight{b}}  \;,\;\;\coqspart\left(\;\tikzmarknode[rectangle,black]{}{\sameheight{a}}\;\;\tikzmarknode[rectangle,black]{}{\sameheight{b}}\;\;\tikzmarknode[rectangle,black]{}{\sameheight{c}}\tikzmarknode[rectangle,black]{}{\sameheight{d}}\;\;\right)\Big \rangle = 1,
\end{align*}
and not as $\tikzmarknode[rectangle,black]{}{\sameheight{a}} \tikzmarknode[rectangle,black]{}{\sameheight{b}}\tikzmarknode[rectangle,black]{}{\sameheight{c}}\tikzmarknode[rectangle,black]{}{\sameheight{d}}$, since
\begin{align*}
  \Big \langle \; \tikzmarknode[rectangle,black]{}{\sameheight{a}}\;\;\tikzmarknode[rectangle,black]{}{\sameheight{b}} \; \otimes \;\tikzmarknode[rectangle,black]{}{\sameheight{a}} \tikzmarknode[rectangle,black]{}{\sameheight{b}}  \;,\;\;\coqspart\left(\;\tikzmarknode[rectangle,black]{}{\sameheight{a}} \tikzmarknode[rectangle,black]{}{\sameheight{b}}\tikzmarknode[rectangle,black]{}{\sameheight{c}}\tikzmarknode[rectangle,black]{}{\sameheight{d}}\;\right)\Big \rangle = 0.
\end{align*}
Since $\stdL = \adjacentSquares{4}$, $\stds= \adjacentSquares{1} \ \adjacentSquares{1}$ and $\stdt = \adjacentSquares{2}$, we have $\Big\langle  \IPC(\stdL), \stds\Big\rangle = 6$, $\Big\langle  \IPC(\stdL), \stdt \Big\rangle = 3$, and using \Cref{rem:explicit} we obtain
\begin{align*}
&\Big\langle  \IPC(\stdL), \adjacentSquares{2}\Big\rangle  = 6,\\
&\Big\langle  \IPC(\stdL), \adjacentSquares{1}\;\adjacentSquares{2}\Big\rangle = \Big\langle  \IPC(\stdL), \adjacentSquares{2}\;\adjacentSquares{1}\Big\rangle = 3,\\
&\Big\langle  \IPC(\stdL), \adjacentSquares{1}\;\adjacentSquares{2}\;\adjacentSquares{1}\Big\rangle =\Big\langle  \IPC(\stdL),  \adjacentSquares{2}\;\adjacentSquares{1}\;\adjacentSquares{1}\Big\rangle =\Big\langle  \IPC(\stdL), \adjacentSquares{1}\;\adjacentSquares{1}\;\adjacentSquares{2}\Big\rangle = 1,
\end{align*}
and computing
\begin{align*}
\adjacentSquares{2} \ \qspart \  \adjacentSquares{1}\;\adjacentSquares{1} &= \adjacentSquares{2} + 2\;\;\adjacentSquares{2}\;\adjacentSquares{1} + 2\;\;\adjacentSquares{1}\;\adjacentSquares{2}+\adjacentSquares{1}\;\adjacentSquares{1}\;\adjacentSquares{2}\\&+\adjacentSquares{1}\;\adjacentSquares{2}\;\adjacentSquares{1}+\adjacentSquares{2}\;\adjacentSquares{1}\;\adjacentSquares{1},
\end{align*}
we have
\begin{align*}
\Big\langle  \IPC(\stdL),\adjacentSquares{2} \Big\rangle \Big\langle  \IPC(\stdL),\adjacentSquares{1}\;\adjacentSquares{1} \Big\rangle &= \Big\langle  \IPC(\stdL),\adjacentSquares{2}\Big\rangle + 2\;\;\Big\langle \IPC(\stdL), \adjacentSquares{2}\;\adjacentSquares{1}\Big\rangle + 2\;\;\Big\langle  \IPC(\stdL),\adjacentSquares{1}\;\adjacentSquares{2}\Big \rangle\\&+\Big\langle  \IPC(\stdL),\adjacentSquares{1}\;\adjacentSquares{1}\;\adjacentSquares{2}\Big \rangle+\Big\langle  \IPC(\stdL),\adjacentSquares{1}\;\adjacentSquares{2}\;\adjacentSquares{1}\Big \rangle+\Big\langle  \IPC(\stdL),\adjacentSquares{2}\;\adjacentSquares{1}\;\adjacentSquares{1}\Big\rangle,
\end{align*}
which verifies \Cref{eq:qs_id_part}.
\end{example}

We also have a form of Chen's identity.

\begin{theorem}[Chen's identity]
\label{thm:ChenI_intpart}
 The maps $(\IPC(\stdL))_{\stdL}$ satisfy a form of Chen's identity, i.e., for all $\stds,\stdL,\stdM \in \SetStandardizedIntervalPartitions$:
    \begin{align}
    \label{eq:Chen_id_part}
       \Big\langle  \IPC(\stdL \conc \stdM), \stds \Big\rangle
       =
       \Big\langle  \IPC(\stdL) \otimes \IPC(\stdM), \deconc(\stds) \Big\rangle
    \end{align}
\end{theorem}
\begin{remark}
As a consequence of Chen's identity, we can sum over numbers of the form of \Cref{eq:single_block_count} to calculate explicitly any occurrence. Let $\stdL =\singleblockN_{1} \conc \cdots \conc \singleblockN_{m}$ and  $\stds = \singleblockn_{1}\conc \cdots \conc  \singleblockn_{k}$, where $m,k \in \N$ and $\forall i\in [m],\forall j \in [k], N_{i} \in \N_{\ge 1},n_{j} \in \N_{\ge 1}, \singleblockN_{i}:=\{[N_{i}]\},\singleblockn_{j}:=\{[n_{j}]\}$. Then

\begin{align*}
&\Big\langle \IPC(\stdL), \stds \Big\rangle = \Big\langle \IPC(\singleblockN_{1})\otimes \cdots \otimes \IPC(\singleblockN_{m}), \deconc^{m}(\singleblockn_{1}\conc \cdots \conc  \singleblockn_{k})\Big\rangle\\ &= \sum_{0 \le i_{1} \le \cdots \le i_{m-1} \le k}\Big\langle \IPC(\singleblockN_{1}), \singleblockn_{1} \conc \cdots \conc \singleblockn_{i_{1}} \Big\rangle\Big\langle \IPC(\singleblockN_{2} ), \singleblockn_{i_{1}+1} \conc \cdots \conc \singleblockn_{i_{2}} \Big\rangle \cdots \\&\qquad\cdots\Big\langle \IPC(\singleblockN_{m}), \singleblockn_{i_{m-1}+1}  \conc \cdots \conc \singleblockn_{k} \Big\rangle\\ &= \sum_{0 \le i_{1} \le \cdots \le i_{m-1} \le k} \mybinom{N_{1}- \sum\limits_{j=1}^{i_{1}}n_{j}+i_{1}}{i_{1}}\mybinom{N_{2}- \sum\limits_{j=i_{1}+1}^{i_{2}}n_{j}+ i_{2}-i_{1}}{i_{2}-i_{1}} \cdots\\&\qquad \cdots \mybinom{N_{m}- \sum\limits_{j=i_{m-1}+1}^{k}n_{j}+k-i_{m-1}}{k-i_{m-1}}
\end{align*}
\end{remark}

To prove \Cref{thm:ChenI_intpart}, we need three lemmas.

\begin{lemma}
\label{lemma:Chen1}
Let $\stdt \in \SetStandardizedIntervalPartitions$
\begin{align*}
\stdt &:= \{[n_{1}]\} \conc \{[n_{2}]\}  \conc \cdots \conc \{[n_{k}]\}  \\&= \left\{\{1,\dots,n_{1}\},\{n_{1}+1,\dots,n_{1}+n_{2}\},\dots,\{\sum_{i=1}^{k-1}n_{i}+1,\dots,\sum_{i=1}^{k}n_{i}\}\right\}.
\end{align*}
where $n_{1},\dots,n_{k} \in \N_{\ge 1}$ and $k \in \N$.
Let $\stds \in \SetStandardizedIntervalPartitions$ such that $\stdt \le \stds$. Then $\stds$ can be written as
\begin{align*}
\stds &:= \{[\sum_{i=1}^{j_{1}}n_{i}]\} \conc \{[\sum_{i=j_{1}+1}^{j_{2}}n_{i}]\} \conc \cdots \conc \{[\sum_{i=j_{N-1}+1}^{j_{N}}n_{i}]\}\\&= \{\{1,\dots,\sum_{i=1}^{j_{1}}n_{i}\},\{\sum_{i=1}^{j_{1}}n_{i}+1,\dots,\sum_{i=1}^{j_{2}}n_{i}\},\dots,\{\sum_{i=1}^{j_{N-1}}n_{i}+1,\dots,\sum_{i=1}^{j_{N}}n_{i}\}\}.
\end{align*}
where 
\begin{align*}
j_{1},\dots,j_{N} \in \N,1 \le j_{1} < \dots < j_{N} = k, 1\le N \le k.
\end{align*}
 In particular, for $t \in \{0\}\cup {N}$, setting $j_{t}:=0$, we have 
\begin{align*}
\{[n_{1}]\} \conc \cdots \conc \{[n_{j_{t}}]\}
&\le \{[\sum_{i=1}^{j_{1}}n_{i}]\} \conc \cdots \conc \{[\sum_{i=j_{t-1}+1}^{j_{t}}n_{i}]\},
\end{align*}
and 
\begin{align*}
\{[n_{j_{t}+1}]\} \conc \cdots \conc \{[n_{j_{N}}]\}
&\le \{[\sum_{i=j_{t}+1}^{j_{t+1}}n_{i}]\} \conc \cdots \conc \{[\sum_{i=j_{N-1}+1}^{j_{N}}n_{i}]\},
\end{align*}
where $\{[n_{1}]\} \conc \cdots \conc \{[n_{0}]\}:= \e$ and
$\{[n_{j_{N+1}}]\} \conc \cdots \conc \{[n_{j_{N}}]\}:= \e$.
\end{lemma}
\begin{proof}
The fact that $\stds$ is of the given form is simply an explicit description of the interval partitions which are coarser than $\stdt$.
\end{proof}

\begin{lemma}
\label{lemma:Chen2}
Let $\singleblockn_{1}\conc \cdots \conc  \singleblockn_{k} \in \SetStandardizedIntervalPartitions$, where $k \in \N$ and $\forall r \in [k]:  n_{r} \in \N_{\ge 1}, \singleblockn_{i}:=\{[n_{r}]\}$. Let $\stdL,\;\stdM \in \SetStandardizedIntervalPartitions$. Then we have
 \begin{align*}
 &\left\{A \subset \punion (\stdL \conc \stdM)|\; \std((\stdL \conc \stdM)(A)) \ge \singleblockn_{1}\conc \cdots \conc  \singleblockn_{k}\right\}\\
       &=\biguplus_{0\le r \le k}
   \Big\{A \subset \punion (\stdL \conc \stdM)|\; \std((\stdL \conc \stdM)(A \cap \punion\stdL)) \ge \singleblockn_{1}\conc \cdots \conc  \singleblockn_{r},\\&\qquad\qquad\quad \std((\stdL \conc \stdM)(A \cap \punion\stdM^{\prime}))\ge \singleblockn_{r+1}\conc \cdots \conc  \singleblockn_{k}\Big\}.
\end{align*}
\end{lemma}
\begin{proof}
Let \begin{align*}
    A \subset \punion (\stdL \conc \stdM) 
\end{align*}such that \begin{align*}
    \std\left((\stdL \conc \stdM)(A)\right)&\ge \singleblockn_{1}\conc \cdots \conc  \singleblockn_{k}.
\end{align*}
We now show that $\exists! r \in \{0,..k\}$ such that 
\begin{align*}
  \std\left((\stdL \conc \stdM)(A \cap \punion\stdL)\right) &\ge \singleblockn_{1}\conc \cdots \conc  \singleblockn_{r},\\ 
  \std\left((\stdL \conc \stdM)(A \cap \punion\stdM^{\prime})\right)&\ge \singleblockn_{r+1}\conc \cdots \conc  \singleblockn_{k}.
\end{align*}
First notice that
\begin{align*}
    \std\left((\stdL \conc \stdM)(A)\right) \ge n_{1} \conc n_{2} \conc \cdots \conc n_{k}
\end{align*}
is equivalent to
\begin{align*}
   (\stdL \conc \stdM)(A) \ge (\stdL \conc \stdM)(A)_{n_{1} \conc n_{2} \conc \cdots \conc n_{k}}
\end{align*}
because of \Cref{lemma:character0}. From \Cref{lemma:Chen1} we know $\exists!t \in \{0\}\cup {N}$ and therefore $\exists! j_{t}$, where
\begin{align*}
&j_{0}:=0\\
&j_{1},\dots,j_{N} \in \N,1 \le j_{1} < \dots < j_{N} = k, 1\le N \le k.
\end{align*}
 such that, using also the definition of $\conc$ and the definition of $(\stdL \conc \stdM)(A)$, we can write
\begin{align*}
     (\stdL \conc \stdM)(A)_{n_{1} \conc n_{2} \conc \cdots \conc n_{k}} =&\{\{\ell_{1},...,\ell_{n_{1}}\},\{\ell_{n_{1}+1},...,\ell_{n_{1}+n_{2}}\},...,\{\ell_{\sum_{i=1}^{j_{t}-1}n_{i}+1},...,\ell_{\sum_{i=1}^{j_{t}}n_{i}}\},\\   &\{m^{\prime}_{\sum_{i=1}^{j_{t}}n_{i}+1},...,m^{\prime}_{\sum_{i=1}^{j_{t}+1}n_{i}}\},\{m^{\prime}_{\sum_{i=1}^{j_{t}+1}n_{i}+1},...,m^{\prime}_{\sum_{i=1}^{j_{t}+2}n_{i}}\},\\&...,\{m^{\prime}_{\sum_{i=1}^{k-1}n_{i}+1},...,m^{\prime}_{\sum_{i=1}^{k}n_{i}}\}\}.
\end{align*}
where 
\begin{align*}
     &\ell_{1} \prec\cdots \prec \ell_{n_{1}}<\ell_{n_{1}+1}\prec \cdots \prec \ell_{n_{1}+n_{2}} < \cdots <\ell_{\sum_{i=1}^{j_{t}-1}n_{i}+1}\prec \cdots \prec \ell_{\sum_{i=1}^{j_{t}}n_{i}} \\&<m^{\prime}_{\sum_{i=1}^{j_{t}}n_{i}+1}\prec \cdots \prec m^{\prime}_{\sum_{i=1}^{j_{t}+1}n_{i}}< m^{\prime}_{\sum_{i=1}^{j_{t}+1}n_{i}+1}\prec \cdots \prec m^{\prime}_{\sum_{i=1}^{j_{t}+2}n_{i}}<\cdots\\&\cdots < m^{\prime}_{\sum_{i=1}^{k-1}n_{i}+1}\prec \cdots \prec m^{\prime}_{\sum_{i=1}^{k}n_{i}}.
\end{align*}
and
\begin{align*}
     (\stdL \conc \stdM)(A) =&\{\{\ell_{1},\dots,\ell_{ \sum_{i=1}^{j_{1}}n_{i}}\},\{\ell_{\sum_{i=1}^{j_{1}}n_{i}+1},\dots,\ell_{\sum_{i=1}^{j_{2}}n_{i}}\},\dots,\{\ell_{\sum_{i=1}^{j_{t-1}}n_{i}+1},\dots,\ell_{\sum_{i=1}^{j_{t}}n_{i}}\},\\   &\{m^{\prime}_{\sum_{i=1}^{j_{t}}n_{i}+1},\dots,m^{\prime}_{\sum_{i=1}^{j_{t+1}}n_{i}}\},\{m^{\prime}_{\sum_{i=1}^{j_{t+1}}n_{i}+1},\dots,m^{\prime}_{\sum_{i=1}^{j_{t+2}}n_{i}}\},\\&\dots,\{m^{\prime}_{\sum_{i=1}^{j_{N-1}}n_{i}+1},\dots,m^{\prime}_{\sum_{i=1}^{j_{N}}n_{i}}\}\}
\end{align*}
where
\begin{align*}
     &\ell_{1} \prec \cdots \prec \ell_{ \sum_{i=1}^{j_{1}}n_{i}}< \ell_{\sum_{i=1}^{j_{1}}n_{i}+1}\prec \cdots\prec \ell_{\sum_{i=1}^{j_{2}}n_{i}}< \cdots <\ell_{\sum_{i=1}^{j_{t-1}}n_{i}+1}\prec \cdots\prec \ell_{\sum_{i=1}^{j_{t}}n_{i}}\\&< m^{\prime}_{\sum_{i=1}^{j_{t}}n_{i}+1} \prec \cdots\prec m^{\prime}_{\sum_{i=1}^{j_{t+1}}n_{i}} < m^{\prime}_{\sum_{i=1}^{j_{t+1}}n_{i}+1}  \prec \cdots\prec m^{\prime}_{\sum_{i=1}^{j_{t+2}}n_{i}} <\cdots\\&\cdots<m^{\prime}_{\sum_{i=1}^{j_{N-1}}n_{i}+1}\prec \cdots\prec m^{\prime}_{\sum_{i=1}^{j_{N}}n_{i}}
\end{align*}
 where the $\ell$'s are elements of $\punion\stdL$ and the $m^{\prime}$'s are elements of $\punion\stdM^{\prime}$. Recall that $a \prec b \iff b = a+1$.
 
 If we consider
\begin{align*}
     (\stdL \conc \stdM)(A \cap \punion\stdL) =&\{\{\ell_{1},\dots,\ell_{ \sum_{i=1}^{j_{1}}n_{i}}\},\{\ell_{\sum_{i=1}^{j_{1}}n_{i}+1},\dots,\ell_{\sum_{i=1}^{j_{2}}n_{i}}\},\\&\dots,\{\ell_{\sum_{i=1}^{j_{t-1}}n_{i}+1},\dots,\ell_{\sum_{i=1}^{j_{t}}n_{i}}\}\}
\end{align*}
and
\begin{align*}
     (\stdL \conc \stdM)(A \cap \punion\stdM^{\prime}) =&\{\{m^{\prime}_{\sum_{i=1}^{j_{t}}n_{i}+1},\dots,m^{\prime}_{\sum_{i=1}^{j_{t+1}}n_{i}}\},\{m^{\prime}_{\sum_{i=1}^{j_{t+1}}n_{i}+1},\dots,m^{\prime}_{\sum_{i=1}^{j_{t+2}}n_{i}}\},\\&\dots,\{m^{\prime}_{\sum_{i=1}^{j_{N-1}}n_{i}+1},\dots,m^{\prime}_{\sum_{i=1}^{j_{N}}n_{i}}\}\}.
\end{align*}
they satisfy
\begin{align*}
  \std\left((\stdL \conc \stdM)(A \cap \punion\stdL)\right) &\ge \{[n_{1}]\} \conc \cdots \conc \{[n_{j_{t}}]\}\\ 
  \std\left((\stdL \conc \stdM)(A \cap \punion\stdM^{\prime})\right) &\ge \{[n_{j_{t}+1}]\} \conc \cdots \conc \{[n_{k}]\}.
\end{align*}

\end{proof}
\begin{lemma}
\label{lemma:Chen3}
Let $\singleblockn_{1}\conc \cdots \conc  \singleblockn_{k} \in \SetStandardizedIntervalPartitions$, where $k \in \N$ and $\forall r \in [k]:  n_{i} \in \N_{\ge 1}, \singleblockn_{r}:=\{[n_{r}]\}$. Let $\stdL,\;\stdM \in \SetStandardizedIntervalPartitions$. Then we have
 \begin{align*}
 &\Big|\bigcup_{0\le r \le k}\Big\{A \subset \punion (\stdL \conc \stdM)\Big|\; \std\left((\stdL \conc \stdM)(A \cap \punion\stdL)\right) \ge \singleblockn_{1}\conc \cdots \conc  \singleblockn_{r},\\
 &\qquad\qquad
 \qquad\qquad
 \qquad\qquad
 \std\left((\stdL \conc \stdM)(A \cap \punion\stdM^{\prime})\right)\ge \singleblockn_{r+1}\conc \cdots \conc  \singleblockn_{k}\Big\}\Big|\\
 &=\Big|\bigcup_{0\le r \le k}\Big\{A \subset \punion \stdL|\; \std\left(\stdL(A)\right) \ge \singleblockn_{1}\conc \cdots \conc  \singleblockn_{r}\Big\}\\
 &\qquad\times \Big\{B \subset \punion \stdM|\; \std\left(\stdM(B)\right) \ge \singleblockn_{r+1}\conc \cdots \conc  \singleblockn_{k}\Big\}\Big|.
\end{align*}
\end{lemma}
\begin{proof}
We explicitly give a bijection between the two sets.
For some $0\le r \le k$ let 
\begin{align*}
    &(A,B) \in \bigcup_{0\le r \le k}\{A \subset \punion \stdL|\; \std\left(\stdL(A)\right) \ge \singleblockn_{1}\conc \cdots \conc  \singleblockn_{r}\}\\&\times \{B \subset \punion \stdM|\; \std\left(\stdM(B)\right) \ge \singleblockn_{r+1}\conc \cdots \conc  \singleblockn_{k}\}
\end{align*}
and define the map
\begin{align*}
    (A,B) \mapsto A \cup (B+|\punion\stdL|).
\end{align*}
Therefore $(B+|\punion\stdL|)\subset \punion \stdM^{\prime}$ and if $B =\emptyset$, set $B+|\punion\stdL|:=\emptyset$.
We have
\begin{align*}
    \std\left((\stdL \conc \stdM)((A \cup (B+|\punion\stdL|)) \cap \punion\stdL)\right) = \std\left((\stdL \conc \stdM)(A  \cap \punion\stdL)\right) \ge \singleblockn_{1}\conc \cdots \conc  \singleblockn_{r}
\end{align*}
and
\begin{align*}
    &\std\left((\stdL \conc \stdM)((A \cup (B+|\punion\stdL|)) \cap \punion\stdM^{\prime})\right) = \std\left((\stdL \conc \stdM)( (B+|\punion\stdL|) \cap \punion\stdM^{\prime})\right)\\&= \std\left((\stdL \conc \stdM)( B \cap \punion\stdM^{\prime})\right) \ge \singleblockn_{r+1}\conc \cdots \conc  \singleblockn_{k}.
\end{align*}
Now, let 
\begin{align*}
   &A \in \bigcup_{0\le r \le k}\{A \subset \punion (\stdL \conc \stdM)|\; \std(\left(\stdL \conc \stdM\right)(A \cap \punion\stdL)) \ge \singleblockn_{1}\conc \cdots \conc  \singleblockn_{k},\\&\std\left((\stdL \conc \stdM)(A \cap \punion\stdM^{\prime})\right)\ge \singleblockn_{i+1}\conc \cdots \conc  \singleblockn_{k}\}.
\end{align*}
Define
\begin{align*}
    A \mapsto (A \cap (\punion\stdL),(A \cap (\punion\stdM^{\prime}))-|\punion\stdL|)
\end{align*}
and if $(A \cap (\punion\stdM^{\prime}))=\emptyset$, set $(A \cap (\punion\stdM^{\prime}))-|\punion\stdL|=\emptyset$, which is clearly the inverse of the other map.

\end{proof}

\begin{proof}[Proof of \Cref{thm:ChenI_intpart}]
We use \Cref{lemma:Chen1,lemma:Chen2,lemma:Chen3}.
\begin{align*}
 &\Big\langle  \IPC(\stdL \conc \stdM), \singleblockn_{1}\conc \cdots \conc  \singleblockn_{k} \Big\rangle\\
       &=\#\{A \subset \punion (\stdL \conc \stdM)|\; \std\left((\stdL \conc \stdM)(A)\right) \ge \singleblockn_{1}\conc \cdots \conc  \singleblockn_{k}\}\\
       &=\sum_{0\le i \le k}\{A \subset \punion (\stdL \conc \stdM)|\; \std\left((\stdL \conc \stdM)(A \cap \punion\stdL)\right) \ge \singleblockn_{1}\conc \cdots \conc  \singleblockn_{i},\\&\std\left((\stdL \conc \stdM)(A \cap \punion\stdM^{\prime})\right)\ge \singleblockn_{i+1}\conc \cdots \conc  \singleblockn_{k}\}\\
       &=\sum_{0\le i \le k}\#\{A \subset \punion \stdL|\; \std\left(\stdL(A)\right) \ge \singleblockn_{1}\conc \cdots \conc  \singleblockn_{i}\}\\&\times \{B \subset \punion \stdM|\; \std\left(\stdM(B)\right) \ge \singleblockn_{i+1}\conc \cdots \conc  \singleblockn_{k}\}\\
       &=\sum_{0\le i \le k}\#\{A|A \subset \punion\stdL,\stdL(A)) \ge \singleblockn_{1}\conc \cdots \conc  \singleblockn_{i}\}\\&\#\{B|B \subset \punion\stdM,\stdM(B)) \ge \singleblockn_{i+1}\conc \cdots \conc  \singleblockn_{k}\}\\&=
      \sum_{0\le i \le k} \Big\langle  \IPC(\stdL), \singleblockn_{1}\conc \cdots \conc  \singleblockn_{i} \Big\rangle \Big\langle  \IPC(\stdM), \singleblockn_{i+1}\conc \cdots \conc  \singleblockn_{k} \Big\rangle\\&=
       \Big\langle  \IPC(\stdL) \otimes \IPC(\stdM), \deconc(\stds) \Big\rangle.
\end{align*}
\end{proof}
\section{Vincular permutation patterns}
\label{sec:Vincular_permutation_patterns}
In this section, we define a Hopf algebra on pairs which consists of an interval partition of $[n]$ and a permutation of length $n$. A pair corresponds to a certain vincular pattern,  since the intervals express which values are consecutive in the permutation pattern. It incorporates both the Hopf algebra of 
\Cref{sec:Finite_interval_partitions} and the Hopf algebra appearing in \cite{vargas2014hopf}.
The underlying set is now given by
\begin{align*}
    \bigcup_{n \in \N}\SetStandardizedIntervalPartitions_{n} \times \bigsigma_{n},
\end{align*}
general elements of which we denote with $(\stds,\sigma)$.
We consider the free $\Q$-vector space over it.
\begin{definition}[$\Q$-vector space over vincular permutations] 
\begin{align}
  \label{eq:free_vector_space_gen_permutations}
  \FreeVincPer &:= \bigoplus_{n \in \N} \Q[\SetStandardizedIntervalPartitions_{n} \times \bigsigma_{n}  ].
\end{align}
graded by the size of the partitions (or alternatively, the length of permutations).
\end{definition}
We now introduce algebraic operations on $\FreeVincPer$. This can be seen as a \quotationmarks{combination} of our Hopf algebra on interval partitions, $\FreeIntPart$, and Vargas' superinfiltration Hopf algebra, $\FreePer$.

\subsection{Algebraic operations}

We recall the operations of the superinfiltration Hopf algebra introduced in \cite{vargas2014hopf}.   
\begin{definition}[Superinfiltration product, \cite{vargas2014hopf}]
\label{def:superinfiltration_coprod}
Let $\sigma,\tau \in \bigsigma$.
\begin{align*}
\sigma \ \superinfiltration \ \tau &:= \sum_{\gamma \in \bigsigma} \sum_{\substack{A \cup B = [|\gamma|]\\ \st(\gamma\evaluatedAt{A})=\sigma,\st(\gamma\evaluatedAt{B})=\tau}} \gamma.
\end{align*}
\end{definition}
\begin{definition}[\cite{vargas2014hopf}]
 \label{def:conc_vargas}
 Let $\sigma \in \bigsigma$.
\begin{align*}
\Delta_{\concvargas}(\sigma)&:=\sum_{\alpha \concvargas \beta}\alpha \otimes \beta,
\end{align*}
where for $\alpha \in \bigsigma_{m}$ and $\beta \in \bigsigma_{n}$ we have $\alpha \concvargas \beta:= \alpha_{1}\dots \alpha_{m}(\beta_{1}+m) \dots (\beta_{n}+m)$.
\end{definition}
Then, $(\FreePer, \superinfiltration, \Delta_{\,\concvargas})$, is a connected filtered Hopf algebra
\cite[Corollary 4.8]{vargas2014hopf}.

\begin{remark}[Malvenuto-Reutenauer Hopf algebra on permutations]
In \cite{vargas2014hopf}, the supershuffle coproduct
\begin{align*}
\Delta_{\,\scriptsize\underline{\shuffle}\;}(\sigma)&:=\sum_{\substack{A,B \subset [|\sigma|]:\\A \uplus B = [|\sigma|]}}\st\left(\sigma\evaluatedAt{A}\right) \otimes \st\left(\sigma\evaluatedAt{B}\right).
\end{align*}
also appears. Vargas has showed that $(\FreePer, \underline{\shuffle}, \Delta_{\concvargas})$ is a Hopf algebra. The supershuffle, $\underline{\shuffle}$, incorporates the algebraic operations of the Malvenuto-Reutenauer Hopf Algebra
on permutations, $(\FreePer, *^{\prime}, \Delta_{*})$, \cite{malvenuto1995duality}. If $\sigma \in \bigsigma_{m},\tau \in \bigsigma_{n}$
\begin{align*}
\sigma *^{\prime} \tau &:= \sigma\, \shuffle \,\tau_{m},\\
\Delta_{*}(\sigma)&:=\sum_{i=0}^{n} \st(\sigma_{1}\cdots \sigma_{i}) \otimes \st(\sigma_{i+1}\cdots \sigma_{n}).
\end{align*}
where $\shuffle$ is the usual shuffle product on words  and $\tau_{m}:=(\tau_{1}+m)\cdots (\tau_{n}+m)$. We present a few computations as examples. 
\begin{align*}
\Delta_{*}(1243)&= \e \otimes 1243 + 1 \otimes 132 + 12 \otimes 21 + 123 \otimes 1 + 1243 \otimes \e,\\
12 \, \underline{\shuffle} \, 21 &= 1243 + 1324 +\textcolor{red}{2}\; 1342 +\textcolor{red}{2}\; 1423 +\textcolor{red}{3}\; 1432 + 2134 +\textcolor{red}{2}\; 2314 \\&+\textcolor{red}{3}\; 2341 + 2413 +\textcolor{red}{2}\; 2431 +\textcolor{red}{2}\; 3124 \\&+ 3142 +\textcolor{red}{3}\; 3214 +\textcolor{red}{2}\; 3241 + 3421 \\&+\textcolor{red}{3}\; 4123 +\textcolor{red}{2}\; 4132 +\textcolor{red}{2}\; 4213 + 4231 + 4312,\\
     12 * 21 &= 1243 + 1342 + 1432 + 2341 + 2431 + 3421,\\
12 *^{\prime} 21 &= 1243 + 1423 + 4123 + 1432 + 4312 + 4132.  
\end{align*}
\end{remark}

\subsubsection{Products, coproducts}

We now endow $\FreeVincPer$ with a Hopf algebra structure.

The product combines the products $\conc$ (\Cref{def:conc_int_part}) and $\concvargas$ (\Cref{def:conc_vargas}).

\begin{definition} We define on basis elements, $(\stds,\sigma), (\stdt,\tau) \in \bigcup_{n \in \N} \SetStandardizedIntervalPartitions_{n} \times \bigsigma_{n}$
\begin{align*}
(\stds,\sigma) \genconc (\stdt,\tau):= (\stds \conc \stdt, \sigma \concvargas \tau).
\end{align*}
and then extend linearly.
 \end{definition}

\begin{proposition}[Associativity]
    The product $\genconc$ is associative.
\end{proposition}
\begin{proof}
    Since $\conc$ and $\concvargas$ are associative products, the result follows.
\end{proof}

\begin{example}[Product]
\begin{align*}
        \tikzmarknode[rectangle,black]{}{\textup{1}}\tikzmarknode[rectangle,black]{}{\textup{2}}
 \   \genconc \ \tikzmarknode[rectangle,black]{}{\textup{2}}\tikzmarknode[rectangle,black]{}{\textup{1}} &=        \tikzmarknode[rectangle,black]{}{\textup{1}}\tikzmarknode[rectangle,black]{}{\textup{2}}
 \; \tikzmarknode[rectangle,black]{}{\textup{4}}\tikzmarknode[rectangle,black]{}{\textup{3}}\\
 \end{align*}
 \end{example}

\begin{remark}
Its dual coproduct
is
\begin{align*}
\deconcgen((\stds,\sigma)) :=\sum_{\mathfrak{a} \conc \mathfrak{b} =\stds}\sum_{\substack{\alpha \concvargas \beta = \sigma\\ \alpha\in \bigsigma_{|\punion\mathfrak{a}|},\,\beta\in \bigsigma_{|\punion\mathfrak{b}|}}} (\stds,\sigma) \otimes (\stdt,\tau)
\end{align*}
For example
 \begin{align*}
 \deconcgen \left(\,\tikzmarknode[rectangle,black]{}{\textup{1}}\tikzmarknode[rectangle,black]{}{\textup{2}}\,\right) &=        \e \otimes \tikzmarknode[rectangle,black]{}{\textup{1}}\tikzmarknode[rectangle,black]{}{\textup{2}} + \tikzmarknode[rectangle,black]{}{\textup{1}}\tikzmarknode[rectangle,black]{}{\textup{2}} \otimes \e\\
       \deconcgen \left(\,\tikzmarknode[rectangle,black]{}{\textup{2}}\;\;\tikzmarknode[rectangle,black]{}{\textup{1}}\,\right) &=        \e \otimes \tikzmarknode[rectangle,black]{}{\textup{2}}\;\;\tikzmarknode[rectangle,black]{}{\textup{1}} + \tikzmarknode[rectangle,black]{}{\textup{2}}
       \;\;\tikzmarknode[rectangle,black]{}{\textup{1}} \otimes \e\\
       \deconcgen \left(\,\tikzmarknode[rectangle,black]{}{\textup{1}}\;\;\tikzmarknode[rectangle,black]{}{\textup{2}}\,\right) &=        \e \otimes \tikzmarknode[rectangle,black]{}{\textup{1}}\;\;\tikzmarknode[rectangle,black]{}{\textup{2}} + \tikzmarknode[rectangle,black]{}{\textup{1}}\;\;\tikzmarknode[rectangle,black]{}{\textup{2}}  \otimes \e + 2\,\,\tikzmarknode[rectangle,black]{}{\textup{1}} \otimes \tikzmarknode[rectangle,black]{}{\textup{1}}
\end{align*}
 \end{remark}

The following definition incorporates \Cref{def:coproduct_partitions}, and \Cref{def:superinfiltration_coprod}.
\begin{definition}
\label{def:coproduct_gen_patt}
 On basis elements,  $\bigcup_{n \in \N}\SetStandardizedIntervalPartitions_{n} \times \bigsigma_{n}$, we define 
\begin{align*}
\coqsgen\left((\stds,\sigma)\right)&:=\sum_{A \cup A^{\prime} = \punion\stds}\;\sum_{\substack{\unstdI \in \SetIntervalPartitions(A),\;\unstdI^{\prime} \in \SetIntervalPartitions(A^{\prime})  \\\\ \unstdI \gluepartitions \unstdI^{\prime} = \stds}}(\std\left(\unstdI\right),\st(\sigma\evaluatedAt{A})) \otimes (\std\left(\unstdI^{\prime}\right),\st(\sigma\evaluatedAt{A^{\prime}})) .
\end{align*}
and extend linearly. Its dual product is defined as
\begin{align*}
(\stds,\sigma) \qsgen (\stdt,\tau)&:=\sum_{\substack{(\stdg, \gamma) \in \bigcup_{n \in \N}\SetStandardizedIntervalPartitions_{n}\times \bigsigma_{n}\\\\A \cup A^{\prime} = \punion\stdg}}\;\sum_{\substack{\unstdI \in \SetIntervalPartitions(A),\;\unstdI^{\prime} \in \SetIntervalPartitions(A^{\prime})\\\\ \std\left(\unstdI\right) = \stds,\; \std\left(\unstdI^{\prime}\right) = \stdt \\\\ \unstdI \gluepartitions \unstdI^{\prime} = \stdg}}\sum_{\st(\gamma\evaluatedAt{A})=\sigma,\st(\gamma\evaluatedAt{A^{\prime}})=\tau}(\stdg,\gamma).
\end{align*}

\end{definition}

For readability, we will also just write:
\begin{align*}
      (12,\{\{1,2\}\})&=: \tikzmarknode[rectangle,black]{}{1}\tikzmarknode[rectangle,black]{}{2}\\
      (21,\{\{1\},\{2\}\})&=:  \tikzmarknode[rectangle,black]{}{2} \ \tikzmarknode[rectangle,black]{}{1}
\end{align*}
and so on.

\begin{proposition}[Coassociativity]
\label{thm:combined_quasicoprod_associative}
The coproduct $\coqsgen$ is coassociative.
\end{proposition}
\begin{proof}
The proof is analogous to the one of \Cref{prop:coprod_partitions_coass}
and therefore omitted.
\end{proof}

\begin{example}[Coproduct]
\begin{align*}
\coqsgen\left(\ \tikzmarknode[rectangle,black]{}{\textup{1}} \;\; \tikzmarknode[rectangle,black]{}{\textup{2}}\ \right) &= \e \otimes \ \tikzmarknode[rectangle,black]{}{\textup{1}} \;\; \tikzmarknode[rectangle,black]{}{\textup{2}}\ + \ \tikzmarknode[rectangle,black]{}{\textup{1}} \;\; \tikzmarknode[rectangle,black]{}{\textup{2}}\  \otimes \e + 2 \ \ \tikzmarknode[rectangle,black]{}{\textup{1}}  \otimes \ \tikzmarknode[rectangle,black]{}{\textup{1}} \\ &+ 2 \ \ \tikzmarknode[rectangle,black]{}{\textup{1}}  \otimes \ \tikzmarknode[rectangle,black]{}{\textup{1}} \;\; \tikzmarknode[rectangle,black]{}{\textup{2}} + 2 \ \tikzmarknode[rectangle,black]{}{\textup{1}} \;\; \tikzmarknode[rectangle,black]{}{\textup{2}} \  \otimes \ \tikzmarknode[rectangle,black]{}{\textup{1}} +  \tikzmarknode[rectangle,black]{}{\textup{1}} \;\; \tikzmarknode[rectangle,black]{}{\textup{2}} \  \otimes \ \tikzmarknode[rectangle,black]{}{\textup{1}} \;\; \tikzmarknode[rectangle,black]{}{\textup{2}}\,.
   \end{align*}
\end{example}

\begin{example}[Product]
\begin{align*}
        \tikzmarknode[rectangle,black]{}{\textup{1}} \  \qsgen \  \tikzmarknode[rectangle,black]{}{\textup{1}} &= 2 \   \tikzmarknode[rectangle,black]{}{\textup{1}} \;\; \tikzmarknode[rectangle,black]{}{\textup{2}} + 2 \ \tikzmarknode[rectangle,black]{}{\textup{2}} \;\; \tikzmarknode[rectangle,black]{}{\textup{1}} +  \tikzmarknode[rectangle,black]{}{\textup{1}}\\
   \tikzmarknode[rectangle,black]{}{\textup{2}}\tikzmarknode[rectangle,black]{}{\textup{1}} \ \qsgen  \ \tikzmarknode[rectangle,black]{}{\textup{1}} &= 2\ \tikzmarknode[rectangle,black]{}{\textup{2}}\tikzmarknode[rectangle,black]{}{\textup{1}}
   +
   \tikzmarknode[rectangle,black]{}{\textup{2}}\tikzmarknode[rectangle,black]{}{\textup{1}} \;\; \tikzmarknode[rectangle,black]{}{\textup{3}}
   +
   \tikzmarknode[rectangle,black]{}{\textup{3}}\tikzmarknode[rectangle,black]{}{\textup{1}} \;\; \tikzmarknode[rectangle,black]{}{\textup{2}}
   +
   \tikzmarknode[rectangle,black]{}{\textup{3}}\tikzmarknode[rectangle,black]{}{\textup{2}} \;\; \tikzmarknode[rectangle,black]{}{\textup{1}}
   +
   \tikzmarknode[rectangle,black]{}{\textup{1}}\;\; \tikzmarknode[rectangle,black]{}{\textup{3}} \tikzmarknode[rectangle,black]{}{\textup{2}}
   \\&+
   \tikzmarknode[rectangle,black]{}{\textup{2}}\;\; \tikzmarknode[rectangle,black]{}{\textup{3}}\tikzmarknode[rectangle,black]{}{\textup{1}}
   +
   \tikzmarknode[rectangle,black]{}{\textup{3}}\;\; \tikzmarknode[rectangle,black]{}{\textup{2}} \tikzmarknode[rectangle,black]{}{\textup{1}}\\\\
\tikzmarknode[rectangle,black]{}{\textup{1}}\tikzmarknode[rectangle,black]{}{\textup{2}}
 \   \qsgen  \ \tikzmarknode[rectangle,black]{}{\textup{1}}\tikzmarknode[rectangle,black]{}{\textup{2}}
  &= \tikzmarknode[rectangle,black]{}{\textup{1}}\tikzmarknode[rectangle,black]{}{\textup{2}} + 2\ \tikzmarknode[rectangle,black]{}{\textup{1}}\tikzmarknode[rectangle,black]{}{\textup{2}}\tikzmarknode[rectangle,black]{}{\textup{3}} +      2\ \tikzmarknode[rectangle,black]{}{\textup{1}}\tikzmarknode[rectangle,black]{}{\textup{2}}
 \;\; \tikzmarknode[rectangle,black]{}{\textup{3}}\tikzmarknode[rectangle,black]{}{\textup{4}}
 + 2\ \tikzmarknode[rectangle,black]{}{\textup{1}}\tikzmarknode[rectangle,black]{}{\textup{3}}
 \;\; \tikzmarknode[rectangle,black]{}{\textup{2}}\tikzmarknode[rectangle,black]{}{\textup{4}}\\&  + 2\ \tikzmarknode[rectangle,black]{}{\textup{1}}\tikzmarknode[rectangle,black]{}{\textup{4}}
 \;\; \tikzmarknode[rectangle,black]{}{\textup{2}}\tikzmarknode[rectangle,black]{}{\textup{3}} + 2\ \tikzmarknode[rectangle,black]{}{\textup{2}}\tikzmarknode[rectangle,black]{}{\textup{3}}
 \;\; \tikzmarknode[rectangle,black]{}{\textup{1}}\tikzmarknode[rectangle,black]{}{\textup{4}}
 + 2\ \tikzmarknode[rectangle,black]{}{\textup{2}}\tikzmarknode[rectangle,black]{}{\textup{4}}
 \;\; \tikzmarknode[rectangle,black]{}{\textup{1}}\tikzmarknode[rectangle,black]{}{\textup{3}}\\&  + 2\ \tikzmarknode[rectangle,black]{}{\textup{3}}\tikzmarknode[rectangle,black]{}{\textup{4}}
 \;\; \tikzmarknode[rectangle,black]{}{\textup{1}}\tikzmarknode[rectangle,black]{}{\textup{2}}\\\\      
\tikzmarknode[rectangle,black]{}{\textup{1}}\tikzmarknode[rectangle,black]{}{\textup{2}}
 \   \qsgen  \ \tikzmarknode[rectangle,black]{}{\textup{2}}\tikzmarknode[rectangle,black]{}{\textup{1}}&= \tikzmarknode[rectangle,black]{}{\textup{1}}\tikzmarknode[rectangle,black]{}{\textup{3}}\tikzmarknode[rectangle,black]{}{\textup{2}} + \tikzmarknode[rectangle,black]{}{\textup{2}}\tikzmarknode[rectangle,black]{}{\textup{1}}\tikzmarknode[rectangle,black]{}{\textup{3}} + \tikzmarknode[rectangle,black]{}{\textup{2}}\tikzmarknode[rectangle,black]{}{\textup{3}}\tikzmarknode[rectangle,black]{}{\textup{1}} + \tikzmarknode[rectangle,black]{}{\textup{3}}\tikzmarknode[rectangle,black]{}{\textup{1}}\tikzmarknode[rectangle,black]{}{\textup{2}}\\&+ \tikzmarknode[rectangle,black]{}{\textup{1}}\tikzmarknode[rectangle,black]{}{\textup{2}} \;\; \tikzmarknode[rectangle,black]{}{\textup{4}}\tikzmarknode[rectangle,black]{}{\textup{3}}+ \tikzmarknode[rectangle,black]{}{\textup{1}}\tikzmarknode[rectangle,black]{}{\textup{3}} \;\; \tikzmarknode[rectangle,black]{}{\textup{4}}\tikzmarknode[rectangle,black]{}{\textup{2}}+ \tikzmarknode[rectangle,black]{}{\textup{1}}\tikzmarknode[rectangle,black]{}{\textup{4}} \;\; \tikzmarknode[rectangle,black]{}{\textup{3}}\tikzmarknode[rectangle,black]{}{\textup{2}}+ \tikzmarknode[rectangle,black]{}{\textup{2}}\tikzmarknode[rectangle,black]{}{\textup{1}} \;\; \tikzmarknode[rectangle,black]{}{\textup{3}}\tikzmarknode[rectangle,black]{}{\textup{4}}\\&+ \tikzmarknode[rectangle,black]{}{\textup{2}}\tikzmarknode[rectangle,black]{}{\textup{3}} \;\; \tikzmarknode[rectangle,black]{}{\textup{4}}\tikzmarknode[rectangle,black]{}{\textup{1}}+ \tikzmarknode[rectangle,black]{}{\textup{2}}\tikzmarknode[rectangle,black]{}{\textup{4}} \;\; \tikzmarknode[rectangle,black]{}{\textup{3}}\tikzmarknode[rectangle,black]{}{\textup{1}}+ \tikzmarknode[rectangle,black]{}{\textup{3}}\tikzmarknode[rectangle,black]{}{\textup{1}} \;\; \tikzmarknode[rectangle,black]{}{\textup{2}}\tikzmarknode[rectangle,black]{}{\textup{4}}+ \tikzmarknode[rectangle,black]{}{\textup{3}}\tikzmarknode[rectangle,black]{}{\textup{2}} \;\; \tikzmarknode[rectangle,black]{}{\textup{1}}\tikzmarknode[rectangle,black]{}{\textup{4}}\\&+ \tikzmarknode[rectangle,black]{}{\textup{3}}\tikzmarknode[rectangle,black]{}{\textup{4}} \;\; \tikzmarknode[rectangle,black]{}{\textup{2}}\tikzmarknode[rectangle,black]{}{\textup{1}}+ \tikzmarknode[rectangle,black]{}{\textup{4}}\tikzmarknode[rectangle,black]{}{\textup{1}} \;\; \tikzmarknode[rectangle,black]{}{\textup{2}}\tikzmarknode[rectangle,black]{}{\textup{3}}+ \tikzmarknode[rectangle,black]{}{\textup{4}}\tikzmarknode[rectangle,black]{}{\textup{2}} \;\; \tikzmarknode[rectangle,black]{}{\textup{1}}\tikzmarknode[rectangle,black]{}{\textup{3}}+ \tikzmarknode[rectangle,black]{}{\textup{4}}\tikzmarknode[rectangle,black]{}{\textup{3}} \;\; \tikzmarknode[rectangle,black]{}{\textup{1}}\tikzmarknode[rectangle,black]{}{\textup{2}}
\end{align*}
\end{example}

\

\subsubsection{Bialgebras on vincular permutations}

We have an analogous result to \Cref{thm:Quasi-shuffle_deconc_HA_part}.

\begin{theorem}[Bialgebras on vincular permutations]
\label{thm:Quasi-shuffle_deconc_HA_gen_perm}
The following holds:
\begin{itemize}
    \item[--] $(\FreeVincPer,\genconc, \coqsgen, u^{\prime},\varepsilon^{\prime})$ is a bialgebra,
    \item[--] $(\FreeVincPer,\qsgen,\deconcgen,u^{\prime},\varepsilon^{\prime})$ is a connected filtered Hopf algebra (whith grading defined in \Cref{eq:free_vector_space_gen_permutations}),
\end{itemize}
\end{theorem}
where $u^{\prime}$ denotes the unit map and $\varepsilon^{\prime}$ the counit map.
\begin{proof}
Let $(\stds,\sigma) \in \SetStandardizedIntervalPartitions_{m} \times \bigsigma_{m}$ and $(\stdt,\tau) \in \SetStandardizedIntervalPartitions_{m} \times \bigsigma_{m}$. Recall that
\begin{align*}
& \stds \conc \stdt := \{\stds_{1},...,\stds_{m},\stdt^{\prime}_{1},...,\stdt^{\prime}_{n}\}
\end{align*}
where for $i=1,...,n$, $\stdt^{\prime}_{i}:=|\punion\stds|+\stdt_{i}$. We also write $\tau_{m}:=(\tau_{1}+m)\cdots (\tau_{n}+m)$.
First, we have that
\begin{align*}
    &\coqsgen((\stds,\sigma) \genconc (\stdt,\tau))\\&=\sum_{A \cup A^{\prime} =[m+n]}
\sum_{\substack{\unstdI \in \SetIntervalPartitions(A),\; \unstdI^{\prime} \in \SetIntervalPartitions(A^{\prime})\\\\ \unstdI \gluepartitions \unstdI^{\prime}=\stds \conc \stdt }}(\std(\unstdI),\st((\sigma \concvargas \tau)|_{A}))\otimes (\std(\unstdI^{\prime}),\st((\sigma \concvargas \tau)|_{A^{\prime}})).
\end{align*}
is equal to
\begin{align*}
 &\sum_{\substack{A_{1} \cup A_{1}^{\prime} = \{1,...,m\}\\\\ A_{2} \cup A_{2}^{\prime} = \{m+1,...,m+n\} }}\;\sum_{\substack{\unstdI_{1} \in\SetIntervalPartitions(A_{1}), \unstdI_{1}^{\prime} \in\SetIntervalPartitions(A_{1}^{\prime})\\\\\unstdI_{2} \in\SetIntervalPartitions(A_{2}), \unstdI_{2}^{\prime} \in\SetIntervalPartitions(A_{2}^{\prime})\\\\ \unstdI_{1}\gluepartitions  \unstdI_{1}^{\prime}= \stds,\;\unstdI_{2}\gluepartitions  \unstdI_{2}^{\prime}= \stdt^{\prime}}}(\std(\unstdI_{1}),\st(\sigma\evaluatedAt{A_{1}})) \genconc (\std(\unstdI_{2}),\st(\tau_{m}\evaluatedAt{A_{2}}))\\\\&\otimes (\std(\unstdI_{1}^{\prime}),\st(\sigma\evaluatedAt{A_{1}^{\prime}})) \genconc (\std(\unstdI_{2}^{\prime}),\st(\tau_{m}\evaluatedAt{A_{2}^{\prime}}))
\end{align*}
 which, using arguments similar to the proof of \Cref{thm:Quasi-shuffle_deconc_HA_part}, can be shown to be equal to
\begin{align*}
&\coqsgen((\stds,\sigma))\; \genconc \; \coqsgen((\stdt,\tau)) \\&=\sum_{A \cup A^{\prime} =\punion\stds}
\sum_{\substack{\unstdI \in \SetIntervalPartitions(A),\; \unstdI^{\prime} \in \SetIntervalPartitions(A^{\prime})\\\\ \unstdI \gluepartitions \unstdI^{\prime}=\stds}}(\std(\unstdI),\st(\sigma |_{A}))\otimes (\std(\unstdI^{\prime}),\st(\sigma |_{A^{\prime}}))\\  &\genconc \sum_{B \cup B^{\prime} =\punion\stdt}
\sum_{\substack{\unstdJ \in \SetIntervalPartitions(B),\; \unstdJ^{\prime} \in \SetIntervalPartitions(B^{\prime})\\\\ \unstdJ \gluepartitions \unstdJ^{\prime}=\stdt}}(\std(\unstdJ),\st(\tau |_{B}))\otimes (\std(\unstdJ^{\prime}),\st(\tau |_{B^{\prime}})).
\end{align*}
\end{proof}

We can embed the Hopf algebra of the previous section.
\begin{proposition}[Embedding of partitions]
\label{prop:emb_part}
The map
\begin{align*}
    (\FreeIntPart,\qspart,\deconc) &\to (\FreeVincPer,\qsgen, \deconcgen)\\
      \psi: \stds &\mapsto \sum_{\sigma \in \bigsigma_{|\cup\mathfrak{s}|}}(\stds, \sigma).
\end{align*}
 is an injective Hopf homomorphism.
\end{proposition}

We will use the following simple observation.
\begin{lemma}
\label{lemma:partition_embedding}
Let $A, A^{\prime} \subset [n]$.
Then
\begin{align*}
    \bigsigma_{n}&=\biguplus_{\substack{\sigma \in \bigsigma_{|A|}\\\sigma' \in \bigsigma_{|A^{\prime}|}}}\{\gamma \in \bigsigma_{n}|\ \st(\gamma\evaluatedAt{A})=\sigma, \ \st(\gamma\evaluatedAt{A^{\prime}})=\sigma'\},
\end{align*}
where the union is over disjoint, but possibly empty, sets.
\end{lemma}
\begin{proof}
    If  $\sigma = \st(\gamma\evaluatedAt{A})=\hat\sigma$,
    $\sigma' = \st(\gamma\evaluatedAt{A'})=\hat\sigma'$,
    then $\sigma=\hat\sigma, \sigma'=\hat\sigma'$, which proves disjointness
    of the sets.
    Finally, taking for $\gamma \in \bigsigma_n$,
    $\sigma := \st(\gamma\evaluatedAt{A}),
    \sigma' := \st(\gamma\evaluatedAt{A'})$, proves
    the claimed identity.
\end{proof}

\begin{proof}[Proof of \Cref{prop:emb_part}]
Let $\stds\in \SetStandardizedIntervalPartitions_{k}$ and $\stdt\in \SetStandardizedIntervalPartitions_{k^{\prime}}$. We have
\begin{align*}
    \psi(\stds\ \qspart\ \stdt) = \sum_{n \in \N}\sum_{A \cup A^{\prime}=[n]}\sum_{\substack{\unstdI \in \SetIntervalPartitions(A),\unstdI \in \SetIntervalPartitions(A^{\prime})\\\std(\unstdI)=\stds, \std(\unstdI^{\prime})=\stdt}}\sum_{\gamma \in \bigsigma_{n}}(\unstdI \gluepartitions \unstdI^{\prime},\gamma)
\end{align*}
and
\begin{align*}
    \psi(\stds) \qsgen \psi(\stdt) &= \sum_{\substack{\sigma \in \bigsigma_{k}\\\tau \in \bigsigma_{k^{\prime}}}}\sum_{n \in \N}\sum_{\substack{A \cup A^{\prime}=[n]\\|A|=k,|A^{\prime}|=k^{\prime}}}\;\sum_{\substack{\unstdI \in \SetIntervalPartitions(A),\unstdI \in \SetIntervalPartitions(A^{\prime})\\ \std(\unstdI)=\stds, \std(\unstdI^{\prime})=\stdt}}\;\sum_{\substack{\gamma \in \bigsigma_{n}\\\st(\gamma\evaluatedAt{A})=\sigma,\st(\gamma\evaluatedAt{A^{\prime}})=\tau}}(\unstdI \gluepartitions \unstdI^{\prime},\gamma)\\&= \sum_{n \in \N}\sum_{\substack{A \cup A^{\prime}=[n]\\|A|=k,|A^{\prime}|=k^{\prime}}}\sum_{\substack{\unstdI \in \SetIntervalPartitions(A),\unstdI \in \SetIntervalPartitions(A^{\prime})\\ \std(\unstdI)=\stds, \std(\unstdI^{\prime})=\stdt}}\; \sum_{\gamma \in \bigsigma_{n}}\;\sum_{\substack{\sigma \in \bigsigma_{k}, \tau \in \bigsigma_{k^{\prime}}\\\st(\gamma\evaluatedAt{A})=\sigma,\st(\gamma\evaluatedAt{A^{\prime}})=\tau}}(\unstdI \gluepartitions \unstdI^{\prime},\gamma)\\&=^{\hspace{-0.7cm}\tiny\raisebox{1em}{\text{\Cref{lemma:partition_embedding}}}}\psi(\stds \qspart \stdt)
    \end{align*}

While for the coalgebra part
\begin{align*}
    &(\psi \otimes \psi) \circ(\deconc(\stds)) = (\psi \otimes \psi) \circ(\sum_{\mathfrak{a} \conc \mathfrak{b} = \stds}\stda \otimes \stdb)\\&=\sum_{\mathfrak{a} \conc \mathfrak{b} = \stds} \sum_{\alpha \in \bigsigma_{|\cup_{j}\mathfrak{a}_{j}|}}(\stda, \alpha) \otimes \sum_{\beta \in \bigsigma_{|\cup_{j}\mathfrak{b}_{j}|}}(\stdb, \beta)\\&=\sum_{\mathfrak{a} \conc \mathfrak{b} = \stds} \sum_{\alpha \in \bigsigma_{|\cup_{j}\mathfrak{a}_{j}|}} \sum_{\beta \in \bigsigma_{|\cup_{j}\mathfrak{b}_{j}|}}(\stda, \alpha) \otimes (\stdb, \beta)\\&=\sum_{\sigma \in \bigsigma_{|\cup_{j}\mathfrak{s}_{j}|}} \sum_{\mathfrak{a} \conc \mathfrak{b} = \stds} \sum_{\alpha \in \bigsigma_{|\cup_{j}\mathfrak{a}_{j}|}} \sum_{\beta \in \bigsigma_{|\cup_{j}\mathfrak{b}_{j}|}}\sum_{\alpha \; \concvargas \; \beta = \sigma}(\stda, \alpha) \otimes (\stdb, \beta).
\end{align*}
and
\begin{align*}
    \deconcgen \circ \psi(\stds) &=  \sum_{\sigma \in \bigsigma_{|\cup_{j}\mathfrak{s}_{j}|}}\deconcgen((\stds, \sigma))\\&=  \sum_{\sigma \in \bigsigma_{|\cup_{j}\mathfrak{s}_{j}|}} \sum_{\substack{\mathfrak{a} \conc \mathfrak{b} = \stds \\ \alpha \concvargas \beta = \sigma\\ \alpha \in \bigsigma_{|\cup_{j}\mathfrak{a}_{j}|}, \beta \in \bigsigma_{|\cup_{j}\mathfrak{b}_{j}|}}}(\stda, \alpha) \otimes (\stdb, \beta).
\end{align*}
\end{proof}

\begin{proposition}[Embedding of permutations]
\label{prop:emb_perm}
    The Hopf algebra  $(\FreePer,\superinfiltration,\Delta_{\scriptsize \concvargas})$ which appears in \cite{vargas2014hopf} can be embedded in $(\FreeVincPer,\qsgen,\deconcgen)$. Indeed the map
    \begin{align*}
    \phi: \sigma &\mapsto (\{\{i\}|1 \le i \le |\sigma|\}\},\sigma),
    \end{align*}
where $\{\{i\}|1 \le i \le |\sigma|\}\}$ is the singletons' partition, is an injective Hopf homomorphism.
\end{proposition}
For the proof, we need two lemmas.
\begin{lemma}
\label{lemma:emb_perm1}
Let $A,A^{\prime}\in \N_{\ge 1},|A|,|A^{\prime}| < \infty$ and let $\unstdI \in \SetIntervalPartitions(A),\unstdI^{\prime}\in \SetIntervalPartitions(A^{\prime})$ be the singletons' partitions of $A$ and $A^{\prime}$ respectively. Then
\begin{align*}
    \unstdI \gluepartitions \unstdI^{\prime} \in  \SetIntervalPartitions(A\cup A^{\prime})
\end{align*}
is the singletons' partition of $A \cup A^{\prime}$.
\end{lemma}
\begin{proof}
    It follows immediately from the definition of $\gluepartitions$.
\end{proof}
\begin{lemma}
\label{lemma:emb_perm2}
Let $n \in \N$. Then
\begin{align*}
&\Big|\biguplus_{A \cup A^{\prime} = [n]}\{\unstdI \in \SetIntervalPartitions(A),\unstdI^{\prime} \in \SetIntervalPartitions(A^{\prime})\mid\std(\unstdI)=\{\{i\}|1 \le i \le |A|\}, \std(\unstdI^{\prime})=\{\{i\}|1 \le i \le |A^{\prime}|\},\\\\
&\qquad\qquad\qquad\qquad\qquad\qquad\qquad\qquad\qquad \unstdI \gluepartitions \unstdI^{\prime} = \{\{1\},...,\{n\}\}\}\Big|\\\\
&= \Big|\{A,A^{\prime} \subset [n]| A \cup A^{\prime}=[n]\}\Big|.
\end{align*}
\end{lemma}
\begin{proof}
It follows immediately, since for $A \cup A^{\prime}= [n]$
\begin{align*}
    &|\{\unstdI \in \SetIntervalPartitions(A),\unstdI^{\prime} \in \SetIntervalPartitions(A^{\prime})|\std(\unstdI)=\{\{i\}|1 \le i \le |A|\}, \std(\unstdI^{\prime})=\{\{i\}|1 \le i \le |A^{\prime}|\},\\\\&\unstdI \gluepartitions \unstdI^{\prime} = \{\{1\},...,\{n\}\}\}|\\\\&=1.
\end{align*}
\end{proof}
\begin{proof}[Proof of \Cref{prop:emb_perm}]
Let $\sigma,\tau \in \bigsigma$. 
\begin{align*}
&\phi(\sigma) \qsgen \phi(\tau)\\&=\sum_{(\stdg, \gamma)}\sum_{A \cup A^{\prime} = \punion\stdg}\;\sum_{\substack{\unstdI \in \SetIntervalPartitions(A),\;\unstdI^{\prime} \in \SetIntervalPartitions(A^{\prime})\\\\ \std(\unstdI) = \{\{i\}|1 \le i \le |\sigma|\}\\\\\std(\unstdI^{\prime}) = \{\{i\}|1 \le i \le |\tau|\} \\\\ \unstdI \gluepartitions \unstdI^{\prime} = \stdg}}\sum_{\st(\gamma\evaluatedAt{A})=\sigma,\st(\gamma\evaluatedAt{A^{\prime}})=\tau}(\stdg,\gamma)\\&=^{\hspace{-0.7cm}\tiny\raisebox{1em}{\text{\Cref{lemma:emb_perm2}}}} \sum_{n \in \N}\sum_{\gamma\in \bigsigma_{n}}\sum_{A \cup A^{\prime} = [n]}\sum_{\substack{\unstdI \in \SetIntervalPartitions(A),\;\unstdI^{\prime} \in \SetIntervalPartitions(A^{\prime})\\\\ \std(\unstdI) = \{\{i\}|1 \le i \le |\sigma|\}\\\\\std(\unstdI^{\prime}) = \{\{i\}|1 \le i \le |\tau|\} \\\\ \unstdI \gluepartitions \unstdI^{\prime} = \{\{i\}|1 \le i \le n\}}}\sum_{\st(\gamma|_{A})=\sigma,\; \st(\gamma|_{A^{\prime}})=\tau}(\{\{i\}|1 \le i \le n\},\gamma)\\&= \sum_{n \in \N}\sum_{\gamma \in \bigsigma_{n}}\sum_{\substack{A \cup A^{\prime} = [n]\\\\ \st(\gamma|_{A})=\sigma,\; \st(\gamma|_{A^{\prime}})=\tau}}(\{\{i\}|1 \le i \le n\},\gamma)\\&= \phi(\sigma\; \superinfiltration \; \tau).
\end{align*}
While for the coalgebra part
\begin{align*}
&\deconcgen(\{\{i\}|1 \le i \le |\sigma|\},\sigma)) = \sum_{\substack{|\alpha|+|\beta| = |\sigma| \\ \alpha \square \beta = \sigma}}(\{\{i\}|1 \le i \le |\alpha|\},\alpha) \otimes (\{\{i\}|1 \le i \le |\beta|\},\beta)\\&=\phi \otimes \phi \circ \Delta_{\square}(\sigma).
\end{align*}

\end{proof}
\subsection{Signature for vincular permutation patterns}

In an analogous fashion to \Cref{def:signature_interval_partition}, define, for $(\stdL,\Lambda),(\stds,\sigma)\in \bigcup_{n \in \N} \SetStandardizedIntervalPartitions_{n} \times \bigsigma_{n}$
\begin{align*}
\Big\langle  \GPC((\stdL,\Lambda)), (\stds,\sigma) \Big\rangle &:=\#\{ A \subset [N]|\;\std(\stdL(A))\ge \stds, \st(\Lambda\evaluatedAt{A}) = \sigma\}.
\end{align*}
and extend linearly to $\FreeVincPer$. These patterns are known in the literature as \textit{vincular patterns}. They appear in \cite{babson2000generalized}.
\subsubsection{Character property and Chen's identity}

Again, we have analogous results to \Cref{thm:qsI_intpart} and \Cref{thm:ChenI_intpart}.
\begin{theorem}[Character property]
\label{thm: qsI_gen_per}
Let $(\stdL,\Lambda) \in \SetStandardizedIntervalPartitions_{N} \times \bigsigma_{N}$, with $N \in \N$. Then $\forall (\stds,\sigma),(\stdt,\tau)\in \bigcup_{n \in \N} \SetStandardizedIntervalPartitions_{n} \times \bigsigma_{n}$,
\begin{align*}
 \Big\langle  \GPC((\stdL,\Lambda)), (\stds,\sigma) \Big\rangle \Big\langle  \GPC((\stdL,\Lambda)), (\stdt,\tau) \Big\rangle
   &=    \Big\langle  \GPC((\stdL,\Lambda)), (\stds,\sigma) \qsgen (\stdt,\tau) \Big\rangle. 
\end{align*}
\end{theorem}
\begin{proof}
Using arguments analogous to the proof of \Cref{thm:qsI_intpart}, we have
\begin{align*}
    &|\{ A,B \subset [N]|\;\std(\stdL(A))\ge \stds, \std(\stdL(B))\ge \stdt, \st(\Lambda\evaluatedAt{A}) = \sigma, \st(\Lambda\evaluatedAt{B}) = \tau \}|\\
    &=| \biguplus_{\substack{(\stdg,\gamma,C) \in \bigcup_{n \in \N}\SetStandardizedIntervalPartitions_{n} \times \bigsigma_{n} \times 2^{N} \\\\ \std(\stdL(C))\ge \stdg\\\\ \st(\Lambda\evaluatedAt{C})=\gamma}}\{A,B \subset [N]\;|A\cup B = C, \std(\stdL(A)) \ge \stds, \std(\stdL(B)) \ge \stdt,\\\\& \std(\stdL(A)_{\stds}\gluepartitions \stdL(B)_{\stdt}) = \stdg, \st(\Lambda\evaluatedAt{A}) = \sigma, \st(\Lambda\evaluatedAt{B}) = \tau, \st(\Lambda \evaluatedAt{C}) = \gamma \}|\\\\ &= \sum_{\substack{(\stdg,\gamma,C) \in  \bigcup_{n \in \N}\SetStandardizedIntervalPartitions_{n} \times \bigsigma_{n} \times 2^{N} \\\\ \std(\stdL(C))\ge \stdg\\\\ \st(\Lambda\evaluatedAt{C})=\gamma}}\langle \qsgen((\stds,\sigma) \otimes (\stdt,\tau)) , (\stdg,\gamma) \rangle\\&= \sum_{(\stdg,\gamma) \in  \bigcup_{n \in \N}\SetStandardizedIntervalPartitions_{n} \times \bigsigma_{n}} \sum_{ \substack{ C   \in 2^{N} \\\\ \std(\stdL(C))\ge \stdg\\\\\st(\Lambda \evaluatedAt{C}) = \gamma }}\langle \qsgen((\stds,\sigma) \otimes (\stdt,\tau)) , (\stdg,\gamma) \rangle \\&= \sum_{(\stdg,\gamma) \in  \bigcup_{n \in \N}\SetStandardizedIntervalPartitions_{n} \times \bigsigma_{n}} \langle \qsgen((\stds,\sigma) \otimes (\stdt,\tau)) , (\stdg,\gamma) \rangle \sum_{ \substack{ C   \in 2^{N}\\\\ \std(\stdL(C))\ge \stdg\\\\\st(\Lambda \evaluatedAt{C}) = \gamma }} 1\\&= \sum_{(\stdg,\gamma) \in  \bigcup_{n \in \N}\SetStandardizedIntervalPartitions_{n} \times \bigsigma_{n}} \langle \qsgen((\stds,\sigma) \otimes (\stdt,\tau)) , (\stdg,\gamma) \rangle  \Big\langle  \GPC((\stdL,\Lambda)), (\stdg,\gamma) \Big\rangle\\&=  \Big\langle  \GPC((\stdL,\Lambda)), \sum_{(\stdg,\gamma) \in  \bigcup_{n \in \N}\SetStandardizedIntervalPartitions_{n} \times \bigsigma_{n}} \langle \qsgen((\stds,\sigma) \otimes (\stdt,\tau)) , (\stdg,\gamma) \rangle  (\stdg,\gamma) \Big\rangle\\&=  \Big\langle  \GPC((\stdL,\Lambda)), (\stds,\sigma) \qsgen (\stdt,\tau) \Big\rangle.
\end{align*}

\end{proof}

\begin{theorem}[Chen's identity]
\label{thm:Chen_gen_per}
Let $N_{1},N_{2}\in \N$ and let $(\stdL,\Lambda)\in \SetStandardizedIntervalPartitions_{N_{1}} \times \bigsigma_{N_{1}}$ and $(\stdM,M) \in \SetStandardizedIntervalPartitions_{N_{2}} \times \bigsigma_{N_{2}}$, then $\forall (\stds,\sigma) \in \bigcup_{n \in \N}\SetStandardizedIntervalPartitions_{n} \times \bigsigma_{n}$
\begin{align*}
 \Big\langle  \GPC\Big((\stdL,\Lambda) \genconc (\stdM,M)\Big), (\stds,\sigma) \Big\rangle 
   &=    \Big\langle  \GPC\Big((\stdL,\Lambda)\Big) \otimes  \GPC\Big((\stdM,M)\Big), \deconcgen((\stds,\sigma)) \Big\rangle 
\end{align*}
\end{theorem}
\begin{proof}
Recall that $\stdL \conc \stdM = \{\stdL,...,\stdL_{p},\stdM^{\prime}_{1},...,\stdM_{q}^{\prime}\}$, where $\stdM^{\prime}_{i} := \stdM_{i} + |\punion\stdL|$ , where $\stdM^{\prime}_{i} := \stdM_{i} + |\punion\stdL|$, and $\sigma \concvargas \tau = \sigma_{1} \cdots \sigma_{m} (\tau_{1}+m)\cdots (\tau_{n}+m)$. We have 
\begin{align*}
       &|\{ A \subset [N_{1}+N_{2}]|\;\std\left((\stdL \conc \stdM)(A)\right)\ge \stds, (\Lambda \concvargas M)\evaluatedAt{A} = \sigma\}|\\&=|\biguplus_{\substack{\mathfrak{a} \conc \mathfrak{b} = \stds\\\\ \alpha \square \beta = \sigma\\\\ \alpha \in \bigsigma_{|\punion\mathfrak{a}|},\,\beta \in \bigsigma_{|\punion\mathfrak{b}|} }}\{ A \subset [N_{1}+N_{2}]|\; \std\left((\stdL \conc \stdM)(A \cap \punion\stdL)\right)\ge \stda, \st((\Lambda \concvargas M) \evaluatedAt{A \cap \punion\stdL})=\alpha,\\\\&  \std\left((\stdL \conc \stdM)(A \cap \punion\stdM^{\prime})\right)\ge \stdb, \st((\Lambda \concvargas M) \evaluatedAt{A \cap \punion\stdM^{\prime}})=\beta\}|.
\end{align*}
using arguments similar to the proof of \Cref{thm:ChenI_intpart}: indeed there exists unique $\stda,\stdb \in \SetIntervalPartitions$ and $\alpha,\beta$ which \quotationmarks{split} $\stds$ and $\sigma$, respectively.
\end{proof}

\begin{example}
\begin{align*}
   &\Big \langle \GPC\Big(\;\tikzmarknode[rectangle,black]{}{\textup{1}}\;\;\tikzmarknode[rectangle,black]{}{\textup{3}}\tikzmarknode[rectangle,black]{}{\textup{2}} \;\; \tikzmarknode[rectangle,black]{}{\textup{5}}\tikzmarknode[rectangle,black]{}{\textup{4}}\;\Big),\;\; \tikzmarknode[rectangle,black]{}{\textup{2}}\;\;\tikzmarknode[rectangle,black]{}{\textup{1}}\; \Big \rangle \\&=\Big \langle \GPC\Big(\;\tikzmarknode[rectangle,black]{}{\textup{1}}\;\Big)\otimes \GPC\Big(\;\tikzmarknode[rectangle,black]{}{\textup{2}}\tikzmarknode[rectangle,black]{}{\textup{1}}\;\Big) \otimes \GPC\Big(\;\tikzmarknode[rectangle,black]{}{\textup{2}}\tikzmarknode[rectangle,black]{}{\textup{1}}\;\Big),\\&\;\; \e \otimes \e \otimes \tikzmarknode[rectangle,black]{}{\textup{2}}\;\;\tikzmarknode[rectangle,black]{}{\textup{1}}+ \e  \otimes \tikzmarknode[rectangle,black]{}{\textup{2}}\;\;\tikzmarknode[rectangle,black]{}{\textup{1}}\otimes \e+ \tikzmarknode[rectangle,black]{}{\textup{2}}\;\;\tikzmarknode[rectangle,black]{}{\textup{1}}\otimes \e \otimes \e\; \Big \rangle \\&=\Big \langle \GPC\Big(\;\tikzmarknode[rectangle,black]{}{\textup{1}}\;\Big),\;\; \tikzmarknode[rectangle,black]{}{\textup{2}}\;\;\tikzmarknode[rectangle,black]{}{\textup{1}}\; \Big \rangle \\&+\Big \langle \GPC\Big(\;\tikzmarknode[rectangle,black]{}{\textup{2}}\tikzmarknode[rectangle,black]{}{\textup{1}}\Big),\;\; \tikzmarknode[rectangle,black]{}{\textup{2}}\;\;\tikzmarknode[rectangle,black]{}{\textup{1}}\; \Big \rangle \\&+\Big \langle \GPC\Big(\;\tikzmarknode[rectangle,black]{}{\textup{2}}\tikzmarknode[rectangle,black]{}{\textup{1}}\Big),\;\; \tikzmarknode[rectangle,black]{}{\textup{2}}\;\;\tikzmarknode[rectangle,black]{}{\textup{1}}\; \Big \rangle = 2.
\end{align*}

\begin{align*}
   &\Big \langle \GPC\Big(\;\tikzmarknode[rectangle,black]{}{\textup{4}}\;\;\tikzmarknode[rectangle,black]{}{\textup{5}}\tikzmarknode[rectangle,black]{}{\textup{3}} \;\; \tikzmarknode[rectangle,black]{}{\textup{2}}\tikzmarknode[rectangle,black]{}{\textup{1}}\;\Big),\;\; \tikzmarknode[rectangle,black]{}{\textup{2}}\;\;\tikzmarknode[rectangle,black]{}{\textup{1}}\; \Big \rangle  = 9.
\end{align*}
\end{example}

\begin{remark}[Arbitrary delays]
  Since we count patterns, where two values can be either consecutive or have an arbitrary gap, a natural question might be whether it is possible to count patterns so that there is a specific time gap between values. This is indeed possible.

  Consider, for example,
  the task of counting the pattern $21$ with \emph{exactly} 
  one gap between the two timepoints of the occurence.
  pictorially, we can represent such a pattern as
  \begin{align}
    \label{eq:gap}
    \tikzmarknode[rectangle,black]{}{\textup{2}}\tikzmarknode[rectangle,black]{}{\textcolor{white}{2}}\tikzmarknode[rectangle,black]{}{\textup{1}}.
  \end{align}
  As the 'middle slot' is arbitrary, we equivalently need to count occurrences of
  patterns
  $(\stds,\sigma) \in \{\{1,2,3\}\} \times \bigsigma_{3}$, such that $\st(\sigma\evaluatedAt{\{1,3\}})=21$,
  i.e. the count of \eqref{eq:gap} in $(\stdL,\Lambda)$ is given by
  \begin{align*}
     \Big\langle \GPC((\stdL,\Lambda)), \tikzmarknode[rectangle,black]{}{\textup{2}}\tikzmarknode[rectangle,black]{}{\textup{3}}\tikzmarknode[rectangle,black]{}{\textup{1}}+ \tikzmarknode[rectangle,black]{}{\textup{3}}\tikzmarknode[rectangle,black]{}{\textup{1}}\tikzmarknode[rectangle,black]{}{\textup{2}}+ \tikzmarknode[rectangle,black]{}{\textup{3}}\tikzmarknode[rectangle,black]{}{\textup{2}}\tikzmarknode[rectangle,black]{}{\textup{1}} \Big\rangle.
  \end{align*}

  In \cite{zunino2010permutation}[Eq. 4], they consider patterns where delays, \quotationmarks{$D$}, between values are fixed. For $D=2$, and the permutation pattern $213$ one searches for
\begin{align*}
 \tikzmarknode[rectangle,black]{}{\textup{2}}\tikzmarknode[rectangle,black]{}{\textcolor{white}{2}}\tikzmarknode[rectangle,black]{}{\textcolor{white}{2}}\tikzmarknode[rectangle,black]{}{\textup{1}}\tikzmarknode[rectangle,black]{}{\textcolor{white}{2}}\tikzmarknode[rectangle,black]{}{\textcolor{white}{2}}\tikzmarknode[rectangle,black]{}{\textup{3}}.
\end{align*}
\end{remark}

\section{Summary and outlook}

In this work, we introduced two new Hopf algebras.
The Hopf algebra on interval partitions in \Cref{sec:Finite_interval_partitions} encodes the \textit{vincular} part of a permutation pattern.\\
This Hopf algebra can be equivalently seen as a Hopf algebra on words of positive integers where the filtered product $\qspart$ possesses a shuffle part and the coproduct is deconcatenation, see \Cref{rem:words_intergers}. 
We define a family of linear functionals parametrized by interval partitions and show that they are characters and satisfy a Chen's type identity.\\ The underlying combinatorics here is quite simple since the number of \quotationmarks{occurrences} of one word into another is a closed-form expression that depends on the letters of the two words, see \Cref{rem:explicit}.\\
 Finally, in \Cref{sec:Vincular_permutation_patterns}, we introduce the Hopf algebra on vincular patterns which is built upon our Hopf algebra on interval partitions and the superinfiltration Hopf algebra introduced in \cite{vargas2014hopf}. We also extend the definition of the functionals from \Cref{sec:Finite_interval_partitions} to store the number of occurrences of vincular patterns. These maps are shown again to behave like \textit{signatures}, satisfying identities which are reminiscent of the shuffle and the Chen's identities for paths.

\subsection{Open question}
We are interested in the following open questions.
\begin{itemize}
\item[--]  Is there a recursive definition for the product $\qsgen$?

\item[--] Both Hopf algebras, $(\FreeIntPart,\qspart,\deconc)$ and $(\FreeVincPer,\qsgen,\deconcgen)$ are free commutative as algebras, see \cite{cartier2021classical}[Theorem 4.4.1]. Are there \quotationmarks{interesting} sets of free commutative generators? An analogy we have in mind: the shuffle Hopf algebra is connected and graded and therefore automatically isomorphic, as an algebra, to a polynomial algebra. Though, one can also explicitly show that the shuffle algebra is a polynomial algebra over the set of Lyndon words.

\item[--] Is the algebra on interval partitions  $(\FreeIntPart,\qspart)$, or equivalently $(\bigoplus_{n\in \N} \Q[\N_{\ge 1}]^{\otimes n},\qswrd)$, isomorphic to the shuffle algebra $(\bigoplus_{n\in \N} \Q[\N_{\ge 1}]^{\otimes n},\shuffle)$?

\item[--] Since the Hopf algebras in this work are connected and filtered, one can compute the antipode using the well-known Takeuchi's formula, see \cite{takeuchi1971free}[lemma 14], but cancellations can occur. Indeed, this happens in our case. If we compute the antipode in $\adjacentSquares{2}\;\adjacentSquares{2}$ on $\FreeIntPart$ we obtain
\begin{align*}
    &\sum_{k \ge 0}(-1)^{k}\qspart^{k-1} \circ (\id - \eta \circ \varepsilon)\circ \deconc^{k-1}(\adjacentSquares{2}\;\adjacentSquares{2})\\=&\quad0-\adjacentSquares{2}\;\adjacentSquares{2}+2\;\adjacentSquares{2}\;\adjacentSquares{2}+2\;\adjacentSquares{3}+2\; \adjacentSquares{2}
    \\=&\quad\adjacentSquares{2}\;\adjacentSquares{2}+2\;\adjacentSquares{3}+2\; \adjacentSquares{2}
\end{align*}
Are there cancellation-free formulas for the antipodes of our Hopf algebras?
In \cite{penaguiao2022antipode}, the authors provide a cancellation-free formula for the antipode of $(\FreePer, \superinfiltration,\Delta_{\square})$.

\item[--] The Hopf algebra $(\FreeVincPer,\qsgen,\deconcgen)$ \quotationmarks{originates} from $(\FreeIntPart,\qspart,\deconc)$ and $(\FreePer,\superinfiltration,\Delta_{\concvargas})$. Is this an instance of a more general construction?
\end{itemize}

\section*{Acknowledgements}
The authors would like to thank Anders Claesson for pointing out, at
the occasion of a talk by one of the authors at the ACPMS \url{https://www.math.ntnu.no/acpms/},
that the patterns used here are known in the literature as \quotationmarks{vincular permutation patterns}.

\appendix

\section{Gluing partitions is associative}
\label{app:sec:gluing_associative}

The aim of this section is to show that the binary operation from \Cref{def:gluing_partitions} is associative.
We found it easier to formulate the operation
more abstractly first.
We will obtain the desired statement as
a special case, \Cref{lemma: associativity gluing}.

\subsection{Assumptions}
\label{subsec:assumptions}
Let $\Omega$ be a non-empty set
and consider maps
\begin{align*}
   & Q:\Omega \times \Omega \to \{0,1\}\\
   & m: 2^{\Omega} \setminus \{\emptyset\} \to \Omega.
\end{align*}
We assume
\begin{enumerate}
\item[P1.] $\forall x \in \Omega: Q(x,x) = 1$,
\item[P2.] $\forall x,y \in \Omega: Q(x,y)=1 \implies Q(y,x)=1$,
\item[P3.] $\forall A \in 2^{\Omega} \setminus \{\emptyset\}: \forall y \in \Omega:  Q(m(A),y)=1 \iff \exists x \in A: Q(x,y)=1$,
\item[P4.] $\forall A,B \in 2^{\Omega} \setminus \{\emptyset\}:\: m(A \cup B) = m(\{m(A)\}\cup B)$.
\end{enumerate}
\begin{remark}
\label{lemma:prop_m}
Let $A_{1},...,A_{n},B \in 2^{\Omega} \setminus \{\emptyset\}$. As a consequence of \textup{4.}, one has 
\begin{align*}
m(\bigcup_{i=1}^{n}A_{i} \cup B) = m(\bigcup_{i=1}^{n}\{m(A_{i})\} \cup B) 
 \end{align*}
\end{remark}

\subsection{Compression}

\newcommand\compr{\mathsf{compr}}

\begin{definition}[Compression of $A$]
Let $\compr:2^{\Omega} \setminus \{\emptyset\} \to 2^{\Omega} \setminus \{\emptyset\}$ where
\begin{align*}
   \compr(A) = \{m([x]_{\sim_{A}})|x \in A\}
\end{align*}
and $\sim_{A}$ is the smallest equivalence relation on $A$ which contains the relation $R \subset A \times A$ defined as:
\begin{align*}
    (x,y) \in R \iff Q(x,y) = 1.
\end{align*}
\end{definition}

Before stating \Cref{proposition:compression1}, we need three lemmas.
\begin{lemma}
\label{lemma:compression1}
Let $A \in 2^{\Omega} \setminus \{\emptyset\}$. Then 
\begin{enumerate}
\item[i.] $\forall x,y \in A: Q(x,y)=1 \implies [x]_{\sim_{A}} = [y]_{\sim_{A}}$
  \item[ii.] $\forall x,y \in A: Q(m([x]_{\sim_{A}}),m([y]_{\sim_{A}}))=1 \iff [x]_{\sim_{A}} = [y]_{\sim_{A}}$
  \item[iii.]  $\forall x,y \in A: Q(x,m([y]_{\sim_{A}}))=1 \iff [x]_{\sim_{A}} = [y]_{\sim_{A}}$
    \item[iv.]  $\forall x,y \in A: m([x]_{\sim_{A}}) = m([y]_{\sim_{A}}) \iff [x]_{\sim_{A}} = [y]_{\sim_{A}}$
    \item[v.]    $\forall x,y \in A: m([x]_{\sim_{A}}) \sim_{\compr(A)} m([y]_{\sim_{A}}) \iff m([x]_{\sim_{A}}) = m([y]_{\sim_{A}})$
\end{enumerate}
\end{lemma}
\begin{proof}~

\textit{i}. follows from the definition of $\sim_{A}$.

\textit{ii}. We have
\begin{align*}
     [x]_{\sim_{A}} = [y]_{\sim_{A}} \implies Q(m([x]_{\sim_{A}}),m([y]_{\sim_{A}})) = Q(m([x]_{\sim_{A}}),m([x]_{\sim_{A}})) =1
\end{align*}
by P1.
On the other hand
\begin{align*}
    Q(m([x]_{\sim_{A}}),m([y]_{\sim_{A}})) = 1
\end{align*}
implies, P3, $\exists w \in [x]_{\sim_{A}}$ such that $ Q(w,m([y]_{\sim_{A}})) = 1$ which implies, P2 and P3, $\exists z \in [y]_{\sim_{A}}$ such that $Q(w,z) = 1$. Therefore $x \sim_{A} w \sim_{A} z \sim_{A} y$, i.e. $[x]_{\sim_{A}} = [y]_{\sim_{A}}$.\\

\textit{iii}.

If  $Q(x,m([y]_{\sim_{A}}))=1$
then, P3, there is $w \in [y]_{\sim_{A}}$ such that $Q(x,w)=1$. Then $x\sim_{A}w \sim_{A}y$, i.e. $[x]_{\sim_{A}}=[y]_{\sim_{A}}$.

If $[x]_{\sim_{A}} = [y]_{\sim_{A}}$, then $Q(x,m([y]_{\sim_{A}}))=Q(x,m([x]_{\sim_{A}}))=1$,
by P3, since $x \in [x]_{\sim_A}$.

\textit{iv}.
If $m([x]_{\sim_{A}}) = m([y]_{\sim_{A}})$, then $Q(m([x]_{\sim_{A}}),m([y]_{\sim_{A}}))=1$ and therefore, by ii., $[x]_{\sim_{A}}=[y]_{\sim_{A}}$.
The other implication is immediate.

\textit{v}. Let $m([x]_{\sim_{A}}) \sim_{\compr(A)} m([y]_{\sim_{A}})$, then $\exists w_{1},...,w_{n} \in A$ such that 
\begin{align*}
    \forall i \in \{1,...,n-1\}: Q(m([w_{i}]_{\sim_{A}}),m([w_{i+1}]_{\sim_{A}}))=1
\end{align*}
and $Q(m([x]_{\sim_{A}}),w_{1})=1$ and $Q(w_{n},m([y]_{\sim_{A}}))=1$. But then we have
\begin{align*}
    &m([x]_{\sim_{A}}) =  m([w_{1}]_{\sim_{A}}) = \cdots = m([w_{n}]_{\sim_{A}}) = m([y]_{\sim_{A}}).
\end{align*}

\end{proof}

\begin{lemma}
\label{lemma:compression2}
Let $A \in 2^{\Omega} \setminus \{\emptyset\}$. Then 
\begin{enumerate}
   \item[i.] $\forall x,y \in A:x \sim_{A} y \iff x \sim_{\compr(A)\cup A} y$
  \item[ii.] $\forall x,y \in A:x \sim_{A} y \iff m([x]_{\sim_{A}}) \sim_{\compr(A)\cup A} y$
\item[iii.] $\forall x,y \in A:x \sim_{A} y \iff x \sim_{\compr(A)\cup A} m([y]_{\sim_{A}})$
   \item[iv.] $\forall x,y \in A:x \sim_{A} y \iff m([x]_{\sim_{A}}) \sim_{\compr(A)\cup A} m([y]_{\sim_{A}})$
    \item[v.] $\forall x,y \in A:x \sim_{A} y \iff m([x]_{\sim_{A}}) = m([y]_{\sim_{A}}) $
 \end{enumerate}
\end{lemma}
\begin{proof}
These are consequences of properties P1, P2 and P3
and of the sets being equivalence classes and of the previous lemma.
\end{proof}

\begin{lemma}
\label{lemma:compression3}
For $A \in 2^{\Omega} \setminus \{\emptyset\}$, $x,y \in A$,
\begin{align*}
    m([x]_{\sim_{A}})\in [y]_{\sim_{A}}
    \Rightarrow
    [x]_{\sim_{A}} = [y]_{\sim_{A}}
\end{align*}
\end{lemma}
\begin{proof}
    Since $Q(m([x]_{\sim_{A}}),m([y]_{\sim_{A}}))=1$, then 
    by \Cref{lemma:compression1} ii., $[x]_{\sim_{A}} = [y]_{\sim_{A}}$.
\end{proof}

The map $\compr$ satisfies a form of idempotence.
\begin{proposition}
\label{proposition:compression1}
For $A \in 2^{\Omega} \setminus \{\emptyset\}$,
\begin{align*}
   \compr(A) =  \compr( \compr(A) \cup A).
\end{align*}
\end{proposition}
\begin{proof}
We define the map
\begin{align*}
    g: \{[x]_{\sim_{A}}|x \in A\} &\to \{[z]_{\sim_{\compr(A)\cup A}}|z \in \compr(A)\cup A\}\\ [x]_{\sim_{A}}& \mapsto [x]_{\sim_{A}} \cup \{m([x]_{\sim_{A}})\}
\end{align*}
and show that it is surjective (one can show that it is also injective, with \Cref{lemma:compression3} but we do not need this here). We need to show that
\begin{align*}
  \{[x]_{\sim_{A}} \cup \{m([x]_{\sim_{A}}\})|x \in A\}  =  \{[z]_{\sim_{\compr(A)\cup A}}|z \in \compr(A)\cup A\}
\end{align*}
Let $x \in A$, then
\begin{align*}
    [x]_{\sim_{A}} \cup \{m([x]_{\sim_{A}})\} \subset  [x]_{\sim_{\compr(A) \cup A}}.
\end{align*}
Indeed, this follows from \Cref{lemma:compression2}
\begin{align*}
y \sim_{A} x \implies y \sim_{\compr(A) \cup A} x
\end{align*}
and, since $Q(m([x]_{\sim_{A}}),x)=1$, from \Cref{lemma:compression1}
\begin{align*}
 m([x]_{\sim_{A}}) \sim_{\compr(A) \cup A} x.
\end{align*}
For the other direction, let $w \in  [x]_{\sim_{\compr(A) \cup A}}$ and $ w \in A$. Then, from \Cref{lemma:compression2}, we know that $w \sim_{\compr(A) \cup A} x$ implies $w \sim_{A} x$, and $w \in [x]_{\sim_{A}}$. If $w \in  [x]_{\sim_{\compr(A) \cup A}}$ and $ w \in \compr(A)$, we can write $w = m([u]_{\sim_{A}})$ for some $u \in A$. Then, from the previous \Cref{lemma:compression2}, we know that $m([u]_{\sim_{A}}) \sim_{\compr(A) \cup A} x$, implies $[u]_{\sim_{A}} = [x]_{\sim_{A}}$, and $w = m([x]_{\sim_{A}})$. Therefore
\begin{align*}
 [x]_{\sim_{A}} \cup \{m([x]_{\sim_{A}})\} = [x]_{\sim_{\compr(A) \cup A}} \in \{[z]_{\sim_{\compr(A)\cup A}}|z \in \compr(A)\cup A\}.
\end{align*}
Now let $z \in A$, we have
\begin{align*}
    [z]_{\sim_{\compr(A) \cup A}} = [z]_{\sim_{A}} \cup \{m([z]_{\sim_{A}})\} 
\end{align*}
and in case $z \in \compr(A)$, which means $z = m([u]_{\sim_{A}})$ for some $u \in A$,
\begin{align*}
   [z]_{\sim_{\compr(A) \cup A}} = [u]_{\sim_{A}} \cup \{m([u]_{\sim_{A}})\} 
\end{align*}
using analogous arguments used to show the previous inclusion. This shows
\begin{align*}
   g( \{[x]_{\sim_{A}}|x \in A\} ) =  \{[x]_{\sim_{A}} \cup \{m([x]_{\sim_{A}})\}|x \in A\} = \{[z]_{\sim_{\compr(A)\cup A}}|z \in \compr(A)\cup A\}
\end{align*}
If we use property 4 of $m$, we get the desired result
\begin{align*}
  \{m([x]_{\sim_{A}}) |x \in  A\} &= \{m([x]_{\sim_{A}} \cup [x]_{\sim_{A}})|x \in  A\}\\& = \{m\left([x]_{\sim_{A}} \cup \{m([x]_{\sim_{A}})\}\right)|x \in A\} \\&=  \{m([z]_{\sim_{\compr(A)\cup A}})|z \in \compr(A)\cup A\}.
\end{align*}
\end{proof}

We now state a lemma used in the proof of the upcoming \Cref{proposition:compression2}.
\begin{lemma}
\label{lemma:compression4}
Let $A,B \in 2^{\Omega} \setminus \{\emptyset\}$. Then
\begin{align*}
&\forall x,y \in A: x \sim_{A \cup B} y \iff m([x]_{\sim_{A}}) \sim_{\compr(A) \cup B} m([y]_{\sim_{A}})\\
&\forall x \in A: \forall y \in B: x \sim_{A \cup B} y \iff m([x]_{\sim_{A}}) \sim_{\compr(A) \cup B} y\\
&\forall x,y \in B: x \sim_{A \cup B} y \iff x\sim_{\compr(A) \cup B} y
\end{align*}
\end{lemma}
\begin{proof}
These statements are similar to the ones of \Cref{lemma:compression1} and \Cref{lemma:compression2}, and the proof is also similar. 
\end{proof}
We now show the main result.

\begin{proposition}
\label{proposition:compression2}
Let $A,B \in 2^{\Omega} \setminus \{\emptyset\}$. Then 
\begin{align*}
   \compr(A \cup B) =  \compr( \compr(A) \cup B)
\end{align*}
\end{proposition}
\begin{proof}
Let $x \in A$. We can write
\begin{align*}
&[m([x]_{\sim_{A}})]_{\sim_{\compr(A) \cup B}}\\&= \{u \in \compr(A)| u \sim_{\compr(A) \cup B} m([x]_{\sim_{A}})\} \cup \{b \in B| b \sim_{\compr(A) \cup B} m([x]_{\sim_{A}})\}. 
\end{align*}
We now show that
\begin{align*}
\{u \in \compr(A)| u \sim_{\compr(A) \cup B}\} = \{m([a]_{\sim_{A}})|a \in A, a \sim_{A \cup B} x\}.
\end{align*}
Thanks to \Cref{lemma:compression4}, $a \in A$ such that $a \sim_{A \cup B} x$, implies $m([a]_{\sim_{A}}) \sim_{\compr(A) \cup B} m([x]_{\sim_{A}})$. Therefore we have
\begin{align*}
\{m([a]_{\sim_{A}})|a \in A, a \sim_{A \cup B} x\} &\subset \{u \in \compr(A)| u \sim_{\compr(A) \cup B} m([x]_{\sim_{A}})\}.
\end{align*}
Now let $u \in \{u \in \compr(A)| u \sim_{\compr(A) \cup B} m([x]_{\sim_{A}})\}$, which means $u = m([z]_{\sim_{A}})$ for some $z \in A$. Then we have
\begin{align*}
    m([z]_{\sim_{A}}) \sim_{\compr(A)\cup B} m([x]_{\sim_{A}}) \implies z \sim_{A \cup B} x,
\end{align*}
which means
\begin{align*}
   m([z]_{\sim_{A}})\in \{m([a]_{\sim_{A}})|a \in A, a \sim_{A \cup B} x\}.
\end{align*}
Therefore
\begin{align*}
\{m([a]_{\sim_{A}})|a \in A, a \sim_{A \cup B} x\} &\supset \{u \in \compr(A)| u \sim_{\compr(A) \cup B} m([x]_{\sim_{A}})\}.
\end{align*}
Similarly, we can also show that
\begin{align*}
    \{b \in B| b \sim_{\compr(A) \cup B}m([x]_{\sim_{A}}) \} &=\{b \in B| b \sim_{A \cup B} x\}.
\end{align*}
We can now finally write
\begin{align}
\label{eq:comprcompr}
m([m([x]_{\sim_{A}})]_{\sim_{\compr(A) \cup B}}) &= m\left(\{m([a]_{\sim_{A}})|a \in A, a \sim_{A \cup B} x\}\cup \{b \in B| b \sim_{A \cup B}\} \right)
\end{align}
We also obviously have
\begin{align}
\label{eq:compr}
m([x]_{\sim_{A \cup B}}) &= m\left(\{a \in A| a \sim_{A \cup B} x\} \cup \{b \in B| b \sim_{A \cup B} x\} \right)  
\end{align}
We now show that \Cref{eq:comprcompr} and \Cref{eq:compr} are equal. Notice that it is quite straightforward to verify that
\begin{align*}
    \{a \in A| a \sim_{A \cup B} x\} = \bigcup_{\substack{a \in A\\a \sim_{A \cup B} x}} [a]_{\sim_{A}}
\end{align*}
and using property 4 of $m$ yields:
\begin{align*}
    m\left(\bigcup_{\substack{a \in A\\a \sim_{A \cup B} x}} [a]_{\sim_{A}}\right) = m\left(\bigcup_{\substack{a \in A\\a \sim_{A \cup B} x}} \{m([a]_{\sim_{A}})\}\right) = m\left(\{m([a]_{\sim_{A}})|a \in A, a \sim_{A \cup B} x\}\right)
\end{align*}
and therefore 
\begin{align}
\label{eq:first_eq}
m([x]_{\sim_{A \cup B}}) = m([m([x]_{\sim_{A}})]_{\sim_{\compr(A) \cup B}}).
\end{align}

Now let $x \in B$. We can write
\begin{align*}
    [x]_{\sim_{A \cup B}} = \{a \in A| a \sim_{A \cup B} x\} \cup \{b \in B| b \sim_{A \cup B} x\} 
\end{align*}
and
\begin{align*}
    [x]_{\sim_{\compr(A) \cup B}} = \{y \in \compr(A)| y \sim_{\compr(A) \cup B} x\} \cup \{b \in B| b \sim_{\compr(A) \cup B} x\} 
\end{align*}
With steps similar to the ones used to show \Cref{eq:first_eq}, one obtains
\begin{align*}
 m([x]_{\sim_{A \cup B}}) = m([x]_{\sim_{\compr(A) \cup B}}).
\end{align*}
Since
\begin{align*}
\forall x \in A: m([x]_{\sim_{A \cup B}}) &= m([m([x]_{\sim_{A}})]_{\sim_{\compr(A) \cup B}}),\\
 \forall x \in B: m([x]_{\sim_{A \cup B}}) &= m([x]_{\sim_{\compr(A) \cup B}}),
\end{align*}
we have shown that any element in $\compr(A\cup B)$ can be written as an element of $\compr(\compr(A) \cup B)$ and vice-versa, which finishes the proof.
\end{proof}

\subsection{Associativity}

Define $M: 2^{\Omega} \setminus \{\emptyset\} \times  2^{\Omega} \setminus \{\emptyset\} \to  2^{\Omega} \setminus \{\emptyset\}$ as 
\begin{align*}
    M(A,B) := \compr(A\cup B).
\end{align*}

$(2^{\Omega} \setminus \{\emptyset\},M)$ is a semigroup as illustrated in the following proposition.

\begin{proposition}[Associativity of M]
\label{prop:associativityM}

\begin{align*}
\forall A,B,C \in 2^{\Omega} \setminus \{\emptyset\}:   M(A,M(B,C)) = M(M(A,B),C)
\end{align*}
\end{proposition}
\begin{proof}
As a consequence of \Cref{proposition:compression2} one has
\begin{align*}
    \compr(A \cup \compr(B\cup C)) =  \compr(A \cup B\cup C)  =  \compr(\compr(A\cup B)\cup C)
\end{align*}
\end{proof}

We are now ready to show that the operation $\gluepartitions$ (see \Cref{def:gluing_partitions}) is associative. 

\begin{lemma}
\label{lemma: associativity gluing}
Let $\unstdI,\unstdI^{\prime},\unstdI^{\prime\prime} \in \SetIntervalPartitions$. Then:
\begin{align*}
(\unstdI \gluepartitions \unstdI^{\prime} )\gluepartitions \unstdI^{\prime\prime} =  \unstdI \gluepartitions (\unstdI^{\prime}\gluepartitions \unstdI^{\prime\prime})
\end{align*}
\end{lemma}
\begin{proof}

$\Omega := 2^{\N_{\ge1}}\setminus \{\emptyset\}$,
   and \begin{align*}
   \forall x,y \in \Omega:\, Q(x,y)&:=\begin{cases}
			1, & \text{if}\; x \cap y \neq \emptyset\\
            0, & \text{else}
		 \end{cases}\\
  \forall A \in 2^{\Omega} \setminus \{\emptyset\}:\, m(A)&:=\bigcup_{x\in A}x
\end{align*}
satisfy assumptions 1,2,3 and 4 from \Cref{subsec:assumptions}. Indeed,
2. follows from the commutativity of $\cap$,
1. follows from the idempotence of $\cap$,
3. follows by definition of union,
and regarding 4., for $A,B \in 2^\Omega \setminus \{\emptyset\}$
\begin{align*}
m\left( \{m(A)\}\cup B\right)
            &= m(\{\bigcup_{x\in A}x\}\cup B)\\
            &= m(\{\bigcup_{x\in A}x\} \cup \{y|y\in  B\})\\
                  &= \bigcup_{x\in A}x \cup \bigcup_{y\in B}y\\
                  &= \bigcup_{z\in A\cup B}z \\&= m(A \cup B),
\end{align*}
as claimed.

We hence get from \Cref{prop:associativityM} that
\begin{align*}
   M: 2^\Omega \setminus \{\emptyset\} \times 2^\Omega \setminus \{\emptyset\} \to 2^\Omega \setminus \{\emptyset\} 
\end{align*}
is associative. Now from \Cref{lemma:gluing_welldefined}, it follows that, in particular, $\forall \unstdI,\unstdJ \in \SetIntervalPartitions \setminus \{\emptyset\}$
\begin{align*}
M(\unstdI,\unstdJ)=\unstdI \gluepartitions \unstdJ  \in \SetIntervalPartitions \setminus \{\emptyset\}.
\end{align*}
Therefore we can restrict and corestrict $M$
\begin{align*}
   M: \SetIntervalPartitions \setminus \{\emptyset\} \times \SetIntervalPartitions \setminus \{\emptyset\} \to \SetIntervalPartitions \setminus \{\emptyset\},
\end{align*}
which is automatically associative. From \Cref{def:gluing_partitions} it follows immediately that $\forall \unstdI \in \SetIntervalPartitions$
\begin{align*}
   \unstdI \gluepartitions \emptyset = \unstdI.
\end{align*}
Therefore 
$\forall \unstdI,\unstdI^{\prime},\unstdI^{\prime\prime} \in \SetIntervalPartitions$
\begin{align*}
(\unstdI \gluepartitions \unstdI^{\prime} )\gluepartitions \unstdI^{\prime\prime} =  \unstdI \gluepartitions (\unstdI^{\prime}\gluepartitions \unstdI^{\prime\prime}).
\end{align*}
\end{proof}

\begin{remark}
    Another example of $m$ and $Q$ satisfying the properties P1-P4 is as follows.
    Let $\Omega := \N_{\ge 1}$ and 
    set 
\begin{align*}
     \forall x,y \in \Omega,\; Q(x,y)&:=1\\
     \forall A \in 2^{\Omega}\setminus \{\emptyset\},\;  m(A) &:= \max(A).
\end{align*}
Properties P1-P3 are trivially satisfied since $Q$ is always 1 and P4 obviously holds since
\begin{align*}
    \max(A \cup B)&=
    \max(\{\max(A)\} \cup B).
  \end{align*}

 Now let $\Omega := \N^{k} \setminus \{(0,\dots,0)\}$ where $k > 1$. Set
  \begin{align*}
     Q((x_{1},\dots,x_{k}),(y_{1},\dots,y_{k}))=\begin{cases}
              1, & \text{if}\;\;\exists i: x_{i}\neq 0 \land y_{i}\neq 0\\
              0, & \text{else}.
           \end{cases}
  \end{align*}
  and 
   \begin{align*}
     m(A) &:= ({\max}^{(1)}(A),\dots,{\max}^{(k)}(A))
  \end{align*}
  where $\max^{(i)}(A)$ is the maximum element of the i-th component of all the elements in $A$. One can show that P1-P4 hold also for this example.
\end{remark}

\printbibliography

\end{document}

%% file: includes.tex
\usepackage[]{algorithm2e}
\usepackage{bbm} 
\usepackage{dsfont} 

\usepackage{commath} 

\usepackage{amsmath, amsthm, amssymb}
\usepackage{mathabx} 
\usepackage{mathrsfs}
\usepackage{url}
\usepackage{amssymb,amsfonts}
\usepackage{color}
\usepackage{shuffle}

\usepackage[linecolor=white,backgroundcolor=white,bordercolor=white,textsize=tiny]{todonotes}
\usepackage{breqn}

\usepackage{silence}
\usepackage[colorlinks=true,linkcolor=blue,citecolor=blue]{hyperref}
\usepackage{cleveref}
\usepackage{graphicx}
\usepackage[font={small,it}]{caption}
\usepackage[flushleft]{threeparttable}
\usepackage{longtable}
\usepackage{enumitem}

\newtheorem{theorem}{Theorem}[section]
\theoremstyle{plain}

\newtheorem{corollary}[theorem]{Corollary}

\newtheorem{example}[theorem]{Example}

\newtheorem{lemma}[theorem]{Lemma}

\newtheorem{proposition}[theorem]{Proposition}
\newtheorem{remark}[theorem]{Remark}

\theoremstyle{definition} 

\newtheorem{definition}[theorem]{Definition}
\newtheorem*{tata}{Generalization}
  {\begin{mdframed}[backgroundcolor=lightgray]\begin{tata}}%
  {\end{tata}\end{mdframed}}

%% file: commands.tex
\newcommand{\N}{\mathbb{N}}

\newcommand{\Q}{\mathbb{Q}}

\newcommand{\id}{\operatorname{id}}

\newcommand{\mute}[1]{}
\let\todon\todo
\renewcommand{\todo}[1]{\todon{\color{red}#1}}

\newcommand{\evaluatedAt}[1]{\,\raisebox{-.5em}{$\vert_{#1}$}}